\documentclass[11pt, twoside,leqno]{article}

\usepackage{amsmath}
\usepackage{amssymb}
\usepackage{titletoc}
\usepackage{mathrsfs}
\usepackage{amsthm}

\usepackage{indentfirst}
\usepackage{color}
\usepackage{txfonts}

\usepackage{graphicx}

\allowdisplaybreaks

\def\red{\color{red}}
\def\blue{\color{blue}}

\def\ls{\lesssim}

\def\fz{\infty}

\def\r{\right}
\def\lf{\left}

\def\supp{{\mathop\mathrm{\,supp\,}}}

\def\rr{{\mathbb R}}

\def\rn{{{\rr}^n}}
\def\zz{{\mathbb Z}}
\def\nn{{\mathbb N}}
\def\cc{{\mathbb C}}

\newcommand{\wz}{\widetilde}

\def\cl{{\mathcal L}}
\def\az{\alpha}
\def\lz{\lambda}

\def\dz {\delta}

\def\bz{\beta}

\def\gz{{\gamma}}

\def\tz{\theta}
\def\sz{\sigma}

\def\wz{\widetilde}

\def\ls{\lesssim}

\def\pat{\partial}

\def\hs{\hspace{0.3cm}}

\def\dsum{\displaystyle\sum}
\def\dint{\displaystyle\int}
\def\dfrac{\displaystyle\frac}
\def\dsup{\displaystyle\sup}

\newtheorem{theorem}{Theorem}[section]
\newtheorem{lemma}[theorem]{Lemma}
\newtheorem{corollary}[theorem]{Corollary}
\newtheorem{proposition}[theorem]{Proposition}
\theoremstyle{definition}
\newtheorem{remark}[theorem]{Remark}

\newtheorem{definition}[theorem]{Definition}

\numberwithin{equation}{section}
\def\supp{{\mathop\mathrm{\,supp\,}}}

\def\loc{{\mathop\mathrm{loc\,}}}

\numberwithin{equation}{section}

\begin{document}

\title{\bf\Large
Gaussian Estimates for Heat Kernels of Higher Order Schr\"odinger Operators with
Potentials in Generalized Schechter Classes
\footnotetext{\hspace{-0.35cm}
2020 {\it Mathematics Subject Classification}. Primary: 35K08;
Secondary: 35J10, 35J30, 47G40.
\endgraf {\it Key words and phrases}.
heat kernel, Schr\"odinger operator, higher order differential operator, Schechter class,
perturbation, Bessel potential.
\endgraf
This project is supported by the National Natural Science Foundation of China
(Grant Nos. 11771446, 11871254, 11671031, 11971058, 11761131002, 11671185 and 11971431),
the Natural Science Foundation of Zhejiang Province (Grant No. LY18A010006)
and the first Class Discipline of Zhejiang-A (Zhejiang Gongshang University-Statistics).}}
\author{Jun Cao, Yu Liu, Dachun Yang\footnote{Corresponding author,
E-mail: \texttt{dcyang@bnu.edu.cn}/{\red December 08, 2020}/Final version.}\ \ and Chao Zhang}
\date{}
\maketitle

\vspace{-0.8cm}

\begin{center}
\begin{minipage}{13cm}
{\small {\bf Abstract.}
Let $m\in\nn$,
$P(D):=\sum_{|\az|=2m}(-1)^m a_\az
D^\az$ be a $2m$-order homogeneous elliptic operator with real
constant coefficients on $\rn$, and $V$ a measurable function
on $\mathbb{R}^n$. In this article, the authors introduce a new
generalized Schechter class
concerning $V$ and show that the higher order Schr\"odinger operator
$\mathcal{L}:=P(D)+V$ possesses a heat kernel that satisfies the Gaussian upper bound and
the H\"older regularity when $V$ belongs to this new class.
The Davies--Gaffney estimates for the associated semigroup and
their local versions are also given. These results pave the way for many further studies
on the analysis of $\mathcal{L}$.}
\end{minipage}
\end{center}

\tableofcontents

\section{Introduction\label{s1}}
The analysis of the Schr\"odinger operator is an important
topic in various fields of mathematics and physics
(see, for instance, \cite{Kato95,ReedSimon75-II,Simon84,Sim00}).
Many aspects of this analysis focus on the estimates of the corresponding
heat kernel, because the latter encodes plenty of information related to the
operator, such as the structure of the parabolic equation and the geometry
of the underlying space (see, for instance, \cite{BGK12,Dav90,Gri09,GHL14,LiYau86}) and hence
has wide applications; see, for instance, \cite{Dav90,Ku99,Ouha05}
for the spectral properties of differential operators,
\cite{CaChYaYa14,DeDiYa18,Dev19,JiLi20,LiDo10,Soya10}
for the boundedness of some singular integral operators
such as Riesz transforms, and
 \cite{DuYaZh14,DzZi99,DzZi04,HLMMY11,HuLiLi18,JiaYaZh09,WuYa16,YaZh11}
for the function spaces associated with the
Schr\"odinger operator.

In what follows, we use $L_\loc^1(\rn)$ to denote the set of all locally
integrable functions on $\rn$.
To estimate the heat kernel of the Schr\"odinger operator, various
conditions on the potential are proposed. Let $-\Delta+V$ be
the time-independent Schr\"odinger operator on the Euclidean space
$\rn$ with $\Delta:=\sum_{j=1}^n\frac{\pat^2}{\pat x_j^2}$ being the Laplace operator
and the potential $V$ a measurable function on $\rn$. If $V\in L_\loc^1(\rn)$ is
nonnegative, then, by the Feynman--Kac formula
(see, for instance, \cite{Simon84}), it is well known that $-\Delta+V$ possesses a heat
kernel $p_t$ on $(0,\fz)\times \rn\times\rn$ that satisfies the
Gaussian upper bound: there exist some positive constants $C$ and $c_1$ such that,
for any $t\in (0,\fz)$ and $x$, $y\in\rn$,
\begin{align}\label{eqn-2GUB}
0\le p_t(x,y)\le  \frac{C}{t^{n/2}}\exp\lf\{-c_1\frac{|x-y|^2}{t}\r\}.
\end{align}
Note that the condition $0\le V\in L_\loc^1(\rn)$
is satisfied when $V(x):=|x|^a$ for any $x\in \rn\setminus \{\vec 0_n\}$ with
$a\in (-n,\fz)$. Here and thereafter, $\vec 0_n$ denotes the origin of $\rn$.
If $V$ has a negative
part, then the Gaussian upper bound  \eqref{eqn-2GUB} still holds true as long
as $V\in L^p(\rn)$ with $p\in (n/2,\fz)$ (see \cite{Zhang03}). It is known that the Lebesgue
space $L^p(\rn)$, $p\in (n/2,\fz)$, is contained in a larger Kato class
$K_2(\rn)$ (see \cite{Simon84} and \eqref{eqn-katoc} below for the
definition of $K_2(\rn)$). In the same article, Simon showed that
the heat kernel of $-\Delta+V$ satisfies the local Gaussian upper bound:
there exist positive constants $C$, $c_1$, and  $w$ such that, for
any $t\in (0,\fz)$ and $x$, $y\in\rn$,
\begin{align}\label{eqn-2lGUB}
0\le p_t(x,y)\le \frac{C}{t^{n/2}}\exp\lf\{wt-c_1\frac{|x-y|^2}{t}\r\},
\end{align}
if $V$ is in $K_2(\rn)$ (see also \cite{LiSe98}).
The positive constant $w$ in \eqref{eqn-2lGUB} can be zero if $V$
satisfies some additional conditions (see \cite{BoDzSz19,DaSi91,Zhang03}).
Recall that, if $V(x):=-|x|^a$ for any $x\in \rn\setminus \{\vec 0_n\}$,
then $V\in K_2(\rn)$ if and only if $a\in (-2,0)$
(see \cite{Simon84} or Remark \ref{ex-GSC} below).
For the critical case $a=-2$, it is known that
the Gaussian upper bounds, \eqref{eqn-2GUB} and \eqref{eqn-2lGUB},
may break down (see \cite{IsKaOu17,MoTe05}).

In what follows, let $\nn:=\{1,2,\ldots\}$, $\zz_+:=\nn\cup\{0\}$,
$D:=(\frac{\pat}{\pat x_1},\cdots,\frac{\pat}{\pat x_n})$, and,
for any $\az:=(\az_1,\cdots,\az_n)\in\zz_+^n:=(\zz_n)^n$,
$|\az|:=\az_1+\cdots+\az_n$ and
$D^\az:=\frac{\pat^{\az_1}}{\pat x_1^{\az_1}}\cdots\frac{\pat^{\az_n}}{\pat x_n^{\az_n}}$.
Let $\cl:=P(D)+V$ be a higher order Schr\"odinger operator on $\rn$
with $P(D)$ being some higher order elliptic operator. The
techniques that can be used to estimate the heat kernel of $\cl$ are limited,
due to the failure of many good properties of the Laplacian
$\Delta$ when the order of the unperturbed operator becomes higher. For instance, it is
known that the biharmonic heat kernel of the operator $\Delta^2$ may
change sign infinitely many times (see \cite{FeGaGr08}). This
indicates that the associated semigroup no longer preserves
the positivity and hence we lose the effect of the Feynman--Kac formula,
which is a fundamental tool in the analysis of the second
order Schr\"odinger operator (see \cite{BaGa13,Dav97} for some excellent
expositions on the state of art of higher-order elliptic
operators and their perturbations).

Let $m\in\nn$ satisfy $n<2m$. If the unperturbed operator
$$P(D)=\sum_{|\az|=m=|\bz|}(-1)^m D^\az(a_{\az,\bz}(x)D^\bz)$$ is a
homogeneous uniformly elliptic operator of order $2m$ on $\rn$ and
the potential $V\in L_\loc^1(\rn)$ is nonnegative locally integrable on $\rn$,
Barbatis and Davies \cite{BaDa96} showed that $\cl$ possesses a heat kernel
$p_t$ on $(0,\fz)\times \rn\times\rn$ that satisfies the following
higher order Gaussian upper bound: there exist some positive constants $C$ and $c_2$
such that, for any $t\in (0,\fz)$ and $x$, $y\in\rn$,
\begin{align}\label{eqn-mmGUB}
\lf|p_t(x,y)\r|\le \frac{C}{t^{n/(2m)}}\exp\lf\{-c_2\frac{|x-y|^{2m/(2m-1)}}{t^{1/(2m-1)}}\r\}.
\end{align}
The key idea used to prove \eqref{eqn-mmGUB} in \cite{BaDa96} is
an exponential perturbation argument. To be precise, let $Q(f,f)$ be the
quadratic form associated with $\cl$ for some suitable function
$f$ in its domain. Barbatis and Davies considered the exponential
perturbation $$Q_{\lz,\phi}(f,f):=Q\lf(e^{\lz\phi} f,e^{-\lz\phi} f\r)$$
of $Q(f,f)$ for some $\lz\in\rr$ and $\phi$ being some smooth
function on $\rn$. They proved an
$L^2(\rn)$ form perturbation estimate of the type that,
for any $\epsilon\in (0,\fz)$, there exists a positive constant
$C_{(\epsilon)}$, independent of $f$, such that, for any $f$ in the domain
of $Q$ on $L^2(\rn)$,
\begin{align}\label{eqn-pqf}
\lf|Q_{\lz,\phi}(f,f)-Q(f,f)\r|\le \epsilon Q(f,f)+C_{(\epsilon)}\lz^{2m}\|f\|_{L^2(\rn)}^2.
\end{align}
Here and thereafter, for any $q\in [1,\fz]$, we use $L^q(\rn)$ to denote the set of
all measurable functions $f$ on $\rn$ such that
\begin{align*}
\lf\|f\r\|_{L^q(\rn)}:=\lf[\dint_\rn \lf|f(x)\r|^q\,dx\r]^{1/q}<\fz
\end{align*}
with the usual modification made when $q=\fz$.
This, together with the Sobolev inequality
\begin{align}\label{eqn-SBI}
W^{m,2}(\rn)\subset L^\fz(\rn),
\end{align}
indicates the following
ultracontractivity of the associated semigroup $\{e^{-t\cl_{\lz,\phi}}\}_{t>0}$
that there exists a positive constant $C$ such that, for any $g\in L^2(\rn)$,
\begin{align}\label{eqn-ucs}
\lf\|e^{-t\cl_{\lz,\phi}}g\r\|_{L^\fz(\rn)}\le C t^{-n/(2m)}\|g\|_{L^2(\rn)};
\end{align}
from this and an optimizing argument, it follows that \eqref{eqn-mmGUB} holds
true. Note that, since the Sobolev inequality \eqref{eqn-SBI} holds only in the
case $n<2m$, the above argument works also only in this case.

To extend the estimates of the heat kernel
to the general case $n\in\nn$, Deng et al. \cite{DeDiYa14} studied
the situation that the unperturbed operator
\begin{align*}
P(D)=\sum_{|\az|=2m}(-1)^m  a_{\az}D^\az
\end{align*}
is a non-negative homogeneous elliptic operator of order $2m$, which has
real constant coefficients $\{a_\az\}_{|\az|=2m}$, and
$V$ is some Kato perturbation of $P(D)$, where the latter means,
for any $\epsilon\in (0,\fz)$, there exists a positive constant $C_{(\epsilon)}$ such that,
for any $f\in \mathcal{C}_{\mathrm c}^\fz(\rn)$ (the set of all infinitely differentiable
functions with compact support),
\begin{align}\label{Katop}
\lf\|Vf\r\|_{L^1(\rn)}\le \epsilon \lf\|P(D) f\r\|_{L^1(\rn)}+C_{(\epsilon)}\lf\|f\r\|_{L^1(\rn)}.
\end{align}
Under this condition, Deng et al. \cite{DeDiYa14} proved the following
$L^1(\rn)$ endpoint operator perturbation estimates of the type that:
for any $\epsilon\in (0,\fz)$, there exists a positive constant
$C_{(\epsilon)}$ such that, for any $f$ in the domain of $\cl$ on $L^1(\rn)$,
\begin{align}\label{eqn-pqf1}
\lf\|\cl_{\lz,\phi} f-P(D)f\r\|_{L^1(\rn)}\le \epsilon \lf\|P(D)\r\|_{L^1(\rn)}
+C_{(\epsilon)}\lf(1+\lz^{2m}\r)\|f\|_{L^1(\rn)},
\end{align}
where $\cl_{\lz,\phi} $ denotes the exponential perturbation of $\cl$ associated with
the form $Q_{\lz,\phi}$ as in \eqref{eqn-pqf}.
Since $[L^1(\rn)]^\ast=L^\fz(\rn)$, one can avoid to use the Sobolev inequality
\eqref{eqn-SBI}, but apply the dual and some iteration arguments,
in order to derive \eqref{eqn-ucs}.
In view of this, Deng et al. \cite{DeDiYa14} proved the following local
version of \eqref{eqn-mmGUB} under the case $n\in\nn$
and $P(D)$ being a nonnegative homogeneous elliptic operator of order $2m$
with real constant coefficients, that is, there exist positive
constants $C$, $c_2$, and $w$ such that, for any $t\in (0,\fz)$ and $x$, $y\in\rn$,
\begin{align}\label{eqn-mlGUB}
\lf|p_t(x,y)\r|\le \frac{C}{t^{n/(2m)}}\exp\lf\{wt-c_2\frac{|x-y|^{2m/(2m-1)}}{t^{1/(2m-1)}}\r\}
\end{align}
if $V\in L_\loc^1(\rn)$ is a Kato type  perturbation of $P(D)$. Recall that, if $V$ is in the higher order
Kato class $K_{2m}(\rn)$ [see \eqref{eqn-katoc} below for its definition],
then $V$ is a Kato type  perturbation of $P(D)$ (see \cite{DeDiYa14,ZhYa09}).
Further developments on the local estimates of the type \eqref{eqn-mlGUB}
can be fund in \cite{Bar17,BaBr18,FSWY20,HuWaZeDu18}.

As was pointed out in \cite{DeDiYa14},
it is still unknown whether or not the local positive constant $w$ in
\eqref{eqn-mlGUB} can be zero even when $V\ge 0$.
This is because the Kato perturbation \eqref{Katop}
can only produce estimates of the type \eqref{eqn-pqf1}
with constant $1+\lz^{2m}$, rather than $\lz^{2m}$ as in \eqref{eqn-pqf}.
Note that, if $w=0$ in \eqref{eqn-mlGUB},
then the local Gaussian upper bound \eqref{eqn-mlGUB} becomes
the global Gaussian upper bound \eqref{eqn-mmGUB}.
In the latter case, much better properties related to
$\cl$ can be obtained (see, for instance, \cite{BGK12,Dav90,Dav97,Gri09,GHL14,HP97}),
because the Gaussian exponential term still works even when $t\in [1,\fz)$.

Motivated by the aforementioned results, it is natural to ask the following question.

\medskip

\noindent{\bf Question.}  Under what general conditions on $V$,
can the local constant $w$ in \eqref{eqn-mlGUB} be zero?

\medskip

In this article, we give an affirmative answer to this question by
introducing a new general Schechter class on
$V$ (see Definition \ref{def-generalizedSC}  below for its definition),
motivated by the now called Schechter class introduced in
\cite{Sch86}. This new potential class, which
coincides with the aforementioned Kato class in some special
cases (see Proposition \ref{rmk-Scheclass} below),
enables us to make $w=0$ in \eqref{eqn-mlGUB} and,
therefore, obtain a global Gaussian upper bound for the heat kernel of the type
\eqref{eqn-mmGUB} for any $n$, $m\in\nn$. Moreover, the H\"older
regularity of the heat kernel is also established in the case $n\ge 2m$.
To be precise,  the main result of this article is as follows.

\begin{theorem}\label{thm-main1}
Let $m\in\nn$, $V$ be a measurable function on $\rn$, and $\cl:=P(D)+V$
the $2m$-order Schr\"odinger operator on $\rn$ as in \eqref{eqn-def-HOSO}
with $P(D)$ being the $2m$-order homogeneous
real constant coefficient elliptic operator as in \eqref{def-PD}. If
\eqref{ap1} and one of \eqref{ap2} through \eqref{ap5} hold true for any
$q\in (1,2]$ or $[2,\fz)$ and
$$\sup_{|\lz|\in(0,\fz)}M_{|\lz|}(V)<1$$ with $M_{|\lz|}(V)$
as in \eqref{HP}, then the operator $\cl$ possesses a heat kernel $p_t$ on
$(0,\fz)\times \rn\times \rn$ that satisfies the following estimates.
\begin{itemize}
\item[{\rm(a)}]
There exist positive constants
$C$ and $c_3$ such that, for any $t\in (0,\fz)$ and $x$, $y\in\rn$,
\begin{align}\label{eqn-mGUB}
\lf|p_t(x,y)\r|\le \frac{C}{t^{n/(2m)}} \exp\lf\{-c_3\frac{|x-y|^{2m/(2m-1)}}
{t^{1/(2m-1)}}\r\}.
\end{align}

\item[{\rm(b)}] If, in addition, $n\ge 2m$, then
there exist a $\gz\in(0,1)$ and positive constants $C$ and $c_4$ such that,
for any $t\in (0,\fz)$ and $x$, $y,\,h\in\rn$ satisfying $|h|<t^{1/2m}$,
\begin{align}\label{eqn-mHE}
\lf|p_t(x+h,y)-p_t(x,y)\r| \le
\frac{C}{t^{n/(2m)}}\exp\lf\{-c_4\frac{|x-y|^{2m/(2m-1)}}{t^{1/(2m-1)}}\r\}
\lf[\frac{|h|}{t^{1/(2m)}}\r]^\gz.
\end{align}
\end{itemize}
\end{theorem}

Theorem \ref{thm-main1} and its local version
(see Corollary \ref{cor-heatkernel} below) are proved in Section \ref{s5.3}.
Recall that the method used in \cite{BaDa96}
can obtain a global Gaussian upper bound \eqref{eqn-mmGUB},
but has the dimension restriction $n<2m$.
The method used in \cite{DeDiYa14}
works in the general case $n\in\nn$, but can only obtain the local
Gaussian upper bound \eqref{eqn-mlGUB}.
Thus, Theorem \ref{thm-main1} makes the first effort
to obtain the global Gaussian upper bound for any dimension $n\in\nn$.

To prove Theorem \ref{thm-main1},
we develop a systematic treatment on the estimates
of the heat kernel of the higher order Schr\"odinger operator $\cl=P(D)+V$
from the condition of the potential $V$, via the spectral perturbation.
More precisely, this treatment is divided into the following three steps:
\begin{enumerate}
\item[(i)] From the generalized Schechter condition on $V$,
we first deduce a series of boundedness of the $T$-operator
$T_{s,\dz}$ from $L^p(\rn)$ to $L^q(\rn)$ for any given $p$, $q\in (1,\fz)$,
$\dz\in (0,\fz)$, and $s\in (0,2m]$, where
\begin{align}\label{eqn-MtoT}
T_{s,\dz}:=V \lf(\dz^2-\Delta\r)^{-s/2}
\end{align}
and $(\dz^2-\Delta)^{-s/2}$ is the \emph{Bessel potential} of order $s$
(see Propositions \ref{prop-bdT1}, \ref{prop-bdT2}, \ref{prop-bdT3},
and \ref{prop-bdT4} below).

\item[(ii)] Using the boundedness of the $T$-operator in \eqref{eqn-MtoT},
we establish the exponential perturbed resolvent estimate [see
\eqref{eqn-forin} below] that there
exists a positive constant $C$ such that,
for any $\lz\in \rho(\cl)$ (the resolvent set of $\cl$) and $f\in L^2(\rn)$,
\begin{align}\label{eqn-intre}
\lf\|\lf(\lz-\cl_\eta\r)^{-(l+1)}f\r\|_{L^\fz(\rn)}\le C |\lz|^{\frac{1}{2}\frac{n}{2m}
-(l+1)}\lf\|f\r\|_{L^2(\rn)}
\end{align}
uniformly for certain $l\in\nn$ and $\eta\in\cc^n$,
where $\cl_\eta$ denotes the exponential perturbation of $\cl$ [see \eqref{eqn-defcleta-1} below].
This exponential perturbed resolvent estimate replaces the ultracontractivity \eqref{eqn-ucs}
of the semigroup used in \cite{BaDa96,DeDiYa14}.

\item[(iii)] From the resolvent estimate \eqref{eqn-intre}, we finally deduce the Gaussian
estimates, \eqref{eqn-mGUB} and \eqref{eqn-mHE}, of $\cl$. An essential tool used in
this step is the following functional calculus identity, originally appeared in \cite{AuMcTc98}, that
\begin{align}\label{step3}
e^{-t\cl}=\frac{[2(l+1)-1]!}{2\pi i (-t)^{2(l+1)-1}}\dint_\Gamma e^{-t\lz} \lf(\lz-\cl\r)^{-2(l+1)}\,d\lz
\end{align}
with $l\in\nn$ and $\Gamma$ being a path in $\rho(\cl)$ (see Figure
\ref{gamma} below). The exponential terms in the Gaussian
estimates, \eqref{eqn-mGUB} and \eqref{eqn-mHE}, then come from the
exponential perturbed resolvent identity [see \eqref{eqn-resolventper} below].
\end{enumerate}

We point out that the parameters $\dz$ and $\lz$ in \eqref{eqn-MtoT}
and \eqref{eqn-intre} are connected by the relation
$\dz=|\lz|^{1/(2m)}$. This, in view of \eqref{step3},
indicates that we need all values of $\lz$ in the path $\Gamma$
and hence all $\dz\in (0,\fz)$ in the proof of Theorem
\ref{thm-main1}. That is why
all $\dz\in (0,\fz)$ are taken into consideration in the definition
of the generalized Schechter class [see \eqref{eqn-Mclass} below].
Recall that, in the case of the Kato class [see \eqref{eqn-katoc} below],
only the information of $\dz$ near $0$ is considered. Also,
since the functions in the generalized Schechter class
may have negative parts, Theorem \ref{thm-main1} is new even when $n<2m$.

The proof of Theorem \ref{thm-main1}(b) depends on the following
Sobolev embedding
$$W^{2m,q}(\rn)\subset \mathcal{C}^\gz(\rn)$$
with $\mathcal{C}^\gz(\rn)$ being the Lipschitz space of order
$\gz:=2m-n/q\in (0,1)$, when $q\in (1,\fz)$ and $(2m-1)q<n<2mq$
[see Lemma \ref{lem-GN2}(ii) below], where the latter condition
implies the dimension condition $n\ge 2m$.

Using an approach similar to that used in the proof of Theorem \ref{thm-main1}, we
are able to establish the following Davies--Gaffney estimates of the heat
semigroup generated by $-\cl$. Recall that,
for many Schr\"odinger operators, they may even not possess
a heat kernel, not to mention the (local) Gaussian upper bound
(see \cite{BeSe90,IsKaOu17}). In such a case, the Davies--Gaffney estimates
are good substitutes. The following Theorem \ref{thm-main2} and its local version (see Corollary
\ref{cor-localDG} below) are proved in Section \ref{s5.2}.

\begin{theorem}\label{thm-main2}
Let $m\in\nn$, $V$ be a measurable function on $\rn$, and
$\cl:=P(D)+V$ the $2m$-order Schr\"odinger operator on
$\rn$ as in \eqref{eqn-def-HOSO} with $P(D)$ being
the $2m$-order homogeneous real constant coefficient elliptic operator as in \eqref{def-PD}. If
\eqref{ap1} and one of \eqref{ap2} through \eqref{ap5} hold true with
$q=2$ and
$$\sup_{|\lz|\in(0,\fz)}M_{|\lz|}(V)<1$$ with $M_{|\lz|}(V)$
as in \eqref{HP}.
Assuming further that $\cl$ satisfies
\eqref{eqn-assumptionadd1} for any $\lz\in \rho(\cl)$, then there
exist positive constants $C$ and $c_5$ such that, for any
disjoint compact convex subsets $E$ and $F$, $t\in (0,\fz)$, and $f\in L^2(E)$ with
$\supp f:=\{x\in\rn:\ f(x)\ne 0\}\subset E$,
\begin{align}\label{eqn-DGE}
\lf\|e^{-t\cl}f\r\|_{L^2(F)}\le C
\exp\lf\{-c_5\frac{[d(E,F)]^{2m/(2m-1)}}{t^{1/(2m-1)}}\r\}\lf\|f\r\|_{L^2(E)},
\end{align}
here and thereafter, $d(E,F):=\inf_{x\in E,y\in F} |x-y|$
and $\|f\|_{L^2(E)}:=[\int_E |f(x)|^2\,dx]^{1/2}$.
\end{theorem}

Now, let $\cl$ be the higher order Schr\"odinger operator satisfying the assumptions
of Theorems \ref{thm-main1} and \ref{thm-main2}. Applying these both theorems,
we immediately obtain the following conclusions:

\begin{itemize}
\item[(i)]  The spectrum $\sz_p(\cl)$ of $\cl$ in $L^p(\rn)$ is independent of $p$
for any given $p\in [1,\fz)$ (see \cite[Theorem 1.1]{Ku99}).

\item[(ii)] The operator $\cl$ has a bounded $H_\fz$-functional calculus on $L^p(\rn)$
for any given $p\in (1,\fz)$ (see \cite[Theorem 3.1]{DuRo96}).

\item[(iii)] When $n>2m$, the integral kernel $K_\lz$ of
the resolvent $(\lz-\cl)^{-1}$ satisfies that, for any $x$, $y\in\rn$ with $x\ne y$,
\begin{align*}
\lf|K_\lz(x,y)\r|\ls \frac{1}{|x-y|^{n-2m}}e^{-c|\lz|^{1/(2m)}|x-y|},
\end{align*}
where the implicit positive constant and the positive constant $c$
are independent of $x$,  $y$ and $\lz$ (see \cite[Theorem 2.2]{Hie96}).

\item[(iv)] For any given $p,\,q\in (1,\fz)$,
let $f\in L^p([0,\fz); L^q(\rn))$  with
\begin{align*}
&L^p\lf([0,\fz); L^q(\rn)\r)\\
&\hs:=\lf\{f:\rn\times [0,\fz) \to \rr:\ \lf\|f\r\|_{L^p([0,\fz); L^q(\rn))}:=\lf[\dint_{0}^\fz
\lf\|f(\cdot,t)\r\|_{L^q(\rn)}^p\,dt\r]^{1/p}<\fz\r\}.
\end{align*}
The following inhomogeneous initial value problem
\begin{align*}
\begin{cases}
\displaystyle\frac{\pat}{\pat t}u(x,t)=\cl u(x, t)+f(x,t),\ \ &(x,t)\in\rn\times (0,\fz),\\
\displaystyle u(x,0)=0, \ \ & x\in\rn
\end{cases}
\end{align*}
has a unique solution that is of maximal $L^p(\rn)$-$L^q(\rn)$ regularity
(see \cite[Theorem 3.1]{HP97}).
\end{itemize}

We also point out that, since the estimate of the
heat kernel is the
start point of many studies on the analysis of the Schr\"odinger operator, our main results pave the way
for further studies such as the Sobolev
inequalities  (in particular, see \cite{BoCoSi15,Ouha05} for the
Nash and the Gagliardo--Nirenberg inequalities),
the boundedness of some singular integral operators,
and the real-variable theory of function spaces associated with
$\cl$. We do not pursue these problems here, in order to limit
the length of this article.

The remainder of this article is organized as follows. In Section \ref{s2}, we provide
some basic facts on the higher order Schr\"odinger operator
$\cl:=P(D)+V$. We first review the definition of $\cl$ in Section
\ref{s2.1}; then, in Section \ref{s2.2}, we introduce the definition
of the generalized Schechter class concerning the potential $V$. Some basic
properties of the Schechter class are also presented in this
section. Section \ref{s3} is devoted to the boundedness of the
$T$-operator defined in \eqref{def-Tlz}. We obtain four kinds of
boundedness in this section. In Section \ref{s4}, we establish two
perturbation estimates for the resolvent related to $\cl$: the summation
perturbations in Section \ref{s4.1} and the exponential
perturbations in Section \ref{s4.2}. Finally, we prove our main
results, Theorems \ref{thm-main1} and \ref{thm-main2}, in Section
\ref{s5}.

We end this section by making some conventions on the notation. Let
$\nn:=\{1,2,\ldots\}$, $\zz_+:=\nn\cup\{0\}$, and $\zz:=\{0,\pm1,
\pm2,\dots\}$. For any $s\in\rr$, let $\lfloor s\rfloor$ be the
largest integer not greater than $s$.
For any set $E\subset \rn$, we use $\mathbf{1}_E$ to
denote its characteristic function. We use $C$ to denote a
\emph{positive constant} that is independent of main parameters
involved, whose value may differ from line to line. Constants with
subscripts, such as $C_1$ and $c_1$, do not change in different
occurrences. We also use $C_{(\az,\,\bz,\ldots)}$ to denote a
positive constant depending on the indicated parameters $\az$,
$\bz,\ldots$. If $f\le Cg$, we write $f\ls g$ and, if $f\ls g\ls f$,
we then write $f\sim g$. If $f\le Cg$ and $g=h$ or $g\le h$, we then
write $f\ls g\sim h$ or $f\ls g\ls h$, \emph{rather than} $f\ls g=h$
or $f\ls g\le h$. We use $\vec 0_n$ to denote the origin of $\rn$.

\section{Higher order Schr\"odinger operators\label{s2}}
In this section, we provide some basic facts on the higher order Schr\"odinger operator
$\cl$ and its potential. We begin with a review of the definition of  $\cl$.

\subsection{Preliminaries on Schr\"odinger operators \label{s2.1}}
Let $P(x)=\sum_{|\az|=2m}a_\az x^{\az}$ be a homogeneous polynomial of degree $2m$
on $x\in \rn$ with real
constant coefficients $\{a_\az\}_{|\az|=2m}$ that satisfy the {\it uniform ellipticity condition}:
there exists a positive constant $\lz\in (0,\fz)$ such that, for any $x\in\rn$,
\begin{align}\label{eqn-uec}
\dsum_{|\az|=2m}a_\az x^\az\ge \lz |x|^{2m}.
\end{align}
For any $f\in \mathcal{C}_{\mathrm c}^\fz(\rn)$, define the differential operator $P(D)$ of order $2m$ on $f$ by
\begin{align}\label{def-PD}
P(D)f:=\dsum_{|\az|=2m}(-1)^m a_\az D^\az f.
\end{align}
It is known that $P(D)$ can be extended to a nonnegative
self-adjoint operator in $L^2(\rn)$ with domain
$\mathrm{dom}\,(P(D))=W^{2m,2}(\rn)$ being the Sobolev space (see
\cite[p. 62, Corollary 2.2]{Sch86}). We call this self-adjoint
extension the {\it $2m$-order homogeneous elliptic operator with
real constant coefficients on $\rn$} and we still use the same
notation $P(D)$ to denote it.

Recall the following properties of $P(D)$ from \cite[Proposition 45]{AQ00},
\cite[p. 177, Problem III-6.16]{Kato95}, and \cite[p. 65, Corollary 3.4]{Sch86}.
In what follows, for any given linear normed spaces $\mathcal{X}$ and $\mathcal{Y}$
and any linear operator $T$ mapping $\mathcal{X}$ into $\mathcal{Y}$,
we use $\|T\|_{\mathcal{X}\to \mathcal{Y}}$ to denote its operator norm.

\begin{lemma}\label{lems2-1}
Let $P(D)$ be a $2m$-order homogeneous elliptic operator with real
constant coefficients on $\rn$ as in \eqref{def-PD}, and let
$p_t(x,y)$  defined on $(0,\fz)\times\rn\times \rn$  be the heat
kernel of the semigroup generated by $-P(D)$. The following
assertions hold true:
\begin{itemize}
\item[{\rm (i)}] The resolvent set of $P(D)$,
$\rho(P(D))=\cc\setminus [0,\fz)$ and, for any $\lz\in \rho(P(D))$,
\begin{align*}
\lf\|(\lz-P(D))^{-1}\r\|_{L^2(\rn)\to L^2(\rn)}\le\frac{1}{d(\lz,[0,\fz))},
\end{align*}
where $d(\lz,[0,\fz)):=\inf_{s\in [0,\fz)} |\lz-s|$.

\item[{\rm (ii)}] There exist positive constants $C$ and $c_6$ such that,
for any $l\in\{0,\ldots,m-1\}$, $t\in (0,\fz)$, and $x$, $y\in\rn$,
\begin{align*}
\lf|D_x^l p_t(x,y)\r|\le
\frac{C}{t^{(n+l)/(2m)}}\exp\lf\{-c_6\frac{|x-y|^{2m/(2m-1)}}{t^{1/(2m-1)}}\r\}.
\end{align*}

\item[{\rm (iii)}] For any given $\gz\in(0,1)$, there exist positive constants $C$ and $c_7$ such that,
for any $l\in\{0,\ldots,m-1\}$, $t\in (0,\fz)$, and $x$, $y,\,h\in\rn$,
\begin{align*}
\lf|D_x^{l} p_t(x+h,y)-D_x^{l} p_t(x,y)\r| \le
\frac{C}{t^{(n+l)/(2m)}}\exp\lf\{-c_7\frac{|x-y|^{2m/(2m-1)}}{t^{1/(2m-1)}}\r\}
\lf[\frac{|h|}{t^{1/(2m)}}\r]^\gz.
\end{align*}
\end{itemize}
\end{lemma}

Now, let $V:\rn\to \rr$ be a measurable function on $\rn$.
It induces a {\it multiplication operator} $f\mapsto Vf$ in $L^2(\rn)$
with the domain
$$\mathrm{dom}\,(V):=\lf\{f\in L^2(\rn):\ Vf\in L^2(\rn)\r\}.$$
From \cite[p.72, Lemma 6.1]{Sch86},
it follows that $V$ is a closed symmetric operator in $L^2(\rn)$.
The operator $V$ is said to be {\it relatively $P(D)$-bounded}
if $\mathrm{dom}\,(P(D))\subset \mathrm{dom}\,(V)$ and there exist constants
$a$, $b\in (0,\fz)$ such that, for any ${g}\in \mathrm{dom}\,(P(D))\subset L^2(\rn)$,
\begin{align*}
\lf\|V{g}\r\|_{L^2(\rn)}\le a\lf\|P(D){g}\r\|_{L^2(\rn)}+b\lf\|{g}\r\|_{L^2(\rn)},
\end{align*}
where the infimum of all such $a$ is called the {\it relative bound of $V$ with respect to} $P(D)$.
Recall the following W\"ust  theorem from \cite[Theorem X. 14]{ReedSimon75-II}.

\begin{lemma}\label{lemma-wust}
Let $A$ be a self-adjoint operator and $B$ a symmetric operator in a Hilbert space $\mathcal{H}$.
Assume that $B$ is relatively $A$-bounded with relative bound $a\le 1$. Then the sum $A+B$
of the operators $A$ and $B$ is essential self-adjoint on $\mathrm{dom}\,(A)$.
\end{lemma}

By Lemma \ref{lemma-wust}, we know that, if $V$ is relatively $P(D)$-bounded with
relative bound $a\le 1$,
then the operator $P(D)+V$ is essentially self-adjoint on $W^{2m,2}(\rn)$.
Denote by
\begin{align}\label{eqn-def-HOSO}
\cl:=P(D)+V
\end{align}
the nonnegative self-adjoint extension of $P(D)+V$ in $L^2(\rn)$. We
call $\cl$ the {\it $2m$-order Schr\"odinger operator on $\rn$}
(see also Proposition \ref{prop-reb} below for an extension of
$\cl$ to any space $L^p(\rn)$ under some Schechter-type conditions).
Note that, if  $P(D)=-\Delta$ is the Laplace operator, then
$\cl:=-\Delta+V$ is the usual second order Schr\"odinger operator on $\rn$.

To study the operator $\cl$, various conditions on the potential $V$
were introduced in literatures. For instance, Kato first introduced
the Kato class $K_2(\rn)$ in \cite{Kato72},
which is useful in the study of many problems related to
the Schr\"odinger operator $\cl$ (see
\cite{DaHi98,Simon84,ZhYa09}). Recall that a real measurable function $V$ on
$\rn$ is said to be in the {\it Kato class} $K_\az(\rn)$, for any given $\az\in (0,\fz)$,
if
\begin{align}\label{eqn-katoc}
\lim_{\dz\to 0}\dsup_{x\in \rn}\dint_{|y-x|<\dz}|V(y)|w_{\az}(x-y)\,dy=0,
\end{align}
where, for any $x\in \rn\setminus \{\vec 0_n\}$,
\begin{align}\label{def-waz}
w_\az(x):=\begin{cases}
|x|^{\az-n} \ \ & \text{if}\ \az\in (0,n),\\
\log(|x|^{-1}) \ \ & \text{if}\ \az=n,\\
1 \ \ & \text{if}\ \az\in (n,\fz).
\end{cases}
\end{align}
It is known that $K_\az(\rn)$ is a closed subspace of a larger
Banach space $\wz K_\az(\rn)$, which is defined to be the set of  all $V\in
L_\loc^1(\rn)$ such that
\begin{align*}
\|V\|_{\wz K_\az (\rn)}:=\dsup_{x\in\rn} \dint_{|y-x|<1/2}
|V(y)|w_\az(x-y)\,dy<\fz.
\end{align*}

Shen considered a number of Schr\"odinger operators with potentials
respectively in the reverse H\"older class (\cite{Shen95}) and the Morrey space (\cite{Shen02}).
Recall that, for any given $p\in (1,\fz)$, a nonnegative measurable
function $V$ is said to belong to the {\it reverse H\"older class}
$RH_p(\rn)$ if there exists a positive constant $C$ such that,
\begin{align}\label{eqn-RHC}
\lf(\frac{1}{|B|}\dint_{B}[V(x)]^p\,dx\r)^{\frac{1}{p}}\le \frac{C}{|B|}\dint_{B}V(x)\,dx
\end{align}
for any ball $B$ of $\rn$. Also, for any given $p\in (1,\fz)$ and $\lz\in [0,n]$,
a measurable function $f$ is said to be in the {\it Morrey space} $L_{p,\lz}(\rn)$ if
\begin{align}\label{eqn-Morrey}
\|f\|_{L_{p,\lz}(\rn)}:=\dsup_{\substack{
x\in\rn\\
r\in (0,\fz)}}
\lf[r^{\lz-n}\dint_{B(x,r)}\lf|f(y)\r|^p\,dy\r]^{\frac{1}{p}}<\fz,
\end{align}
here and thereafter, for any $x\in \rn$ and $r\in (0,\fz)$, $B(x,r):=\{y\in \rn:\
|y-x|<r\}$; see also \cite{Sch86,Simon84,Stu56} for many other kinds of
potential classes.

\subsection{Generalized Schechter classes \label{s2.2}}

In this subsection, we introduce the generalized Schechter
class which has close relations with the aforementioned potential classes.
In what follows, for any given $r\in (1,\fz)$, we use $L_\loc^r(\rn)$ to denote
the set of all measurable functions $f$ such that $|f|^r\in L_\loc^1(\rn)$.

\begin{definition}\label{def-generalizedSC}
Suppose $V$ is a measurable function on $\rn$.
For any given $\az\in (0,\fz)$, $r\in [1,\fz)$, $t\in [1,\fz]$, and $\dz\in (0,\fz)$,
let
\begin{align}\label{def-Mazqrlz}
M_{\az,r,t,\dz}(V):=\lf\|\lf[\dint_{|y-\cdot|<\dz}\lf|V(y)\r|^r
w_\az({\cdot-y})\,dy\r]^{\frac{1}{r}}\r\|_{L^t(\rn)}
\end{align}
with $w_\az$ as in \eqref{def-waz}.
Assume that ${S}\in\rr$ is a real number.
The {\it generalized Schechter classes} $M_{\az,r,t,{S}}(\rn)$
and $\wz M_{\az,r,t,{S}}(\rn)$ are defined, respectively, by setting
\begin{align}\label{eqn-Mclass}
M_{\az,r,t,{S}}(\rn):=\lf\{V\in L_\loc^r(\rn):\
\|V\|_{M_{\az,r,t,{S}}(\rn)}:=\dsup_{\dz\in (0,\fz)} \dz^{{{S}}}M_{\az,r,t,\dz}(V)<\fz\r\}
\end{align}
and
\begin{align}\label{eqn-wzMclass}
\wz M_{\az,r,t,{S}}(\rn):=\lf\{V\in L_\loc^r(\rn):\
\lim_{\dz\to 0} \dz^{{{{S}}}}M_{\az,r,t,\dz}(V)=0\r\}.
\end{align}
\end{definition}

\begin{remark}\label{rem-example}
The spaces $M_{\az,r,t,{S}}(V)$ and  $\wz M_{\az,r,t,{S}}(\rn)$ are motivated by
the definitions of the now called Schecheter class
in \cite[Chapter 6.4]{Sch86}, where Schechter originally introduced
the following {\it Schechter class}
\begin{align*}
M_{\az,r,t}(\rn):=\lf\{V\in L_\loc^r(\rn):\ M_{\az,r,t,0}(V)<\fz\r\}.
\end{align*}
The parameter ${S}$ in \eqref{eqn-Mclass} and \eqref{eqn-wzMclass} comes
from the scaling of $\dz\in (0,\fz)$.
\end{remark}

The following proposition establishes the relations between the generalized Schechter class and some
other known potential classes.

\begin{proposition}\label{rmk-Scheclass}
Let $\az\in (0,\fz)$, $r\in [1,\fz)$, $t\in [1,\fz]$, and $S\in\rr$ and
let $V$ be a measurable function on $\rn$.
\begin{itemize}
\item [{\rm (i)}]  Then $\wz M_{\az,1,\fz,0}(\rn)=K_\az(\rn)$ with $K_\az(\rn)$ being the
Kato class as in \eqref{eqn-katoc}.

\item [{\rm (ii)}] Let $\az\in (0,n)$, $p\in (nr/\az,\fz)$, and  $0\le V\in RH_p(\rn)$
be as in \eqref{eqn-RHC}.
For any given $\dz\in (0,\fz)$, let $ m\,(V,\dz):=\sup_{x\in\rn}\int_{B(x,\dz)}
V(y)\,dy$. If $  \lim_{\dz\to 0}\dz^{({S}+\az)/r-n}m\,(V,\dz)=0$, then
$$V\in \wz M_{\az,r,\fz,{S}}(\rn).$$
In particular, the above assertion holds true when
$\az\in (0,n)$, ${S}+\az-nr\in[0,\fz)$, and $V\in RH_p(\rn)\cap
L_{\mathrm{unif}}^1(\rn)$ with $p\in (nr/\az,\fz)$, where
$$L_{\mathrm{unif}}^1(\rn):=\lf\{f\in L_\loc^1(\rn):\
\|f\|_{L_{\mathrm{unif}}^1(\rn)}:=\sup_{x\in\rn}\int_{|y-x|<1}|f(y)|\,dy<\fz\r\}.$$

\item [{\rm (iii)}] Let $\az\ne n$, $Sr+\az\in  [0,n]$, and $L_{r,\lz}(\rn)$ be the Morrey space
as in \eqref{eqn-Morrey} with $\lz={S} r+\min\{\az,n\}$ and
$S\in (-\fz,0)$ when $\az\in (0,n)$. Then
$$L_{r,\lz}(\rn)=M_{\az,r,\fz,{S}}(\rn).$$
\end{itemize}
\end{proposition}

\begin{proof}
Note that (i) is an easy consequence of \eqref{eqn-katoc} and \eqref{eqn-wzMclass}.

To prove (ii), by assumptions $\az\in (0,n)$, $p\in (nr/\az,\fz)$,
and $0\le V\in RH_p(\rn)$, we obtain
$(\az-n)(p/r)'+n>0$ and
\begin{align*}
\lim_{\dz\to 0}\dz^{{S}}M_{\az,r,\fz,\dz}(V)&=\lim_{\dz\to 0}\dz^{{S}}
\dsup_{x\in\rn} \lf\{\dint_{|y-x|<\dz}[V(y)]^r|x-y|^{\az-n}\,dy\r\}\\
&\ls \lim_{\dz\to 0}\dz^{{S}+nr/p}\dsup_{x\in\rn} \lf\{ \frac{1}{\dz^n}\dint_{|y-x|<\dz}
[V(y)]^p\,dy\r\}^{\frac{r}{p}}\lf\{\dint_{|y-x|<\dz} |x-y|^{(\az-n)(p/r)'}\,dy\r\}^{\frac{1}{(p/r)'}}\\
&\sim\lim_{\dz\to 0}\dz^{{S}+\az-nr}[m(V,\dz)]^r,
\end{align*}
which turns to $0$ when $  \lim_{\dz\to 0}\dz^{({S}+\az)/r-n}m(V,\dz)=0$.
This, combined with \eqref{eqn-wzMclass}, shows that $V\in \wz
M_{\az,r,\fz,{S}}(\rn)$ and hence  (ii) holds true.

To prove (iii), we consider two cases on the size of $\az$.  If $\az\in (n,\fz)$, then ${{{S}}r\in
(-n,0)}$. Thus, by \eqref{eqn-Morrey} through \eqref{eqn-Mclass}, and
\eqref{def-waz}, it is easy to see that
$L_{r,{{S}}r+n}(\rn)=M_{\az,r,\fz,{S}}(\rn)$.

If $\az\in (0,n)$, for any $x\in\rn$ being fixed,
let $d\mu_x(y):=\mathbf{1}_{B(x,\dz)}(y)|V(y)|^r\,dy$
be a nonnegative measure on $\rn$. Then we obtain
\begin{align*}
{N}_\dz(x):=&\,\dz^{{S}}\lf[\dint_{|y-x|<\dz} |V(y)|^r |x-y|^{\az-n}\,dy\r]^{\frac{1}{r}}
=\dz^{{S}}\lf[\dint_{\rn} |x-y|^{\az-n}\,d\mu_x(y)\r]^{\frac{1}{r}}\\
=&\,(n-\az)^{1/r}\dz^{{S}}\lf[\dint_{0}^\fz s^{n-\az}
\mu_x\lf(\lf\{y:\ |y-x|^{-1}>s\r\}\r)\,\frac{ds}{s}\r]^{\frac{1}{r}}\\
\sim&\,\dz^{{S}}\lf[\dint_{0}^\fz s^{\az-n}
\mu_x\lf(B(x,s)\r)\,\frac{ds}{s}\r]^{\frac{1}{r}}\\
\sim&\,\dz^{{S}}\lf[\dint_{0}^\fz s^{\az-n}
\lf\{\dint_{B(x,s)\cap B(x,\dz)}|V(y)|^r\,dy\r\}
\,\frac{ds}{s}\r]^{\frac{1}{r}}\\
\sim&\,\dz^{{S}}\lf[\dint_{0}^\dz s^{\az-\lz}\lf\{s^{\lz-n}\dint_{B(x,s)}|V(y)|^r\,dy\r\}\,\frac{ds}{s}
+\dz^{n-\lz}\dint_{\dz}^\fz s^{\az-n}\lf\{\dz^{\lz-n}\dint_{B(x,\dz)}|V(y)|^r\,dy\r\}\,\frac{ds}{s}\r]^{\frac{1}{r}}\\
\ls&\,\dz^{{S}+(\az-\lz)/r}\|f\|_{L_{r,\lz}(\rn)},
\end{align*}
where we used \eqref{eqn-Morrey} and the assumption $\lz=Sr+\az<\az$
in the last inequality. This, together
with \eqref{eqn-Mclass} and the assumption $Sr+\az-\lz=0$ when $\az\in (0,n)$, shows that
$$M_{\az,r,\fz,{S}}(V)=\dsup_{\dz\in (0,\fz)}\dsup_{x\in\rn} N_\dz(x)\ls \|f\|_{L_{r,\lz}(\rn)}.$$

On the other hand, by \eqref{eqn-Morrey} and the assumptions that
$\az\in (0,n)$ and ${S} r+\az-\lz=0$ again, we have
\begin{align*}
\|f\|_{L_{r,\lz}(\rn)}&=\dsup_{\dz\in (0,\fz)}\dsup_{x\in\rn}\lf[\dz^{\lz-n}
\dint_{B(x,\dz)}\lf|f(y)\r|^r\,dy\r]^{\frac{1}{r}}\\
&\ls \dsup_{\dz\in (0,\fz)}\dsup_{x\in\rn}\lf[\dz^{\lz-\az}
\dint_{B(x,\dz)}\lf|f(y)\r|^r|x-y|^{\az-n}\,dy\r]^{\frac{1}{r}}\\
&\ls \dsup_{{\dz>0}} \dz^{S}\dsup_{x\in\rn}\lf[\dint_{B(x,\dz)}\lf|f(y)\r|^r|x-y|^{\az-n}\,dy\r]^{\frac{1}{r}},
\end{align*}
which, combined with \eqref{eqn-Mclass}, indicates $\|f\|_{L_{r,\lz}(\rn)}\ls
M_{\az,r,\fz,{S}}(V)$. This shows (iii) and hence finishes
the proof of Proposition \ref{rmk-Scheclass}.
\end{proof}

\begin{remark}\label{ex-GSC}
For any given $\az\in (0,n)$, $r\in [1,\fz)$,
$a\in (-{\az}/{r},\fz)$, $\dz\in (0,\fz)$, and, for any $x\in\rn\setminus \{\vec 0_n\}$,
let $V(x)=\pm|x|^a$ and
\begin{align}\label{eqn-ndz}
N_{\az,r,\dz}(x):=\lf[\dint_{|y-x|<\dz}|V(y)|^r |x-y|^{\az-n}\,dy\r]^{\frac{1}{r}}.
\end{align}
By an elementary calculation, we find that there exists a positive constant $C$, independent
of $\dz$, such that, for any $x\in\rn$,
\begin{align*}
N_{\az,r,\dz}(x)\le C
\begin{cases}
\dz^{\az/r}|x|^a\ \ & \text{if}\ |x|\ge 2\dz,\\
\dz^{\az/r+a}\ \ & \text{if}\ |x|< 2\dz.
\end{cases}
\end{align*}
If $a$ further satisfies $a\in(-\fz,-{n}/{t})$, then
\begin{align*}
M_{\az,r,t,\dz}(V)=\|N_{\az,r,\dz}\|_{L^t(\rn)}\ls \dz^{\az/r+a+n/t}
\end{align*}
with the implicit positive constant independent of $\dz$.
Thus, we conclude that, for any given
$\az\in (0,n)$, $r\in [1,\fz)$, $t\in [1,\fz]$, $S\in\rr$, and
$a\in (-{\az}/{r},-{n}/{t})$,
\begin{align*}
\pm|x|^a\in
\begin{cases}
\displaystyle M_{\az,r,t,S}(\rn) \ \ &\text{if}\   S+\displaystyle\frac{\az}{r}+a+\frac{n}{t}=0,\\
\displaystyle \wz M_{\az,r,t,S}(\rn) \ \ &\text{if}\  S+\displaystyle\frac{\az}{r}+a+\frac{n}{t}>0.
\end{cases}
\end{align*}
In particular, by applying Proposition \ref{rmk-Scheclass}(i),
we know that $\pm|x|^a$ is in the Kato class $K_\az(\rn)$ as in \eqref{eqn-katoc}
if $a\in (-\az,0)$.

Similarly, for any given $a\in (-\fz,-n/t)$ and any $x\in\rn$,
let $V(x)=(1+|x|)^a$ and
$N_{\az,r,\dz}(x)$ be as in \eqref{eqn-ndz}.
By an elementary calculation, we have
\begin{align*}
N_{\az,r,\dz}(x)\ls
\begin{cases}
\dz^{\az/r}|x|^a\ \ & \text{if}\ |x|\ge 2\dz,\\
\dz^{\az/r}\ \ & \text{if}\ |x|< 2\dz
\end{cases}
\end{align*}
with the implicit positive constant independent of $\dz$ and $x$.
This implies that
$$\pm(1+|x|)^a \in\wz M_{\az,r,t,S}(\rn)$$
if $S+\frac{\az}{r}+a+\frac{n}{t}>0$. In particular,
$\pm(1+|x|)^a$ is in the Kato class $K_\az(\rn)$ as in \eqref{eqn-katoc}
when $a\in (-\az,0)$.
\end{remark}

We end this section by giving some embedding properties of generalized Schechter classes.
Another embedding property, based on the properties of the
Bessel potential, is established in Proposition \ref{prop-EBP2} below.

\begin{proposition}\label{prop-EBP}
Let $\az\in (0,n)$, $r\in [1,\fz)$, $t\in [1,\fz]$, and $S\in\rr$.
\begin{itemize}
\item[{\rm (i)}] Then $M_{\az,r,\fz,{S}}(\rn)\subset L_{\mathrm{unif}}^r(\rn)$, where
\begin{align}\label{eqn-defLu}
L_{\mathrm{unif}}^r(\rn):=\lf\{f\in L_\loc^r(\rn):\ \|f\|_{L_{\mathrm{unif}}^r(\rn)}:=\dsup_{x\in\rn}
\lf[\dint_{|y-x|<1} |f(y)|^r\,dy\r]^{1/r}<\fz\r\}.
\end{align}

\item [{\rm(ii)}] For any $\az_1\in (0,\az]$,
\begin{align*}
M_{\az-\az_1,r,t,S+\az_1/r}(\rn)\subset M_{\az,r,t,S}(\rn).
\end{align*}

\item [{\rm(iii)}] For any $a\in (1,\fz)$ and $\tz\in (0,1)$ satisfying
$  1-\frac{n}{(n-\az)a'}<\tz<\frac{n}{(n-\az)a}$ with $  a':=\frac{a}{a-1}$,
\begin{align*}
M_{\az_1,ra,t,S+\wz S}(\rn)\subset M_{\az,r,t,S}(\rn),
\end{align*}
where $\az_1:=n+(\az-n)\tz a$ and
$  \wz{S}:=\frac{(\az-n)(1-\tz)a'+n}{a'r}$.
\end{itemize}
\end{proposition}

\begin{proof}
Note that (i) is an easy consequence of \eqref{eqn-defLu} and
\eqref{eqn-Mclass} by taking $\dz=1$.

To prove (ii), let $V$ be a measurable function on $\rn$.
It follows from \eqref{def-Mazqrlz} and \eqref{def-waz} that
\begin{align*}
\dz^{S}M_{\az,r,t,\dz}(V)&=\dz^S \lf\|\lf[\dint_{|y-\cdot|<\dz}
\lf|V(y)\r|^r|y-\cdot|^{\az-n}\,dy\r]^{1/r}\r\|_{L^t(\rn)}\\
&\le\dz^{S+\az_1/r} \lf\|\lf[\dint_{|y-\cdot|<\dz}
\lf|V(y)\r|^r|y-\cdot|^{\az-\az_1-n}\,dy\r]^{1/r}\r\|_{L^t(\rn)}
=\dz^{S+\az_1/r}M_{\az-\az_1,r,t,\dz}(V),
\end{align*}
which, together with \eqref{eqn-Mclass}, shows (ii).

To prove (iii), by \eqref{def-Mazqrlz}, we write
\begin{align*}
\dz^{S}M_{\az,r,t,\dz}(V)&=\dz^S \lf\| \lf[\dint_{|y-\cdot|<\dz}
\lf|V(y)\r|^r|y-\cdot|^{\az-n}\,dy\r]^{1/r} \r\|_{L^t(\rn)}\\
&\le \dz^S \lf\| \lf[\dint_{|y-\cdot|<\dz}
\lf|V(y)\r|^{ra}|y-\cdot|^{(\az-n)\tz a}\,dy\r]^{1/(ar)}\r.\\
&\lf.\hs\hs\hs\times\lf[\dint_{|y-\cdot|<\dz}
|y-\cdot|^{(\az-n)(1-\tz)a'}\,dy\r]^{1/(a'r)} \r\|_{L^t(\rn)}.
\end{align*}
Since $  1-\frac{n}{(n-\az)a'}<\tz<\frac{n}{(n-\az)a}$, we have
$n+(\az-n)(1-\tz)a'>0$ and $n+(\az-n)\tz a\in (0,n)$, which further
implies that
\begin{align*}
\dz^{S}M_{\az,r,t,\dz}(V)&\ls \dz^{S+\frac{(\az-n)(1-\tz)a'+n}
{a'r}}\lf\| \lf[\dint_{|y-\cdot|<\dz}
\lf|V(y)\r|^{ra}|y-\cdot|^{(\az-n)\tz a}\,dy\r]^{1/(ar)} \r\|_{L^t(\rn)}\\
&\sim \dz^{S+ \wz{S}}M_{\az_1,ra,t,\dz}(V),
\end{align*}
where we used the definitions of $\az_1$ and $\wz S$ in the last
equality. This, combined with \eqref{eqn-Mclass}, shows (iii)
and hence finishes the proof of Proposition
\ref{prop-EBP}.
\end{proof}

\section{Boundedness of $T$-operators\label{s3}}

Let $s,\,\dz\in (0,\fz)$ and $V$ be a measurable function on $\rn$.
The {\it $T$-operator} $T_{s,\dz}$ is defined by setting,
for any $f\in \mathcal{S}(\rn)$ (the set of all Schwartz functions)
and $x\in \rn$,
\begin{align}\label{def-Tlz}
T_{s,\dz}(f)(x):=V (G_{s,\dz}\ast f)(x),
\end{align}
where
\begin{align}\label{def-Besselp}
G_{s,\dz}(x):=\mathcal{F}^{-1} \lf[(2\pi)^{-n/2} \lf(\dz^2+|\xi|^2\r)^{-s/2}\r](x)
\end{align}
denotes the {\it $s$-order Bessel potential function} and $\mathcal{F}^{-1}$ the inverse
Fourier transform. If $\dz=1$, we remove the subscript $\dz$
and write $G_s(x)$ simply. Recall that, for any $f\in \mathcal{S}(\rn)$ and
$\xi$, $x\in\rn$,
$\mathcal{F}(f)(\xi):=\frac{1}{(2\pi)^{n/2}}\int_\rn e^{-ix\cdot \xi} f(x)\,dx$ and
$\mathcal{F}^{-1}(f)(x):=\mathcal{F}(f)(-x)$.
Moreover,
\begin{align}\label{def-Besselp2}
G_{s,\dz}\ast f=(\dz^2-\Delta)^{-s/2}f.
\end{align}
The main purpose of this section is to study the boundedness of the
$T$-operator as in \eqref{def-Tlz},
which plays an important role in the perturbation
estimates for the resolvent of the higher order Schr\"odinger operator
$\cl$. We obtain four kinds of boundedness of the $T$-operator
by following the ideas used in \cite[Chapter 6]{Sch86}.

\subsection{The first and the second boundedness \label{s3.1}}

Recall the following basic properties of the Bessel potential function $G_{s,\dz}$
from \cite{Ste70,Sch86}.

\begin{lemma}\label{lem-s3-Besselp}
Let $s$, $\dz\in (0,\fz)$ and $G_{s,\dz}$ be as in \eqref{def-Besselp}.
The following assertions hold true:
\begin{enumerate}
\item[{\rm (i)}] $G_{s,\dz}(x)=\dz^{n-s}G_s(\dz x)$ for any $x\in \rn$.

\item[{\rm(ii)}] $G_s(x)\le C_{(s)} w_s(x)$ for any $x\in \rn$, where $w_s$ is as in \eqref{def-waz}
and the positive constant $C_{(s)}$ depends only on $s$ and $n$.
The converse inequality also holds true when $x\in \rn$ satisfies $|x|<1$.

\item[{\rm (iii)}] For any given $a\in (0,1)$ and $b\in (0,\fz)$, there exists a positive
constant $C_{(s,a,b)}$,
depending on $s$, $a$, $b$, and $n$, such that, for any $|x|>b$,
\begin{align*}
G_s(x)\le C_{(s,a,b)}e^{-a|x|}.
\end{align*}

\item[{\rm (iv)}] For any given $s\ne n<sq'$ and any $x\in\rn$,
\begin{align*}
[G_s(x)]^q\le C_{(s,q)} G_{q(s-n)+n}(x),
\end{align*}
where the positive constant $C_{(s,q)}$ depends only on $s$, $q$, and $n$.
\end{enumerate}
\end{lemma}

The following lemma is useful to the first boundedness of the $T$-operator.

\begin{lemma}\label{lem-s3-Gslz}
Let $p\in (0,\fz), q\in [1,\,\fz)$, $t\in [1,\,\fz]$, $s\in (0,\fz)$, and $\az\in (0,n)$.
Then there exists a positive constant $C$ such that, for any $\dz\in (0,\fz)$ and
  measurable function $V$ on $\rn$,
\begin{align*}
\lf\|\lf\{\dint_{|y-\cdot|\ge 1/\dz}
\lf|V(y)\r|^q [G_{s,\dz}(\cdot-y)]^p\,dy\r\}^{\frac{1}{q}}\r\|_{L^t(\rn)}
\le C\dz^{[(n-s)p+\az-n]/q} M_{\az,q,t,1/\dz}(V).
\end{align*}
\end{lemma}

\begin{proof}
For any given $\dz\in (0,\fz)$ and $k\in\nn$,
let $S_{k,\dz}:=\{x\in\rn:\ k/\dz\le |x|< (k+1)/\dz\}$
be the annulus with center at the origin,
$Q_k$ the inscribed cube of $B(\vec 0_n,k/\dz)$,
and $Q_{k+1}$ the externally tangent cube
of $B(\vec 0_n,(k+1)/\dz)$. Let $Q_0$ be the inscribed cube of
$B(\vec 0_n,1/\dz)$ (see Figure \ref{covering} as below).
\begin{figure}[ht]
    \centering
    \includegraphics[width=0.6\textwidth]{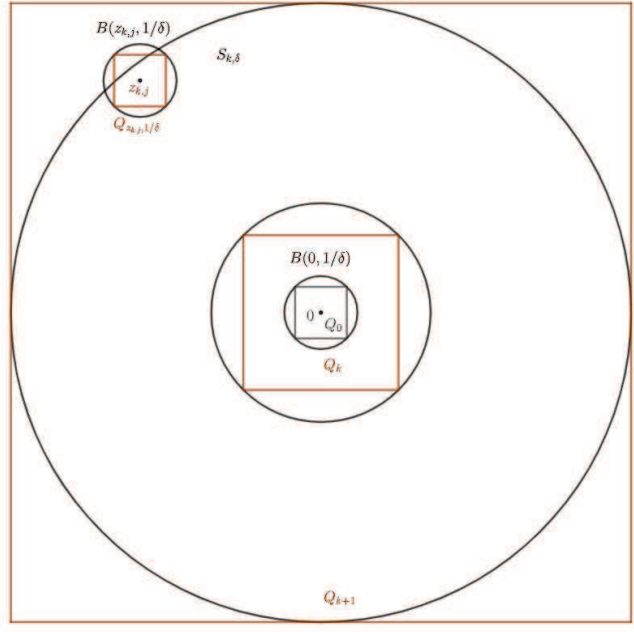}
    \caption{The covering of the annulus $S_{k,\dz}$}
    \label{covering}
\end{figure}

Since the volume
$ |Q_{k+1}\setminus Q_k|=(\frac{2}{\dz})^n[(k+1)^n-(\frac{k}{\sqrt n})^n]$ and
$  |Q_0|=(\frac{2}{\dz})^n(\frac{1}{\sqrt n})^n$, we know that
$S_{k,\dz}$ can be covered by $N(k)\le c k^{n}$ balls
$\{B(z_{k,j},1/\dz)\}_{j=1}^{N(k)}$  which consist of translations of
$B(\vec 0_n,1/\dz)$ with the implicit constant $c$ depending only on $n$,
that is,
\begin{align}\label{eqn-covering}
S_{k,\dz}\subset \bigcup_{j=1}^{N(k)}B(z_{k,j},1/\dz).
\end{align}
Note that, if $y-x\in B(z_{k,j},1/\dz)$, we then have $y\in B(z_{k,j}+x,1/\dz)$.

Now, we write by the assumption $q\in [1,\fz)$ that
\begin{align*}
&\lf\|\lf\{\dint_{|y-\cdot|\ge 1/\dz}|V(y)|^q
[G_{s,\dz}(\cdot-y)]^p\,dy\r\}^{\frac{1}{q}}\r\|_{L^t(\rn)}\\
&\hs\le\dsum_{k=1}^\infty \lf\|\lf\{\dint_{k/\dz\le|\cdot-y|<(k+1)/\dz}
|V(y)|^q[G_{s,\dz}(\cdot-y)]^p\,dy\r\}^{\frac{1}{q}}\r\|_{L^t(\rn)}\\
&\hs=:\dsum_{k=1}^\infty\mathrm{I}_k.
\end{align*}

Since $\dz| x-y|\ge k\ge 1$, applying (i) and (iii) of Lemma \ref{lem-s3-Besselp},
we have
\begin{align*}
[G_{s,\dz}(x-y)]^p = \dz^{(n-s)p} [G_s(\dz(x-y))]^p\ls \dz^{(n-s)p} e^{-akp}.
\end{align*}
This, together with \eqref{eqn-covering}, \eqref{def-Mazqrlz}, and the assumption $\az\in (0,n)$,
implies that
\begin{align*}
\dsum_{k=1}^\infty\mathrm{I}_k&
\ls \dsum_{k=1}^\infty\dz^{(n-s)p/q}e^{-akp/q}\lf\|\lf[\dsum_{j=1}^{N(k)} \dint_{B(z_{k,j}+\cdot,1/\dz)}|V(y)|^q\,dy\r]^{1/q}\r\|_{L^t(\rn)}\\
&\ls\dsum_{k=1}^\infty\dz^{(n-s)p/q} e^{-akp/q}\dsum_{j=1}^{N(k)}\lf\|\lf[ \dint_{B(z_{k,j}+\cdot,1/\dz)}|V(y)|^q\,dy\r]^{1/q}\r\|_{L^t(\rn)}\\
&\ls \dsum_{k=1}^\infty\dz^{(n-s)p/q} e^{-akp/q}N(k) \lf\|\lf[ \dint_{|y-\cdot|<1/\dz}|V(y)|^q|\dz(\cdot-y)|^{\az-n}\,dy\r]^{1/q}\r\|_{L^t(\rn)}\\
&\sim \dz^{(n-s)p/q+(\az-n)/q} M_{\az,q,t,1/\dz}(V),
\end{align*}
which completes the proof of Lemma \ref{lem-s3-Gslz}.
\end{proof}

The following proposition establishes the first boundedness of the  $T$-operator from $L^p(\rn)$
to $L^q(\rn)$.

\begin{proposition}\label{prop-bdT1}
Let $p\in [1,\,\fz]$, $s\in (0,n)$, $q\in [1,\,n/(n-s))$, $\az\in (0,(s-n)q+n]$, and
$\dz\in (0,\fz)$. Assume that $V$ is a measurable function on $\rn$ satisfying
$M_{\az,q,p',1/\dz}(V)<\fz$.  Then the operator $T_{s,\dz}$ can be extended to a bounded
linear operator
from $L^p(\rn)$ to $L^q(\rn)$. Moreover, there exists a positive constant $C$, independent of $\dz$,
such that, for any $f\in L^p(\rn)$,
\begin{align}\label{BC1}
\lf\|T_{s,\dz}(f)\r\|_{L^q(\rn)}\le C \dz^{[(n-s)q+\az-n]/q} M_{\az,q,p',1/\dz}(V)\|f\|_{L^p(\rn)}.
\end{align}
\end{proposition}

\begin{proof}
For any $f\in \mathcal{S}(\rn)$, by \eqref{def-Tlz} and the Minkowski inequality, we have
\begin{align*}
\lf\|T_{s,\dz}(f)\r\|_{L^q(\rn)}&=\lf\{\dint_{\rn} \lf|\dint_{\rn}
V(x)G_{s,\dz}(x-y)f(y)\,dy\r|^q\,dx\r\}^{1/q}\\
&\le \dint_{\rn} \lf[\dint_{\rn} |V(x)|^q \{G_{s,\dz}(x-y)\}^q\,dx\r]^{1/q}|f(y)|\,dy\\
&\le \lf\{\dint_{\rn} \lf[\dint_{\rn} |V(x)|^q
\{G_{s,\dz}(x-y)\}^q\,dx\r]^{p'/q}\,dy\r\}^{1/p'}\|f\|_{L^p(\rn)}\\
&=:C_{s,q,p',\dz}(V)\|f\|_{L^p(\rn)}.
\end{align*}
To estimate $C_{s,q,p',\dz}(V)$, we write
\begin{align*}
C_{s,q,p',\dz}(V)&\ls\lf\{\dint_{\rn} \lf[\dint_{|x-y|<1/\dz}
|V(x)|^q \{G_{s,\dz}(x-y)\}^q\,dx\r]^{p'/q}dy\r\}^{1/p'}\\
&\qquad+\lf\{\dint_{\rn} \lf[\dint_{|x-y|\ge 1/\dz}
\ldots dx\r]^{p'/q}dy\r\}^{1/p'}\\
&=:\mathrm{I}_1+\mathrm{I}_2.
\end{align*}
For $\mathrm{I}_2$, using Lemma \ref{lem-s3-Gslz}, we obtain
\begin{align*}
\mathrm{I}_2\ls \dz^{[(n-s)q+\az-n]/q} M_{\az,q,p',1/\dz}(V).
\end{align*}
To bound $\mathrm{I}_1$, for any $x$,   $y\in \rn$ satisfying $|x-y|<1/\dz$,
by (i) and (ii) of Lemma \ref{lem-s3-Besselp} and the assumptions that
$s\in (0,n)$, $q\in [1,\,n/(n-s))$, and $\az\in (0,(s-n)q+n]$,
we find that
\begin{align*}
\lf|G_{s,\dz}(x-y)\r|^q\le \dz^{(n-s)q}[w_s(\dz(x-y))]^q\le \dz^{(n-s)q+\az-n}w_\az(x-y),
\end{align*}
where the last inequality follows from \eqref{def-waz},
$(s-n)q\ge \az-n$, and $\dz|x-y|<1$. Thus, we have
\begin{align*}
\mathrm{I}_1\ls \dz^{(n-s)+(\az-n)/q}\lf\|\dint_{|x-\cdot|<1/\dz}
|V(x)|^qw_\az(x-\cdot)\,dx\r\|_{L^{p'}(\rn)}\sim \dz^{[(n-s)q+\az-n]/q} M_{\az,q,p',1/\dz}(V).
\end{align*}
Altogether, we conclude that $C_{s,q,p',\dz}(V)\ls \dz^{[(n-s)q+\az-n]/q} M_{\az,q,p',1/\dz}(V)$,
which completes the proof of Proposition \ref{prop-bdT1}.
\end{proof}

To show the second boundedness of the $T$\!-operator, we need the following
Gagliardo--Nirenberg inequality from \cite[Corollary 2.4]{HMOW11}.

\begin{lemma}\label{lem-GNI}
Let $p, \,p_0,\,p_1\in (1,\fz)$, $s,\,s_1\in (0,\fz)$, and $\tz\in [0,1]$. Then
there exists a positive constant $C$ such that,
for any $f\in \mathcal{S}(\rn)$,
\begin{align*}
\lf\|\Delta^{s/2}f\r\|_{L^p(\rn)}\le C \|f\|_{L^{p_0}(\rn)}^{1-\tz}\lf\|\Delta^{s_1/2} f\r\|_{L^{p_1}(\rn)}^{\tz}
\end{align*}
if and only if $  \frac{n}{p}-s=(1-\tz)\frac{n}{p_0}+\tz (\frac{n}{p_1}-s_1)$
and $s\le \tz s_1$.
\end{lemma}

The following result gives the $L^p(\rn)$-boundedness of the resolvent of the Laplace
operator $\Delta$.

\begin{lemma}\label{lem-resolvent-Laplace}
Let $p\in (1,\fz)$ and $\dz,\,\az\in (0,\fz)$. Then there exists a positive constant
$C$, independent of $\dz$, such that, for any $f\in L^p(\rn)$,
\begin{itemize}
\item [{\rm(i)}] $ \|(\dz^2-\Delta)^{-\az}f\|_{L^p(\rn)}\le  \frac{C}{\dz^{2\az}}\|f\|_{L^p(\rn)}$,

\item [{\rm (ii)}] $\|\Delta^{\az}(\dz^2-\Delta)^{-\az}f\|_{L^p(\rn)}\le C \|f\|_{L^p(\rn)}$.
\end{itemize}
\end{lemma}

\begin{proof}
Let $  m(\xi):=\frac{1}{(1+|\xi|^2)^{\az}}$ for any  $\xi\in\rn$ and $T_m$ be
the Fourier multiplier associated with $m$. By the H\"ormander--Mihlin
multiplier theorem (see \cite[Theorem 5.2.7]{Graf14}), we
conclude that $T_m$ is bounded on $L^p(\rn)$, which, combined with
the dilation invariant property of the $L^p(\rn)$ multiplier, shows that
(i) holds true. Using an argument similar to that used in the proof of (i) and
letting $  \wz m(\xi)=\frac{|\xi|^{2\az}}{(1+|\xi|^2)^{\az}}$ for any
$\xi\in\rn$, we know
that (ii) also holds true. This finishes the proof of
Lemma \ref{lem-resolvent-Laplace}.
\end{proof}

Based on Lemmas \ref{lem-GNI} and \ref{lem-resolvent-Laplace}, we
now establish the second boundedness of the $T$-operator as follows.

\begin{proposition}\label{prop-bdT2}
Let $p,\,q\in (1,\fz)$, $t$, $\sz\in [1,\fz]$, and $s\in(0,n)$ satisfy
\begin{align*}
\frac{1}{q}=\frac{1}{t}+\frac{1}{\sz} \ \ \text{and}\ \ \frac{1}{\sz}\le
\frac{1}{p}\le\frac{s}{n}+\frac{1}{\sz}.
\end{align*}
Assume that $V\in L^t(\rn)$. Then $T$ can be extended to a bounded linear
operator from $L^p(\rn)$ to $L^q(\rn)$.
Moreover, there exists a positive constant $C$
such that, for any $\dz\in (0,\fz)$ and $f\in L^p(\rn)$,
\begin{align}\label{BC2}
\lf\|T_{s,\dz}(f)\r\|_{L^q(\rn)}\le C \dz^{-s(1-\tz)}\|V\|_{L^t(\rn)}\|f\|_{L^p(\rn)}
\end{align}
with $  \tz:=\frac{n}{s}\Big(\frac{1}{p}-\frac{1}{\sz}\Big)\in [0,1]$.
\end{proposition}

\begin{proof}
By $\frac{1}{q}=\frac{1}{t}+\frac{1}{\sz}$,
\eqref{def-Tlz}, \eqref{def-Besselp2}, Lemmas \ref{lem-GNI} and \ref{lem-resolvent-Laplace},
we conclude that, for any $f\in \mathcal{S}(\rn)$,
\begin{align*}
\lf\|T_{s,\dz} (f)\r\|_{L^q(\rn)}&\le \|V\|_{L^t(\rn)}\lf\|\lf(\dz^2-\Delta\r)^{-s/2}f\r\|_{L^\sz(\rn)}\\
&\ls \|V\|_{L^t(\rn)} \lf\|\lf(\dz^2-\Delta\r)^{-s/2}f\r\|_{L^p(\rn)}^{1-\tz}\lf\|\Delta^{s/2}
\lf(\dz^2-\Delta\r)^{-s/2}f\r\|_{L^p(\rn)}^{\tz}\\
&\ls \|V\|_{L^t(\rn)} \dz^{-s(1-\tz)}\|f\|_{L^p(\rn)},
\end{align*}
which completes the proof of Proposition \ref{prop-bdT2}.
\end{proof}

We end this subsection by giving another embedding property of the
generalized Schechter class, as a byproduct of Lemma \ref{lem-s3-Gslz}.

\begin{proposition}\label{prop-EBP2}
Let $\az\in (0,n)$, $r\in [1,\fz)$, $t\in [1,\fz)$, and $S\in\rr$.
If $  1\le r\le t<\tau<\fz$ and $ \frac{\az}{nr}+\frac{1}{t}\le
\frac{\bz}{nr}+\frac{1}{\tau}$, then
\begin{align*}
M_{\az,r,t,S+S_1}(\rn)\subset M_{\bz,r,\tau,S}(\rn)
\end{align*}
with $ S_1:=\frac{1}{r}[nr(\frac{1}{\tau}-\frac{1}{t})+\bz-\az]$.
\end{proposition}

\begin{proof}
Let $s:=\bz-\az$. Since  $t<\tau$ and $  \frac{\az}{nr}+\frac{1}{t}\le
\frac{\bz}{nr}+\frac{1}{\tau}$, we know that $s>0$ and
$  \frac{r}{t}-\frac{r}{\tau}\le \frac{s}{n}$. Thus, let
$  p=\frac{t}{r}$ and $  \sz=\frac{\tau}{\tau-r}$. It is easy to see
$  \sz'=\frac{\tau}{r}>1$ and
\begin{align}\label{eqn-pp1}
\frac{1}{p}+\frac{1}{\sz}\le 1+\frac{s}{n}.
\end{align}
Now, using the Gaussian upper bound for the heat kernel of the Laplacian
$\Delta$,
we obtain $$\lf\|e^{t\Delta}\r\|_{L^p(\rn)\to L^{\sz'}(\rn)}\ls t^{-\frac{n}{2}(\frac{1}{p}-\frac{1}{\sz'})}.$$
This, together with \eqref{eqn-pp1} and the formula
\begin{align*}
\lf(\dz^2-\Delta\r)^{-s/2}=\frac{1}{\Gamma(s/2)}
\dint_0^\fz t^{s/2-1}e^{-\dz^2 t}e^{t\Delta}\,dt,
\end{align*}
shows that $(\dz^2-\Delta)^{-s/2}$ is bounded from $L^p(\rn)$
to $L^{\sz'}(\rn)$ with
\begin{align}\label{eqn-lplsz}
\lf\|\lf(\dz^2-\Delta\r)^{-s/2}\r\|_{L^p(\rn)\to L^{\sz'}(\rn)}\ls \dz^{-[n(\frac{1}{\sz'}-\frac{1}{p})+s]}
\sim\dz^{-S_1 r}.
\end{align}
Moreover, let $G_{s,\dz}$ be the Bessel potential function as in \eqref{def-Besselp2}.
For any $|x-y|<1/\dz$, by (i) and (ii) of Lemma \ref{lem-s3-Besselp} and \eqref{def-waz}, it is easy
to see
\begin{align*}
|x-y|^{\bz-n}=w_\bz(\dz(x-y))\dz^{n-\bz}\ls G_\bz(\dz(x-y))\dz^{n-\bz}
\sim G_{\bz,\dz}(x-y).
\end{align*}
Combining the above arguments with \eqref{eqn-lplsz}, we obtain
\begin{align}\label{eqn-tem1}
\dz^{-Q}M_{\bz,r,\tau,1/\dz}(V)&=\dz^{-Q} \lf\|\lf[\dint_{|x-\cdot|<1/\dz}
\lf|V(x)\r|^r|x-\cdot|^{\bz-n}\,dx\r]^{1/r}\r\|_{L^{\tau}(\rn)}\\ \notag
&\ls \dz^{-Q}\lf\|G_{\bz,\dz}\ast V^r\r\|_{L^{\tau/r}(\rn)}^{1/r}
\ls \dz^{-Q}\lf\|G_{s,\dz}\ast\lf(G_{\az,\dz}\ast V^r\r)\r\|_{L^{\tau/r}(\rn)}^{1/r}\\ \notag
&\ls \dz^{-Q}\lf\|(\dz^2-\Delta)^{-s/2}\lf(G_{\az,\dz}\ast V^r\r)\r\|_{L^{\sz'}(\rn)}^{1/r}\\  \notag
&\ls \dz^{-Q-S_1}\lf\|G_{\az,\dz}\ast V^r\r\|_{L^{p}(\rn)}^{1/r}.
\end{align}
Now, using the assumption $  p=\frac{t}{r}$, we write
\begin{align}\label{eqn-tem2}
\lf\|G_{\az,\dz}\ast V^r\r\|_{L^{p}(\rn)}^{1/r}
&=\lf\{\dint_\rn\lf[\dint_{\rn}
\lf|V(y)\r|^r|y-x|^{\az-n}\,dy\r]^{t/r}dx\r\}^{1/t}\\ \notag &\ls
\lf\{\dint_\rn\lf[\dint_{|y-x|<1/\dz}
\lf|V(y)\r|^r|y-x|^{\az-n}\,dy\r]^{t/r}dx\r\}^{1/t}\\ \notag
&\qquad+\lf\{\dint_\rn\lf[\dint_{|y-x|\ge
1/\dz} \cdots dy\r]^{t/r}dx\r\}^{1/t}\\ \notag
&=:\mathrm{I}_1+\mathrm{I}_2.
\end{align}
By \eqref{def-Mazqrlz}, it is easy to see that $\mathrm{I}_1\sim
M_{\az,r,t,1/\dz}(V)$. For $\mathrm{I}_2$, using Lemma
\ref{lem-s3-Gslz} with $p=1$ and $s=\az$ therein, we also obtain
$\mathrm{I}_2\sim M_{\az,r,t,1/\dz}(V)$. This, combined with
\eqref{eqn-tem1} and \eqref{eqn-tem2}, implies that
\begin{align*}
\dz^{-Q}M_{\bz,r,\tau,1/\dz}(V)\ls \dz^{-Q-S_1}M_{\az,r,t,1/\dz}(V),
\end{align*}
which, together with \eqref{eqn-Mclass}, then completes the proof of
Proposition \ref{prop-EBP2}.
\end{proof}

\subsection{The third and the fourth boundedness \label{s3.2}}

In this subsection, we consider the third and the fourth
boundedness of the $T$-operator. First, we need some additional notation.
To be precise, for any given $s\in (0,n/2)$, $r\in [1,\fz)$, $t\in [1,\fz]$,
and $\dz\in (0,\fz)$, let
\begin{align}\label{eqn-wzC}
\wz C_{s,r,t,\dz}(V):=\lf\| \lf\{\dint_{\rn} \lf|V(x)\r|^r [G_{2s,\dz}(x-\cdot)]^{r/2}\,dx\r\}^{1/r}\r\|_{L^t(\rn)}
\end{align}
with $G_{2s,\dz}$ as in \eqref{def-Besselp}.

\begin{lemma}\label{lem-estimatewzC}
Let $t\in [1,\fz)$,
$\dz\in (0,\fz)$, $s\in (0,n/2)$, $r\in [1,\frac{2n}{n-2s})$, and $\az\in (0,n+(2s-n)r/2]$.
Assume that $V$ is a   measurable function on $\rn$. Then
\begin{align*}
\wz C_{s,r,t,\dz}(V)\ls \dz^{[(n-2s)r/2+(\az-n)]/r}M_{\az,r,t,1/\dz}(V)
\end{align*}
with the implicit positive constant independent of $\dz$ and $V$.
\end{lemma}

\begin{proof}
The proof of this lemma is similar to that of
\eqref{eqn-tem2}. By \eqref{eqn-wzC}, we first write
\begin{align*}
\wz C_{s,r,t,\dz}(V)&\ls \lf\|\lf[\dint_{|x-\cdot|<1/\dz}\lf|V(x)\r|^r
\{G_{2s,\dz}(x-\cdot)\}^{r/2}\,dx\r]^{1/r}\r\|_{L^t(\rn)} +\lf\|\lf[\dint_{|x-\cdot|\ge 1/\dz}  \cdots
dx\r]^{1/r}\r\|_{L^t(\rn)}\\
&=:\mathrm{I}_1+\mathrm{I}_2.
\end{align*}
For $\mathrm{I}_2$,   applying Lemma \ref{lem-s3-Gslz} with $q=r$,
$p=r/2$, and $s=2s$ therein, we have
\begin{align*}
\mathrm{I}_2\sim\lf\|\lf\{\dint_{|x-\cdot|\ge 1/\dz} \lf|V(x)\r|^r [G_{2s,\dz}(x-\cdot)]^{r/2}\,dx\r\}^{1/r}\r\|_{L^t(\rn)}
\ls \dz^{[(n-2s)r/2+(\az-n)]/r}M_{\az,r,t,1/\dz}(V).
\end{align*}
For $\mathrm{I}_1$, since $\dz|x-y|<1$, we deduce from (i) and (ii) of Lemma \ref{lem-s3-Besselp},
\eqref{def-waz}, and the assumption $2s<n$, that
\begin{align*}
\lf[G_{2s,\dz}(x-y)\r]^{r/2}\ls \dz^{(n-2s)r/2}[w_{2s}(\dz(x-y))]^{r/2}
\sim  |x-y|^{(2s-n)r/2}.
\end{align*}
This, combined with the assumption $\az\le n+(2s-n)r/2$, shows that
\begin{align*}
\lf[G_{2s,\dz}(x-y)\r]^{r/2}\ls \dz^{(n-2s)r/2} |\dz(x-y)|^{\az-n}
\sim \dz^{(n-2s)r/2+(\az-n)}w_\az(x-y),
\end{align*}
which then implies that
\begin{align*}
\mathrm{I}_1&\ls \dz^{[(n-2s)r/2+\az-n]/r}
\lf\|\lf[\dint_{|x-y|< 1/\dz} \lf|V(x)\r|^r w_\az(x-y)\,dx\r]^{1/r}\r\|_{L^t(\rn)}\\
&\sim \dz^{[(n-2s)r/2+\az-n]/r} M_{\az,r,t,1/\dz}(V).
\end{align*}
Altogether the estimates for $\mathrm{I}_1$ and $\mathrm{I}_2$, we conclude that
\begin{align*}
\wz C_{s,r,t,\dz}(V)\ls \dz^{[(n-2s)r/2+\az-n]/r} M_{\az,r,t,1/\dz}(V),
\end{align*}
which completes the proof of Lemma \ref{lem-estimatewzC}.
\end{proof}

Based on Lemma \ref{lem-estimatewzC}, we now state the third boundedness of the $T$-operator.

\begin{proposition}\label{prop-bdT3}
Let $s\in (0,n/2)$, $q\in [2,\fz)$,
$t\in [q,\fz]$, and $r\in [q,\frac{2n}{n-2s})$ satisfy
\begin{align*}
\frac{1}{t}+\frac{1}{r}=\frac{1}{q}.
\end{align*}
Assume that $V$ is a   measurable function on $\rn$ satisfying
$M_{\az,r,t,1/\dz}(V)<\fz$ for some  $\dz\in (0,\fz)$
and  $\az\in (0,n+(2s-n)r/2]$.
Then $T_{s,\dz}$ can be extended to
a bounded linear operator from $L^2(\rn)$ to $L^q(\rn)$. Moreover, there exists a positive
constant $C$, independent of $\dz$, such that, for any $f\in L^2(\rn)$,
\begin{align}\label{BC3}
\lf\|T_{s,\dz}(f)\r\|_{L^q(\rn)}\le C\dz^{[(n-2s)r/2+\az-n]/r} M_{\az,r,t,1/\dz}(V)\|f\|_{L^2(\rn)}.
\end{align}
\end{proposition}

\begin{proof}
Let  $\wz C_{s,r,t,\dz}(V)$ be as in \eqref{eqn-wzC}. From \cite[p. 129, Theorem 6.3]{Sch86},
it follows that
$$\lf\|T_{s,\dz}(f)\r\|_{L^q(\rn)}\ls\wz C_{s,r,t,\dz}(V)\|f\|_{L^2(\rn)},$$
which, together with Lemma \ref{lem-estimatewzC}, then completes the proof
of Proposition \ref{prop-bdT3}.
\end{proof}

The following proposition gives the fourth boundedness of the $T$-operator.

\begin{proposition}\label{prop-bdT4}
Let $s\in (0,n)$, $\dz\in (0,\fz)$, and $T_{s,\dz}$ be as in
\eqref{def-Tlz}.
\begin{itemize}
\item[{\rm (i)}]
For any given $p$, $q\in [1,\fz)$, let $r\in [q,\fz)$, $t\in [1,\fz)$, $\az\in (0,(1-\tz_1)r(s-n)+n]$,
and $\wz\az\in (0,\tz_1r'(s-n)+n]$ for some
$  \tz_1\in (0,\min\{1,\frac{n}{r'(n-s)}\})$.
If
\begin{align}\label{eqn-parameters-1}
\frac{1}{q}<\frac{1}{p}+\frac{1}{t}<\min\lf\{1,1+\tz_1\lf(\frac{s}{n}-1\r)
+\frac{1}{q}-\frac{1}{r}\r\}
\end{align}
and  $M_{\az,r,t,1/\dz}(V)<\fz$,
then $T_{s,\dz}$ can be extended to
a bounded linear operator from $L^p(\rn)$ to $L^q(\rn)$. Moreover, there exists a positive
constant $C$, independent of $\dz$ and $V$, such that, for any $f\in L^p(\rn)$,
\begin{align*}
\lf\|T_{s,\dz}(f)\r\|_{L^q(\rn)}\le C \dz^{(n-s)+(\az-n)/r+{(\wz\az-n)}/{r'}} M_{\az,r,t,1/\dz}(V)\|f\|_{L^p(\rn)}.
\end{align*}

\item[{\rm (ii)}] Let $r\in (1,2]$ and $\az\in (0,n)$ satisfy
\begin{align*}
\az-n\le r(s-n)+\frac{nr}{r'} \ \ \text{and} \ \ 2n>r'(n-s).
\end{align*}
Then $T_{s,\dz}$ can be
extended to a bounded linear operator on $L^r(\rn)$. Moreover, there exists a positive
constant $C$, independent of $\dz$ and $V$, such that, for any $f\in L^r(\rn)$,
\begin{align}\label{BC4}
\lf\|T_{s,\dz}(f)\r\|_{L^r(\rn)}\le C \dz^{\az/r-s} M_{\az,r,\fz,1/\dz}(V)\|f\|_{L^r(\rn)}.
\end{align}
\end{itemize}
\end{proposition}

To prove Proposition \ref{prop-bdT4}, we first need to study two
terms $h_{s,\delta}$ and $W_{s,\delta}$. To be precise, for any given
$s\in (0,n)$, $\dz\in (0,\fz)$, $r\in [1,\fz)$, and ${\tz_1}\in [0,1]$,
and any $y\in\rn$, let
\begin{align}\label{eqn-def-hlz}
h_{s,\dz}(y):=\lf\{\dint_\rn \lf|V(x)\r|^r [G_{s,\dz}(x-y)]^{(1-{\tz_1})r}\,dx\r\}^{1/r}.
\end{align}
Similarly to the proof of Lemma \ref{lem-estimatewzC}, we know that,
for any given $t\in [1,\fz]$, and $\az\in (0,n)$ satisfying $\az-n\le (1-{\tz_1})r(s-n)$,
\begin{align}\label{eqn-est-hslz}
\|h_{s,\dz}\|_{L^t(\rn)}\ls \dz^{(n-s)(1-{\tz_1})+(\az-n)/r}M_{\az,r,t,1/\dz}(V),
\end{align}
where the implicit positive constant is independent of $\dz$ and $V$.

On the other hand, since $1\le q\le r<\fz$ and
$  \frac{1}{q}<\frac{1}{p}+\frac{1}{t}<1$, it follows that
\begin{align*}
0<\frac{1}{p}+\frac{1}{t}-\frac{1}{q}\le \frac{1}{p}+\frac{1}{t}-\frac{1}{r}< 1-\frac{1}{r}=\frac{1}{r'}.
\end{align*}
Moreover, let
\begin{align}\label{eqn-parameters}
\begin{cases}
\dfrac{1}{\rho_1}:=\dfrac{r'}{q'},\\
\dfrac{1}{\rho_2}:=r'\lf(\dfrac{1}{p}+\dfrac{1}{t}-\dfrac{1}{r}\r),\\
\nu:=t\lf(\dfrac{1}{p}+\dfrac{1}{t}\r),\\
\tz_2:=\dfrac{p}{r\nu'},\\
\gz:=1-\dfrac{t}{r\nu}.
\end{cases}
\end{align}
Clearly,  ${\rho_1}\in [1,\fz)$ and ${\rho_2},\,\nu\in (1,\fz)$. Since
$  \nu'=1+\frac{p}{t}$, we conclude that  $ \tz_2,\,\gz\in (0,1)$.

Now, let $h_{s,\dz}$ be as in \eqref{eqn-def-hlz}. For any $f$, $g\in \mathcal{S}(\rn)$, define
\begin{align}\label{eqn-def-Wslz}
W_{s,\dz}:=\lf\{\iint_{\rn\times \rn} [G_{s,\dz}(x-y)]^{\tz_1 r'} |g(x)|^{r'} |f(y)|^{(1-\tz_2)r'}
[h_{s,\dz}(y)]^{\gz r'}\,dx\,dy\r\}^{1/r'}.
\end{align}
The following lemma provides an estimate for $W_{s,\dz}$.

\begin{lemma}\label{lem-EstW}
Let $s$, $r$, $\rho_1$, $\rho_2$, $\tz_1$, $\tz_2$, $\wz\az$, and $\gz$ be as in
\eqref{eqn-parameters-1} and \eqref{eqn-parameters}.
If
\begin{align*}
\frac{1}{\rho_1}+\frac{1}{\rho_2}=\frac{\wz\az}{n}+1,
\end{align*}
then
\begin{align*}
W_{s,\dz}\ls \dz^{(n-s)\tz_1 +(\wz\az-n)/r'}\lf\||g|^{r'}\r\|^{1/r'}_{L^{\rho_1}(\rn)}
\lf\|\lf|f\r|^{(1-\tz_2)r'}h_{s,\dz}^{\gz r'}\r\|^{1/r'}_{L^{\rho_2}(\rn)}
\end{align*}
with the implicit positive constant independent of $f$, $g$, and $\dz$.
\end{lemma}

\begin{proof}
By \eqref{eqn-def-Wslz}, write
\begin{align}\label{eqn-dcW}
W_{s,\dz}^{r'}&=\dint_{\rn}\lf(\dint_{|x-y|<1/\dz}+\dint_{|x-y|\ge
1/\dz}\r)  [G_{s,\dz}(x-y)]^{\tz_1 r'} |g(x)|^{r'} |f(y)|^{(1-\tz_2)r'}
[h_{s,\dz}(y)]^{\gz r'}\,dx\,dy\\ \notag
&=:\mathrm{I}_1+\mathrm{I}_2.
\end{align}

For $\mathrm{I}_1$, since $0<\wz\az\le \tz_1 r' (s-n)+n<n$ and $\dz|x-y|<1$,
it follows, from (i) and (ii) of Lemma \ref{lem-s3-Besselp} and \eqref{def-waz}, that
\begin{align*}
\lf[G_{s,\dz}(x-y)\r]^{\tz_1 r'}&=\lf[\dz^{n-s} G_s(\dz(x-y))\r]^{\tz_1 r'}
\ls \dz^{(n-s) \tz_1 r'} [w_s(\dz(x-y))]^{\tz_1r'}\\
&\ls \dz^{(n-s)\tz_1 r'+(\wz\az-n)}\,|x-y|^{\wz\az-n}.
\end{align*}
This implies that
\begin{align*}
\mathrm{I}_1\ls \dz^{(n-s)\tz_1 r' +\wz\az -n}
\lf\{\iint_{\rn\times \rn} |x-y|^{\wz\az-n} |g(x)|^{r'} |f(y)|^{(1-\tz_2)r'}
[h_{s,\dz}(y)]^{\gz r'}\,dx\,dy\r\}.
\end{align*}
By the assumption
$  \frac{1}{\rho_1}+\frac{1}{\rho_2}=\frac{\wz\az}{n}+1$ and the
boundedness of the Riesz potential of order $\wz\az$ from $L^{\rho_1}(\rn)$
to $L^{\rho_2'}(\rn)$, we obtain
\begin{align*}
\mathrm{I}_1\ls \dz^{(n-s)\tz_1r' +(\wz\az-n)}\lf\||g|^{r'}\r\|_{L^{\rho_1}(\rn)}\lf\|\lf\|f\r|^{(1-\tz_2)r'}
h_{s,\dz}^{\gz r'}\r\|_{L^{\rho_2}(\rn)},
\end{align*}
which is the desired estimate.

For $\mathrm{I}_2$, by Lemma \ref{lem-s3-Besselp}(iii), we know that
$[G_{s} (\dz(x-y))]^{\tz_1 r'}\ls e^{-a \tz_1 r'\dz|x-y|}\ls
(\dz|x-y|)^{\wz\az-n}$ for any $\dz|x-y|>1$. From this, we further deduce that
$\mathrm{I}_2$ satisfies the same estimate as that of
$\mathrm{I}_1$, which, combined with \eqref{eqn-dcW}, then completes the
proof of Lemma \ref{lem-EstW}.
\end{proof}

Based on Lemma \ref{lem-EstW}, we now show Proposition
\ref{prop-bdT4}(i).

\begin{proof}[Proof of Proposition \ref{prop-bdT4}(i)]
From \eqref{eqn-def-hlz} and \eqref{eqn-def-Wslz}, we first deduce
that, for any $s\in (0,n)$, $\dz\in (0,\fz)$, and $f, \,g\in \mathcal{S}(\rn)$,
\begin{align}\label{eqn-bdT4-1}
\lf|(T_{s,\dz}(f),g)\r|&=\lf|(V(G_{s,\dz}\ast f),g)\r|\\ \notag
&\le \iint_{\rn\times \rn} |V(x)|G_{s,\dz}(x-y)
|f(y)||g(x)|\,dx\,dy\\ \notag
&=\iint_{\rn\times \rn}\lf[ |V(x)|[G_{s,\dz}(x-y)]^{1-\tz_1}
[h_{s,\dz}(y)]^{-\gz} |f(y)|^{\tz_2}\r]\\ \notag
&\hs\times\lf[[G_{s,\dz}(x-y)]^{\tz_1}|g(x)||f(y)|^{1-\tz_2}
[h_{s,\dz}(y)]^{\gz}\r]dx\,dy\\ \notag
&\le W_{s,\dz}  \lf\{\iint_{\rn\times \rn}|V(x)|^r [G_{s,\dz}(x-y)]^{(1-\tz_1)r}
[h_{s,\dz}(y)]^{-\gz r} |f(y)|^{\tz_2 r}dx\,dy\r\}^{1/r}\\ \notag
&= W_{s,\dz}  \lf\{\dint_{\rn} [h_{s,\dz}(y)]^{(1-\gz) r} |f(y)|^{\tz_2 r}\,dy\r\}^{1/r}\\ \notag
&\le  W_{s,\dz} \lf\|[h_{s,\dz}]^{1-\gz}\r\|_{L^{r\nu}(\rn)} \lf\||f|^{\tz_2}\r\|_{L^{r\nu'}(\rn)},
\end{align}
where $W_{s,\dz}$ is as in \eqref{eqn-def-Wslz}.
Moreover, by \eqref{eqn-parameters}, the assumption
$  \frac{1}{q}<\frac{1}{p}+\frac{1}{t}<\min\{1,1+\tz_1(\frac{s}{n}-1)
+\frac{1}{q}-\frac{1}{r}\}$, and an elementary calculation, we conclude that
\begin{align*}
\begin{cases}\displaystyle
\frac{1}{r'\nu'\rho_2}+\frac{1}{r\nu'}=\frac{1}{p},\\
\displaystyle \frac{1}{r'\nu \rho_2}+\frac{1}{r\nu}=\frac{1}{t},\\
\displaystyle 1<\frac{1}{\rho_1}+\frac{1}{\rho_2}\le 1+\frac{\tz_1 r' (s-n)+n}{n},\\
\displaystyle r'\rho_1=q', \ (1-\tz_2)r'\nu'\rho_2=p, \ \gz r'\nu \rho_2 =t,\
(1-\gz)r\nu=t,\ \tz_2r\nu'=p.
\end{cases}
\end{align*}
Now, applying \eqref{eqn-bdT4-1}, \eqref{eqn-est-hslz}, and Lemma
\ref{lem-EstW} with $\wz\az\le \tz_1 r' (s-n)+n$ therein, we obtain
\begin{align*}
&\lf|(VG_{s,\dz}\ast f,g)\r|\\
&\hs\le W_{s,\dz} \lf\|[h_{s,\dz}]^{1-\gz}\r\|_{L^{r\nu}(\rn)} \lf\||f|^{\tz_2}\r\|_{L^{r\nu'}(\rn)}\\
&\hs\ls \dz^{(n-s)\tz_1 +(\wz\az-n)/r'}\lf\||g|^{r'}\r\|^{1/r'}_{L^{\rho_1}(\rn)}\lf\|\lf|f\r|^{(1-\tz_2)r'}
h_{s,\dz}^{\gz r'}\r\|^{1/r'}_{L^{\rho_2}(\rn)}\lf\|[h_{s,\dz}]^{1-\gz}\r\|_{L^{r\nu}(\rn)} \lf\||f|^{\tz_2}\r\|_{L^{r\nu'}(\rn)}\\
&\hs\ls  \dz^{(n-s)\tz_1 +(\wz\az-n)/r'} \|g\|_{L^{q'}(\rn)}\lf\||f|^{1-\tz_2}\r\|_{L^{r'\nu'\rho_2}(\rn)}
\lf\|[h_{s,\dz}]^\gz\r\|_{L^{r'\nu\rho_2}(\rn)} \lf\|[h_{s,\dz}]^{1-\gz}\r\|_{L^{r\nu}(\rn)} \lf\||f|^{\tz_2}\r\|_{L^{r\nu'}(\rn)}\\
&\hs\sim \dz^{(n-s)\tz_1 +(\wz\az-n)/r'} \|g\|_{L^{q'}(\rn)} \|f\|_{L^p(\rn)}\|h_{s,\dz}\|_{L^{t}(\rn)}\\
&\hs\ls   \dz^{(n-s)+(\az-n)/r+{(\wz\az-n)}/{r'}}M_{\az,r,t,1/\dz}(V)\|g\|_{L^{q'}(\rn)} \|f\|_{L^p(\rn)},
\end{align*}
which completes the proof of Proposition \ref{prop-bdT4}(i).
\end{proof}

To prove Proposition \ref{prop-bdT4}(ii), for any given $s\in (0,n)$ and
$\dz\in (0,\fz)$, and, for any $g\in \mathcal{S}(\rn)$, define
\begin{align*}
\wz W_{s,\dz}:=\lf\{\iint_{\rn\times \rn} [G_{s,\dz}(x-y)]^{\tz_1 r'}
|g(x)|^{r'} [h_{s,\dz}(y)]^{r'}\,dx\,dy\r\}^{1/r'},
\end{align*}
where $G_{s,\dz}$ and $h_{s,\dz}$ are, respectively, as in \eqref{def-Besselp}
and \eqref{eqn-def-hlz}.
Similarly to Lemma \ref{lem-EstW}, we have the following estimate for
$\wz W_{s,\dz}$.

\begin{lemma}\label{lem-EstwzW}
Let $s\in (0,n)$, $\dz\in (0,\fz)$, $r\in (1,2]$, and $\az\in (0,n)$ satisfy
$\az-n<r(s-n)+\frac{nr}{r'}$ and $2n>r'(n-s)$. Then
\begin{align*}
\wz W_{s,\dz}\ls \dz^{\az/r-s} M_{\az,r,\fz,1/\dz}(V)
\|g\|_{L^{r'}(\rn)}
\end{align*}
with the implicit positive constant independent of $\dz$.
\end{lemma}

\begin{proof}
Write
\begin{align*}
\wz W_{s,\dz}&=\lf\{\dint_\rn \lf[\dint_\rn  [G_{s,\dz}(x-y)]^{\tz_1 r'}[h_{s,\dz}(y)]^{r'}\,dy\r]
|g(x)|^{r'}\,dx\r\}^{1/r'}\\
&=:\lf\{\dint_\rn Q_{s,\dz}(x) |g(x)|^{r'}\,dx\r\}^{1/r'}.
\end{align*}
By \eqref{eqn-def-hlz},  $r'/r\ge 1$, and the Minkowski integral inequality, we conclude that,
for any $x\in \rn$,
\begin{align*}
Q_{s,\dz}(x)&=\lf\{ \dint_{\rn} \lf[G_{s,\dz}(x-y)\r]^{\tz_1 r'}
\lf[\dint_\rn |V(z)|^r\lf\{G_{s,\dz}(z-y)\r\}^{(1-\tz_1)r}\,dz\r]^{r'/r}\,dy\r\}^{\frac{r}{r'}\cdot \frac{r'}{r}}\\
&\le \lf[\dint_{\rn} \lf\{\dint_\rn  [G_{s,\dz}(z-y)]^{(1-\tz_1)r'}
[G_{s,\dz}(x-y)]^{\tz_1r'}\,dy\r\}^{r/r'} |V(z)|^r\,dz\r]^{\frac{r'}{r}}\\
&\le \lf[\dint_{\rn} \lf\{[G_{s,\dz}]^{(1-\tz_1)r'}\ast[G_{s,\dz}]^{\tz_1r'}(x-z)\r\}^{r/r'} |V(z)|^r\,dz\r]^{\frac{r'}{r}}.
\end{align*}
Using Lemma \ref{lem-s3-Besselp}(i), we obtain
\begin{align*}
[G_{s,\dz}]^{(1-\tz_1)r'}\ast[G_{s,\dz}]^{\tz_1r'}
&=\dz^{(n-s)r'}[G_{s}(\dz\cdot)]^{(1-\tz_1)r'}\ast[G_{s}(\dz\cdot)]^{\tz_1r'}.
\end{align*}
Moreover, by the assumptions $s\in (0,n)$ and $2n>r'(n-s)$, we know that
there exists a $\tz_1\in (0,1)$ such that $\tz_1r'(s-n)+n>0$ and
$(1-\tz_1)r'(s-n)+n>0$, which, combined with Lemma
\ref{lem-s3-Besselp}(iv), implies that
\begin{align*}
[G_{s}(\dz\cdot)]^{(1-\tz_1)r'}\ast[G_{s}(\dz\cdot)]^{\tz_1r'}
=G_{(1-\tz_1)r'(s-n)+n}(\dz\cdot)\ast G_{\tz_1r'(s-n)+n}(\dz\cdot).
\end{align*}
Thus, by Lemma \ref{lem-s3-Besselp}(i) again, we further obtain
\begin{align*}
[G_{s,\dz}]^{(1-\tz_1)r'}\ast[G_{s,\dz}]^{\tz_1r'}
&=\lf[\dz^{(1-\tz_1)(n-s)r'}G_{(1-\tz_1)r'(s-n)+n}(\dz\cdot)\r]
\ast\lf[\dz^{\tz_1(n-s)r'}G_{\tz_1r'(s-n)+n}(\dz\cdot)\r]\\
&=G_{(1-\tz_1)r'(s-n)+n,\dz}\ast G_{\tz_1r'(s-n)+n,\dz}
=G_{r'(s-n)+2n,\dz}.
\end{align*}
This implies that
\begin{align*}
\lf\|Q_{s,\dz}\r\|_{L^\fz(\rn)}^{1/r'}\ls
\lf\|\lf\{ \dint_{\rn} \lf[G_{r'(s-n)+2n,\dz}(\cdot-z)\r]^{r/r'} |V(z)|^r\,dz\r\}^{\frac{1}{r}}\r\|_{L^\fz(\rn)}.
\end{align*}
Using the assumption $  \az-n<r(s-n)+\frac{nr}{r'}$ and following the estimation
of $\wz C_{s,r,t,\dz}(V)$
in the proof of Lemma \ref{lem-estimatewzC}, we obtain
\begin{align*}
\lf\|Q_{s,\dz}\r\|^{1/r'}_{L^\fz(\rn)}\ls \dz^{\az/r-s} M_{\az,r,\fz,1/\dz}(V),
\end{align*}
which implies that
\begin{align*}
\wz W_{s,\dz}&=\lf\{\dint_\rn Q_{s,\dz}(x) |g(x)|^{r'}\,dx\r\}^{1/r'}
\ls \dz^{\az/r-s} M_{\az,r,\fz,1/\dz}(V) \|g\|_{L^{r'}(\rn)}
\end{align*}
and hence completes the proof of Lemma \ref{lem-EstwzW}.
\end{proof}

Applying Lemma \ref{lem-EstwzW}, we now show
Proposition \ref{prop-bdT4}(ii).

\begin{proof}[Proof of Proposition \ref{prop-bdT4}(ii)]
For any given $s\in (0,n)$ and $\dz\in (0,\fz)$, and any $f,\, g\in \mathcal{S}(\rn)$,
by Lemma \ref{lem-EstwzW}, we write
\begin{align*}
&\lf|(T_{s,\dz}(f),g)\r|\\
&\hs=\lf|(VG_{s,\dz}\ast f,g)\r|\\
&\hs\le \iint_{\rn\times \rn} |V(x)|G_{s,\dz}(x-y)
|f(y)||g(x)|\,dx\,dy\\
&\hs=\iint_{\rn\times \rn}\lf\{|V(x)|[G_{s,\dz}(x-y)]^{1-\tz_1}
[h_{s,\dz}(y)]^{-1} |f(y)|\r\}\lf\{[G_{s,\dz}(x-y)]^{\tz_1}
[h_{s,\dz}(y)]|g(x)|\r\}dx\,dy\\
&\hs\le \wz W_{s,\dz}  \lf\{\iint_{\rn\times \rn}|V(x)|^r [G_{s,\dz}(x-y)]^{(1-\tz_1)r}
[h_{s,\dz}(y)]^{-r} |f(y)|^r\,dx\,dy\r\}^{1/r}
= \wz W_{s,\dz} \lf\|f\r\|_{L^{r}(\rn)}\\
&\hs\ls \dz^{\az/r-s} M_{\az,r,\fz,1/\dz}(V) \lf\|f\r\|_{L^{r}(\rn)}\|g\|_{L^{r'}(\rn)}.
\end{align*}
This finishes the proof of Proposition \ref{prop-bdT4}(ii).
\end{proof}

\section{Perturbation estimates for resolvents \label{s4}}

This section is devoted to the perturbation estimate for the
resolvent related to the Schr\"odinger operator $\cl=P(D)+V$, which is
essential to the estimate of the heat kernel. We consider two kinds
of perturbations: the summation and the exponential perturbations.

\subsection{Summation perturbation estimates \label{s4.1}}

We first need the following estimates for the resolvent of the homogeneous elliptic operator
$P(D)$. For the sake of homogeneity, we always use $(\lambda^{2m}-T)^{-1}$, where
$\lz\in\cc$,
to denote the resolvent of a $2m$-order differential operator $T$ in this section.

\begin{lemma}\label{lem-resolventforLc}
Let $m\in\nn$, $\lambda^{2m}\in \cc\setminus[0,\fz)$, $q\in (1,\fz)$,
and $P(D)$ be a homogeneous elliptic operator of order
$2m$ with real constant coefficients as in \eqref{def-PD}.
\begin{itemize}
\item[{\rm (i)}] Then there exists a positive constant
$C$, independent of $\lambda$, such that, for any $f\in L^q(\rn)$,
\begin{align*}
\lf\|\lf(\lambda^{2m}-P(D)\r)^{-1}f\r\|_{L^q(\rn)}\le \frac{C}{|\lambda|^{2m}}\|f\|_{L^q(\rn)}.
\end{align*}

\item[{\rm (ii)}] For any $s\in (0,2m]$ and $p\in [q,\fz)$
satisfying $  \frac{n}{2}(\frac{1}{q}-\frac{1}{p})<m-s/2$, there exists a positive constant
$C_2$, independent of $\lambda$, such that, for any $f\in L^q(\rn)$,
\begin{align}\label{eqn-ConstatC3}
\lf\| \lf(|\lambda|^2-\Delta\r)^{s/2}\lf(\lambda^{2m}-P(D)\r)^{-1}
f\r\|_{L^p(\rn)}\le C_2|\lz|^{-[2m-s+n(\frac{1}{p}-\frac{1}{q})]} \|f\|_{L^q(\rn)}.
\end{align}
\end{itemize}
\end{lemma}

\begin{proof}
We first prove (i). If $\lz^{2m}\in (-\fz,0)$, then let
$ m(\xi)=\frac{1}{1+P(\xi)}$ with $  P(\xi)=\sum_{|\az|=2m}a_\az
\xi^{\az}$ for any $\xi\in\rn$
as in \eqref{eqn-uec}. From an elementary calculation
and the uniform ellipticity condition \eqref{eqn-uec}, it follows that $m(\xi)$
satisfies the Mihlin condition
\begin{align*}
\lf|\pat_\xi^\az m(\xi)\r|\ls |\xi|^{-|\az|}
\end{align*}
for any multi-index $  |\az|\le \lfloor\frac{n}{2}\rfloor+1$ and
$\xi\in\rn\setminus \{\vec 0_n\}$. Thus, by the H\"ormander--Mihlin multiplier
theorem (see \cite[Theorem 5.2.7]{Graf14}), we know that $T_m$ is
bounded on $L^q(\rn)$  for any given $q\in (1,\fz)$. Moreover, for
any $\xi\in \rn$, let
$$m_{\lz}(\xi):=m\lf({\xi}/{\lz}\r)=\frac{1}{1+P\lf({\xi}/{\lz}\r)}=\frac{\lz^{2m}}
{\lz^{2m}+P(\xi)}$$ for any given $\lz^{2m}\in (0,\fz)$. By the
dilation invariant property of the $L^q(\rn)$ Fourier multiplier (see
\cite[Theorem 2.5.14]{Graf14}), we conclude that  $T_{m_{\lz}}$
is also bounded on $L^q(\rn)$ for any given $q\in (1,\fz)$. This
immediately shows that (i) holds true when $\lz^{2m}\in (-\fz,0)$.

We now consider the case $\lz^{2m}\in\cc\setminus \rr$. To this end,
we follow the idea used in the proof of \cite[Theorem 2.6]{Pang90}. For any given
$\lz^{2m}=|\lz|^{2m}e^{i\tz}\in\cc\setminus \rr$, let
$z=|\lz|^{2m}e^{i\tz_0}$ be a rotation of $\lz^{2m}$ to the right
half complex plane satisfying $  |\tz_0|<\frac{\pi}{2}$ and
$  |\tz-\tz_0|<\frac{\pi}{2}$. Since $P(D)$ is a sectorial operator of
angle $0$, we find that $P(D)$ generates a bounded holomorphic
semigroup in the right half complex plane (see, for instance,
\cite[Proposition 3.4.4]{Haase06}). Moreover, by the higher order
Gaussian upper bound  of its heat kernel (see Lemma \ref{lems2-1})
and an argument similar to that used in the proof of
\cite[Theorem 6.16]{Ouha05}, we know that,
for any given such $\tz$ and $\tz_0$, and any
 $t\in (0,\fz)$ and $x$, $y\in\rn$,
\begin{align}\label{eqn-complexhk}
\lf|K_{te^{-i(\tz-\tz_0)}}(x,y)\r|\ls \frac{1}{t^{n/(2m)}}
\exp\lf\{-\frac{\wz c_0|x-y|^{2m/(2m-1)}}{t^{1/(2m-1)}}\r\}.
\end{align}
Now, let $k\in\nn$ satisfy $kn>2m$.
Since the semigroup generated by $-P(D)$
is holomorphic in the right half complex plane,
we deduce from the Phillips calculus for semigroup generators
(see \cite[Proposition 3.3.5]{Haase06}),
$z=\lz^{2m}e^{i(\tz_0-\tz)}$, and
the change of variables ($t=e^{-i(\tz-\tz_0)}\wz t$) that
\begin{align*}
e^{i(\tz-\tz_0)}\lf(\lambda^{2m}+P(D)\r)^{-1/k}&=\frac{e^{i(\tz-\tz_0)}}{\Gamma(1/k)}
\int_0^\infty t^{1/k-1}e^{-t\lambda^{2m}}e^{-tP(D)}dt\\
&= \frac{e^{i(\tz-\tz_0)(1-1/k)}}{\Gamma(1/k)}\dint_0^\fz
\wz t^{\,1/k-1}e^{-\wz t z}e^{-\wz te^{-i(\tz-\tz_0)}P(D)}\,d\wz t,
\end{align*}
which, together with \eqref{eqn-complexhk}, implies that
the integral kernel $K_{\lz,k}(x,y)$ of $(\lambda^{2m}+P(D))^{-1/k}$ satisfies, for
any $x$, $y\in \rn$,
\begin{align}\label{eqn-Klz}
\lf|K_{\lz,k}(x,y)\r|
&\ls \dint_0^\fz e^{-t |\lz|^{2m} \cos\tz_0}
{t^{1/k-n/(2m)}}\exp\lf\{-\frac{\wz c_0|x-y|^{2m/(2m-1)}}{t^{1/(2m-1)}}\r\}\,\frac{dt}{t}\\ \notag
&\sim |x-y|^{2m/k-n}\dint_0^\fz s^{1/k-n/(2m)}\exp\lf\{-\cos \tz_0|\lz|^{2m}
|x-y|^{2m}s-\wz c_0s^{-1/(2m-1)}\r\}\,\frac{ds}{s}\\ \notag
&\sim |x-y|^{2m/k-n}\dint_0^\fz s^{1/k-n/(2m)}\exp\{-\epsilon s^{-\frac1{2m-1}}\}\\ \notag
&\quad\times\exp\lf\{-\cos \tz_0|\lz|^{2m}
|x-y|^{2m}s-(\wz c_0-\epsilon)s^{-\frac1{2m-1}}\r\}\,\frac{ds}{s},
\end{align}
where we made the change of variables $t=s|x-y|^{2m}$ in the
second equality. Since $kn>2m$, we know that
\begin{align*}
\dint_0^\fz s^{1/k-n/(2m)}\exp\lf\{-\epsilon s^{-\frac1{2m-1}}\r\}\,\,\frac{ds}{s}\ls 1.
\end{align*}
Next, for any $s\in (0,\fz)$, let
$$F(s):=\exp\lf\{-\epsilon
s^{-\frac1{2m-1}}\r\}\exp\lf\{-\cos \tz_0|\lz|^{2m} |x-y|^{2m}s-(\wz
c_0-\epsilon)s^{-\frac1{2m-1}}\r\}.$$ It is easy
to see that $F(s)\ls \exp\{-c_8|\lz||x-y|\}$ for some positive
constant $c_8$ which is independent of $s$, $x$, $y$, and $\lz$.
By this and \eqref{eqn-Klz}, we obtain
\begin{align*}
\lf|K_{\lz,k}(x,y)\r|\ls |x-y|^{2m/k-n}\exp\lf\{-c_8|\lz||x-y|\r\}=:P_{\lz,k}(x-y).
\end{align*}
Now, applying the Young inequality, we obtain, for any $f\in L^q(\rn)$,
\begin{align*}
\lf\|\lf(\lambda^{2m}+P(D)\r)^{-1/k}f\r\|_{L^q(\rn)}=\lf\|P_{\lz,k}\ast f\r\|_{L^q(\rn)}
\le \|P_{\lz,k}\|_{L^1}\|f\|_{L^q(\rn)}
\ls \frac{1}{|\lz|^{2m/k}}\|f\|_{L^q(\rn)},
\end{align*}
which immediately shows (i).

To prove (ii), for any $f\in L^q(\rn)$, write
\begin{align}\label{eqn-lem-resolventforLc2}
\lf(|\lambda|^2-\Delta\r)^{s/2}\lf(\lambda^{2m}-P(D)\r)^{-1}f=\lf(|\lambda|^2-\Delta\r)^{-(m-s/2)}\circ
\lf(|\lambda|^2-\Delta\r)^{m}\lf(\lambda^{2m}-P(D)\r)^{-1}f.
\end{align}
Using the uniform ellipticity condition \eqref{eqn-uec} of $P(D)$ and the
H\"ormander--Mihlin multiplier theorem, we know that the operator
$\Delta^mP(D)^{-1}$ is bounded on $L^q(\rn)$ for any given $q\in (1,\fz)$.
From this and (i), it follows that
\begin{align*}
\lf\|\Delta^m\lf(\lambda^{2m}-P(D)\r)^{-1}f \r\|_{L^q(\rn)}&
\ls \lf\|P(D)\lf(\lambda^{2m}-P(D)\r)^{-1}f \r\|_{L^q(\rn)}\\
&\ls \lf\|f \r\|_{L^q(\rn)}+\lf\||\lz|^{2m}\lf(\lambda^{2m}-P(D)\r)^{-1}f \r\|_{L^q(\rn)} \ls \|f\|_{L^q(\rn)},
\end{align*}
which, combined with the Gagliardo--Nirenberg inequality (see Lemma
\ref{lem-GNI}), implies that
\begin{align}\label{eqn-lem-resolventforLc1}
&\lf\|\lf(|\lambda|^2-\Delta\r)^{m}\lf(\lambda^{2m}-P(D)\r)^{-1}f\r\|_{L^q(\rn)}\\ \notag
&\hs\ls\dsum_{k=0}^m |\lz|^{2k}\lf\|\Delta^{m-k} \lf(\lambda^{2m}-P(D)\r)^{-1}f\r\|_{L^q(\rn)}\\ \notag
&\hs\ls\dsum_{k=0}^m |\lz|^{2k}\lf\|\lf(\lambda^{2m}-P(D)\r)^{-1}f\r\|_{L^q(\rn)}^{k/m}
\lf\|\Delta^{m} \lf(\lambda^{2m}-P(D)\r)^{-1}f\r\|_{L^q(\rn)}^{(m-k)/m}\\ \notag
&\hs\ls \lf[\dsum_{k=0}^m |\lz|^{2m}\lf\|\lf(\lambda^{2m}-P(D)\r)^{-1}f\r\|_{L^q(\rn)}\r]^{k/m}
 \lf[\dsum_{k=0}^m\lf\|\Delta^m \lf(\lambda^{2m}-P(D)\r)^{-1}f\r\|_{L^q(\rn)}\r]^{(m-k)/m}
\\  \notag
&\hs\ls \lf\|f\r\|_{L^q(\rn)}.
\end{align}

On the other hand, since the kernel of the semigroup
$\{e^{t\Delta}\}$ satisfies the Gaussian upper bound, we know that,
for any given $1\le q\le p\le\fz$, and any $t\in (0,\fz)$ and $g\in L^q(\rn)$,
\begin{align*}
\lf\|e^{t\Delta}g\r\|_{L^p(\rn)}\ls t^{\frac{n}{2}(\frac{1}{p}-\frac{1}{q})}\|g\|_{L^q(\rn)}.
\end{align*}
Thus, by the Phillips calculus for semigroup generators again (see \cite[Proposition 3.3.5]{Haase06})
and the assumption $  \frac{n}{2}(\frac{1}{p}-\frac{1}{q})<m-s/2$, we know that
\begin{align*}
\lf\|\lf(|\lambda|^2-\Delta\r)^{-(m-s/2)}g\r\|_{L^p(\rn)}
&=\lf\|\frac{1}{\Gamma(m-s/2)}\dint_0^\fz t^{m-s/2-1} e^{-|\lz|^2t}e^{t\Delta}g\,dt\r\|_{L^p(\rn)}\\
&\ls\dint_0^\fz t^{m-s/2-1} e^{-|\lz|^2t}\lf\| e^{t\Delta}g\r\|_{L^p(\rn)}\,dt\\
&\ls |\lz|^{-[2m-s+n(\frac{1}{p}-\frac{1}{q})]}\|g\|_{L^q(\rn)},
\end{align*}
which, together with \eqref{eqn-lem-resolventforLc2} and \eqref{eqn-lem-resolventforLc1},
shows  (ii). This finishes the proof of Lemma \ref{lem-resolventforLc}.
\end{proof}

Based on Lemma \ref{lem-resolventforLc}, we now give the first perturbation result as follows.

\begin{proposition}\label{prop-perburbationresol}
Let $m\in\nn$, $s\in (0,2m]$, $\lambda^{2m}\in \cc\setminus[0,\fz)$, $q\in (1,\fz)$, and
$ p\in [q,\fz)$ satisfy $\frac{n}{2}\Big(\frac{1}{q}-\frac{1}{p}\Big)<m-s/2$.
Assume that $V$ is a measurable function on $\rn$ and
the operator $T_{s,|\lz|}$, defined as in \eqref{def-Tlz}, is bounded from
$L^p(\rn)$ to $L^q(\rn)$ satisfying
\begin{align}\label{eqn-conditionforT}
 C_{(|\lz|)}:=C_2|\lz|^{-[2m-s+n(\frac{1}{p}-\frac{1}{q})]}\lf\|T_{s,|\lz|}\r\|_{L^p(\rn)\to L^q(\rn)}<1,
\end{align}
where the constant $C_2$ is as in \eqref{eqn-ConstatC3}.
Then there exists a positive constant $C$ such that, for any $f\in L^q(\rn)$,
\begin{align*}
\lf\|\lf(\lz^{2m}-P(D)-V\r)^{-1}f\r\|_{L^q(\rn)}\le \frac{C}{1- C_{(|\lz|)}}\frac{1}{|\lz|^{2m}}\|f\|_{L^q(\rn)}.
\end{align*}
\end{proposition}

\begin{proof}
By Lemma \ref{lem-resolventforLc} and \eqref{eqn-conditionforT}, we obtain, for any $f\in L^q(\rn)$,
\begin{align}\label{eqnprop-perburbationresolx1}
&\lf\| V\lf(\lz^{2m}-P(D)\r)^{-1}f\r\|_{L^q(\rn)}\\ \notag
&\hs=\lf\| V\lf(|\lz|^2-\Delta\r)^{-s/2}\circ\lf(|\lz|^2-\Delta\r)^{s/2}
\lf(\lz^{2m}-P(D)\r)^{-1}f\r\|_{L^q(\rn)}\\ \notag
&\hs\le \lf\|T_{s,|\lz|}\r\|_{L^p(\rn)\to L^q(\rn)} \lf\| \lf(|\lz|^2-\Delta\r)^{s/2}
\lf(\lz^{2m}-P(D)\r)^{-1}f\r\|_{L^p(\rn)}\\ \notag
&\hs\le  C_2|\lz|^{-[2m-s+n(\frac{1}{p}-\frac{1}{q})]}\lf\|T_{s,|\lz|}\r\|_{L^p(\rn)\to L^q(\rn)}\lf\|f\r\|_{L^q(\rn)}<\|f\|_{L^q(\rn)}.
\end{align}
This implies that the operator
$I-V(\lz^{2m}-P(D))^{-1}$ has a bounded inverse given by the Neumann series
$$\lf[I-V\lf(\lz^{2m}-P(D)\r)^{-1}\r]^{-1}=\dsum_{k=0}^\fz \lf[ V\lf(\lz^{2m}-P(D)\r)^{-1}\r]^k$$
on $L^q(\rn)$ (see, for instance, \cite[p. 162]{KlausNegal00}) that satisfies
\begin{align}\label{eqn-perburbationresol-1}
\lf\|\lf[I-V(\lz^{2m}-P(D))^{-1}\r]^{-1}\r\|_{L^q(\rn)\to L^q(\rn)}&\le
\dsum_{k=0}^\fz \lf\|\lf( V\lf(\lz^{2m}-P(D)\r)^{-1}\r)^k\r\|_{L^q(\rn)\to L^q(\rn)}\\
&\le\frac{1}{1- C_{(|\lz|)}}<\fz.\notag
\end{align}
Thus, the following perturbed resolvent identity (see, for instance,
\cite[p. 172, Lemma 2.5]{KlausNegal00})
\begin{align}\label{eqnprop-perburbationresolx2}
\lf(\lz^{2m}-P(D)-V\r)^{-1}=\lf(\lz^{2m}-P(D)\r)^{-1}\lf[I-V(\lz^{2m}-P(D))^{-1}\r]^{-1}
\end{align}
holds true.
Then \eqref{eqnprop-perburbationresolx2}, combined with
\eqref{eqn-perburbationresol-1} and Lemma \ref{lem-resolventforLc},
implies that
\begin{align*}
\lf\|\lf(\lz^{2m}-P(D)-V\r)^{-1}\r\|_{L^q(\rn)\to L^q(\rn)}\ls \frac{1}{|\lz|^{2m}}\frac{1}{1- C_{(|\lz|)}}.
\end{align*}
This finishes the proof of Proposition
\ref{prop-perburbationresol}.
\end{proof}

\subsection{Exponential perturbation estimates \label{s4.2}}

Let $\cl:=P(D)+V$ be the higher order Schr\"odinger operator as in \eqref{eqn-def-HOSO}.
In this subsection, we establish the exponential perturbation estimate for
the resolvent of $\cl$. To this end,
for any two closed subsets $E$ and $F$ in $\rn$,
the {\it intrinsic distance} $d_{\mathcal{E}}$ is defined by setting
\begin{align*}
d_{\mathcal{E}}(E,F):=\dsup_{{\phi}\in \mathcal{E}_{2m}(\rn)}
\lf[\inf\lf\{{\phi}(x)-{\phi}(y): \forall \ x\in E,\ y\in F\r\}\r],
\end{align*}
where
\begin{align}\label{eqn-e2m}
\mathcal{E}_{2m}(\rn):=\lf\{{\phi}\in \mathcal{C}^\fz(\rn):\
\lf\|D^\az {\phi}\r\|_{L^\fz(\rn)}\le 1,\ \forall \ 0\le |\az|\le 2m\r\}
\end{align}
and $\mathcal{C}^\fz(\rn)$ denotes the set of all infinitely differentiable functions
on $\rn$.

Let us recall the following properties of the intrinsic distance from \cite[Lemma 4]{Davies95}.

\begin{lemma}\label{lem-ed}
Let $E$ and $F$ be two disjoint compact convex subsets in $\rn$. Then
\begin{align*}
d_{\mathcal{E}}(E,F)\sim d(E,F)
\end{align*}
with the positive equivalence constants independent of $E$ and $F$, where
\begin{align*}
d(E,F):=\inf\{|x-y|: \ \forall\ x\in E, \ y\in F\}
\end{align*}
denotes the Euclidean distance between $E$ and $F$.
\end{lemma}

Now, for any given $\eta\in \cc$ and ${\phi}\in \mathcal{E}_{2m}(\rn)$,
the {\it exponential perturbed operator $\cl_{\eta,\phi}$ of} $\cl$ is defined by setting
\begin{align*}
\mathrm{dom}\,(\cl_{\eta,\phi}):=\lf\{f\in L^2(\rn):\ e^{\eta\phi} f\in \mathrm{dom}\,(\cl)\r\}
\end{align*}
and, for any $f\in \mathrm{dom}\,(\cl_{\eta,\phi})$,
\begin{align}\label{eqn-defcleta}
\cl_{\eta,\phi}f:=e^{-\eta\phi}\cl \lf(e^{\eta\phi} f\r).
\end{align}

\begin{remark}\label{rem-cophi}
\begin{enumerate}
\item[(i)] If $\eta\in \rr$, let $L_{\eta,\phi}^2(\rn)$ be the weighted space equipped with
the norm
\begin{align*}
\lf\|f\r\|_{L^2_{\eta,\phi}(\rn)}:=\lf\{\dint_\rn \lf|f(x)\r|^2\,e^{2\eta\phi(x)}\,dx\r\}^{1/2}.
\end{align*}
Since $e^{\eta\phi}$ is an isometry from $L^2_{\eta,\phi}(\rn)$ to
$L^2(\rn)$,  it follows from
\cite[p. 59]{KlausNegal00} that the operator $-\cl_{\eta,\phi}$ generates a semigroup
$\{e^{-\eta\phi} e^{-t\cl} e^{\eta\phi}\}_{t>0}$
on the weighted space $L_{\eta,\phi}^2(\rn)$
that is similar to $\{e^{-t\cl}\}_{t>0}$ on $L^2(\rn)$,
Moreover, $\sz(\cl)=\sz(\cl_{\eta,\phi})$ and the following
perturbation identity
\begin{align}\label{eqn-scalingresolvent}
\lf(\lz^{2m}-\cl_{\eta,\phi}\r)^{-1}=e^{-\eta\phi} \lf(\lz^{2m}-\cl\r)^{-1} e^{\eta\phi} \ \ \ \text{on}\ \ \  L^2_{\eta,\phi}(\rn)
\end{align}
holds true for any given $\lz^{2m}\in \rho(\cl)$.

\item[(ii)] If $\eta\in\cc^n$ is a $n$-dimensional complex vector, then, similarly to
$\cl_{\eta,\phi}$ in \eqref{eqn-defcleta}, another {\it exponential
perturbed operator $\cl_\eta$} of $\cl$ is defined by setting
\begin{align}\label{eqn-defcleta-1}
\cl_{\eta}f:=e^{-\eta x}\cl \lf(e^{\eta x} f\r)
\end{align}
with $\mathrm{dom}\,(\cl_{\eta}):=\{f\in L^2(\rn):\ e^{\eta x} f\in \mathrm{dom}\,(\cl)\}$
(see \cite[Chapter 5.3]{Tanabe97} for more details).
As we show later in this subsection, the perturbed
operators $\cl_\eta$ and $\cl_{\eta,\phi}$ share many properties.
\end{enumerate}
\end{remark}

In what follows, we establish some resolvent estimates for both
$\cl_{\eta,\phi}$ and $\cl_\eta$, and
show that the perturbation identity
\eqref{eqn-scalingresolvent} holds true also for some
$\eta$ that are complex. To this end, we use $\mathcal{C}^{2m}(\rn)$
for any $m\in\nn$ to denote the set of all functions having continuous
derivatives till order $2m$ and we need the following technical lemma.

\begin{lemma}\label{lem-identityforcleta}
Let $\eta\in \cc$ and $V$ be a   measurable function on $\rn$.
Assume that $\cl$ and $\cl_{\eta,\phi}$ are respectively as in \eqref{eqn-def-HOSO} and \eqref{eqn-defcleta}.
Then, for any $f\in \mathcal{C}^{2m}(\rn)$ and $x\in\rn$,
\begin{align}\label{eqn-lem-identityforcleta-1}
\cl_{\eta,\phi} f(x)=\lf(P(D+\eta D\phi)+V\r)f(x),
\end{align}
where $  P(D+\eta D{\phi})=\sum_{|\az|=2m}(-1)^m a_\az (D+\eta D{\phi})^\az$.
\end{lemma}

\begin{proof}
For any given $j\in \{1,\ldots,n\}$, any $f\in \mathcal{C}^{2m}(\rn)$ and
$x\in \rn$, it is easy to see that
\begin{align*}
e^{-\eta\phi}\frac{\pat}{\pat x_j}\lf(e^{\eta\phi}f\r)(x)=\lf(\frac{\pat}{\pat x_j}+\eta \frac{\pat {\phi}}{\pat x_j}\r)f(x),
\end{align*}
which, together with  \eqref{eqn-defcleta}, shows \eqref{eqn-lem-identityforcleta-1}
holds true. This finishes the proof of Lemma \ref{lem-identityforcleta}.
\end{proof}

The following proposition establishes the resolvent estimate for $\cl_{\eta,\phi}$.

\begin{proposition}\label{prop-perburbationresol2}
Let $p\in (1,\fz)$ and $\lz^{2m}\in \cc\setminus[0,\fz)$. Assume
that $\cl$ and $\cl_{\eta,\phi}$ are defined, respectively, as in
\eqref{eqn-def-HOSO} and \eqref{eqn-defcleta}. If
\begin{align}\label{eqn-prop-perburbationresol2-3}
\lf\|\lf(1+|\lz|^{2m}\r)\lf(\lz^{2m}-\cl\r)^{-1}\r\|_{L^p(\rn)\to L^p(\rn)}+\lf\|V\lf(\lz^{2m}-\cl\r)^{-1}\r\|_{L^p(\rn)\to L^p(\rn)}<\fz,
\end{align}
then there exist positive constants $C$ and $\dz\in (0,1)$ such that, for any $\eta\in \cc$
satisfying $|\eta|<\dz |\lz|$ and $f\in L^p(\rn)$,
\begin{align}\label{eqn-assumptionforcleta1}
\lf\||\lz|^{2m}\lf(\lz^{2m}-\cl_{\eta,\phi}\r)^{-1}f\r\|_{L^p(\rn)}+
\lf\|\Delta^m\lf(\lz^{2m}-\cl_{\eta,\phi}\r)^{-1}f\r\|_{L^p(\rn)}\le C\|f\|_{L^p(\rn)}.
\end{align}
\end{proposition}

\begin{proof}
For any given $p\in (1,\fz)$ and
$\lz^{2m}\in \cc\setminus[0,\fz)$, and any $f\in L^p(\rn)$, let
\begin{align}\label{eqn-defg}
g:=\lf(\lz^{2m}-\cl\r)^{-1}(f).
\end{align}
By Lemma \ref{lem-identityforcleta} and \eqref{eqn-e2m}, we obtain, for any $x\in\rn$,
\begin{align}\label{eqn-sm}
\lf|(\cl_{\eta,\phi}-\cl)g(x)\r|&\ls \dsum_{|\az|=2m}\lf|(D+\eta D\phi)^\az g(x)-D^\az g(x)\r|\\ \notag
&\ls \dsum_{k=0}^{2m-1}\dsum_{l=1}^{2m-k}\lf|\nabla^{k} g(x)\r|
|\eta|^{l}\ls \dsum_{k=0}^{2m-1} \lf(|\eta|+|\eta|^{2m-k}\r)\lf|\nabla^{k} g(x)\r|,
\end{align}
which, combined with the Gagliardo--Nirenberg inequality (see Lemma
\ref{lem-GNI}) and the assumptions $|\eta|<\dz |\lz|$ with $\dz<1$, implies that
\begin{align}\label{prop-perburbationresol2-1}
\lf\|(\cl_{\eta,\phi}-\cl)g\r\|_{L^p(\rn)}&\ls \dsum_{k=0}^{2m-1}
\lf(|\eta|+|\eta|^{2m-k}\r)\lf\|\nabla^{k} g\r\|_{L^p(\rn)}\\ \notag
&\ls \dsum_{k=0}^{2m-1} \lf(|\eta|^{2m-k} \|g\|_{L^p(\rn)}^{\frac{2m-k}{2m}}
\lf\|\Delta^m g\r\|^{\frac{k}{2m}}_{L^p(\rn)}+|\eta|\|g\|_{L^p(\rn)}^{\frac{2m-k}{2m}}
\lf\|\Delta^m g\r\|^{\frac{k}{2m}}_{L^p(\rn)}\r)\\ \notag
&\ls \dz \dsum_{k=0}^{2m-1} \lf[\lf(|\lz|^{\frac{2m}{2m-k}}+|\lz|^{2m}\r) \|g\|_{L^p(\rn)}\r]^{\frac{2m-k}{2m}}
\lf\|\Delta^m g\r\|_{L^p(\rn)}^{\frac{k}{2m}}\\ \notag
&\ls \dz \lf[\lf(1+|\lz|^{2m}\r)\|g\|_{L^p(\rn)}+\|\Delta^{m}g\|_{L^p(\rn)}\r].
\end{align}
On the other hand, similarly to the proof of
Lemma \ref{lem-resolventforLc}(ii), we know that the operator $\Delta^m P(D)^{-1}$
is bounded on $L^p(\rn)$. By this, \eqref{eqn-prop-perburbationresol2-3}, and
\eqref{eqn-defg}, we obtain
\begin{align*}
&\lf(1+\lf|\lz\r|^{2m}\r)\lf\|g\r\|_{L^p(\rn)}+\lf\|\Delta^m g\r\|_{L^p(\rn)}\\
&\hs\ls\lf(1+\lf|\lz\r|^{2m}\r)\lf\|g\r\|_{L^p(\rn)}+\lf\|P(D) g\r\|_{L^p(\rn)}\\
&\hs\ls\lf(1+\lf|\lz\r|^{2m}\r)\lf\|g\r\|_{L^p(\rn)}+\lf\|\lf(P(D)+V-\lz^{2m}\r)g\r\|_{L^p(\rn)}
+\lf\|Vg\r\|_{L^p(\rn)}\ls \lf\|f\r\|_{L^p(\rn)}.
\end{align*}
Thus, if the constant $\dz$ in \eqref{prop-perburbationresol2-1} is
sufficiently small, we then have
\begin{align*}
\lf\|(\cl_{\eta,\phi}-\cl)\lf(\lz^{2m}-\cl\r)^{-1}\r\|_{L^p(\rn)\to L^p(\rn)}<1,
\end{align*}
which implies that the operator
$I-(\cl_{\eta,\phi}-\cl)\lf(\lz^{2m}-\cl\r)^{-1}$ has a bounded inverse on $L^p(\rn)$.
By the identity
\begin{align*}
\lz^{2m}-\cl_{\eta,\phi}=\lf[I-(\cl_{\eta,\phi}-\cl)\lf(\lz^{2m}-\cl\r)^{-1}\r](\lz^{2m}-\cl)
\end{align*}
and \eqref{eqn-prop-perburbationresol2-3},
we conclude that $(\lz^{2m}-\cl_{\eta,\phi})^{-1}$ exists and \eqref{eqn-assumptionforcleta1}
holds true. This finishes the proof of
Proposition \ref{prop-perburbationresol2}.
\end{proof}

\begin{remark}\label{rem-perburbationresol2}
Let $\cl_{\eta}$ be the exponential perturbed operator as in \eqref{eqn-defcleta-1} with $\eta\in\cc^n$.
Since, for any $j\in \{1,\ldots,n\}$, $f\in \mathcal{C}^{2m}(\rn)$, and $x\in \rn$,
\begin{align*}
e^{-\eta x}\frac{\pat}{\pat x_j}\lf(e^{\eta x}f\r)(x)=\lf(\frac{\pat}{\pat x_j}+\eta_j\r)f(x),
\end{align*}
it follows that, for any $f\in \mathcal{C}^{2m}(\rn)$ and $x\in\rn$,
\begin{align*}
\cl_{\eta} f(x)=\lf(P(D+\eta)+V\r)f(x),
\end{align*}
where $  P(D+\eta)=\sum_{|\az|=2m}(-1)^m a_\az (D+\eta)^\az$ (see also
\cite[(5.92)]{Tanabe97}). Similarly to \eqref{eqn-sm}, we also obtain,
for $g$ as in \eqref{eqn-defg} and any $x\in\rn$,
\begin{align*}
\lf|(\cl_{\eta}-\cl)g(x)\r|\ls \dsum_{|\az|=2m}\lf|(D+\eta)^\az g(x)-D^\az g(x)\r|
\ls\dsum_{k=0}^{2m-1}|\eta|^{2m-k}\lf|\nabla^{k} g(x)\r|.
\end{align*}
Here, since $\eta$ is a constant vector, we only have the term $|\eta|^{2m-k}$
in the last formula,  while in \eqref{eqn-sm} we need $|\eta|+|\eta|^{2m-k}$.
This is the main difference for applications of the exponential perturbed
operators $\cl_{\eta,\phi}$ and $\cl_{\eta}$.

Thus, following the argument used in \eqref{prop-perburbationresol2-1} in the proof of Proposition
\ref{prop-perburbationresol2}, we further obtain
\begin{align*}
\lf\||\lz|^{2m}\lf(\lz^{2m}-\cl_{\eta}\r)^{-1}f\r\|_{L^p(\rn)}
+\lf\|\Delta^m\lf(\lz^{2m}-\cl_{\eta}\r)^{-1}f\r\|_{L^p(\rn)}\ls \|f\|_{L^p(\rn)}
\end{align*}
with the implicit positive constant independent of $\lz$, $\eta$, and $f$,
under the assumption that
\begin{align*}
\lf\||\lz|^{2m}\lf(\lz^{2m}-\cl\r)^{-1}\r\|_{L^p(\rn)\to L^p(\rn)}
+\lf\|V\lf(\lz^{2m}-\cl\r)^{-1}\r\|_{L^p(\rn)\to L^p(\rn)}<\fz
\end{align*}
which is slightly different from \eqref{eqn-prop-perburbationresol2-3}.
\end{remark}

Now, for any given $\lz^{2m}\in \rho(\cl_{\eta,\phi})$, it is well known that
the resolvent $(\lz^{2m}-\cl_{\eta,\phi})^{-1}$ is holomorphic on
$\lz^{2m}$ due to the resolvent identity
$$\lf(\lz^{2m}-\cl_{\eta,\phi}\r)^{-1}-\lf(\zeta^{2m}-\cl_{\eta,\phi}\r)^{-1}
=\lf(\zeta^{2m}-\lz^{2m}\r)\lf(\lz^{2m}-\cl_{\eta,\phi}\r)^{-1}\lf(\zeta^{2m}-\cl_{\eta,\phi}\r)^{-1}.$$
In what follows, we show that $(\lz^{2m}-\cl_{\eta,\phi})^{-1}$ is
also holomorphic on $\eta$ for any fixed $\lz^{2m}\in \rho(\cl_{\eta,\phi})$. That is, for any $f\in
L^p(\rn)$, the limit
\begin{align}\label{eqn-hlmit}
\lim_{h\to 0}\frac{(\lz^{2m}-\cl_{\eta+h,\phi})^{-1}f-(\lz^{2m}-\cl_{\eta,\phi})^{-1}f}{h}
\end{align}
exists in $L^p(\rn)$ (see \cite[Section B.3]{HNVW16} for some
backgrounds about the vector-valued holomorphic
functions). To this end, we need the following resolvent identity.

\begin{lemma}\label{lem-resolventi2}
Let $\eta$, $h\in\cc$, $\lz^{2m}\in \cc\setminus [0,\fz)$, and $\cl$ and
$\cl_{\eta,\phi}$ be defined, respectively, as in \eqref{eqn-def-HOSO} and
\eqref{eqn-defcleta}. Assume that $\cl$ satisfies
\eqref{eqn-prop-perburbationresol2-3}, and $\eta$ and $h$ satisfy
$|\eta|+|h|<\dz |\lz|$ with $\dz$ as in Proposition
\ref{prop-perburbationresol2}. Then it holds true that
$$\lf(\lz^{2m}-\cl_{\eta+h,\phi}\r)^{-1}-\lf(\lz^{2m}-\cl_{\eta,\phi}\r)^{-1}
=\lf(\lz^{2m}-\cl_{\eta,\phi}\r)^{-1} \lf(\cl_{\eta+h,\phi}-\cl_{\eta,\phi}\r)
\lf(\lz^{2m}-\cl_{\eta+h,\phi}\r)^{-1}.$$
\end{lemma}

\begin{proof}
Since $\cl$ satisfies \eqref{eqn-prop-perburbationresol2-3} and
$|\eta|+|h|<\dz |\lz|$, we deduce from Proposition
\ref{prop-perburbationresol2} that both
$(\lz^{2m}-\cl_{\eta+h,\phi})^{-1}$ and
$(\lz^{2m}-\cl_{\eta,\phi})^{-1}$ exist. Moreover,  we have
\begin{align*}
&\lf(\lz^{2m}-\cl_{\eta,\phi}\r)^{-1}+\lf(\lz^{2m}-\cl_{\eta,\phi}\r)^{-1} \lf(\cl_{\eta+h,\phi}-\cl_{\eta,\phi}\r)
\lf(\lz^{2m}-\cl_{\eta+h,\phi}\r)^{-1}\\
&\hs=\lf(\lz^{2m}-\cl_{\eta,\phi}\r)^{-1}\lf[\lf(\lz^{2m}-\cl_{\eta+h,\phi}\r)+\lf(\cl_{\eta+h,\phi}-\cl_{\eta,\phi}\r)\r]
\lf(\lz^{2m}-\cl_{\eta+h,\phi}\r)^{-1}\\
&\hs=\lf(\lz^{2m}-\cl_{\eta,\phi}\r)^{-1}\lf(\lz^{2m}-\cl_{\eta,\phi}\r)\lf(\lz^{2m}-\cl_{\eta+h,\phi}\r)^{-1}
=\lf(\lz^{2m}-\cl_{\eta+h,\phi}\r)^{-1},
\end{align*}
which completes the proof of Lemma \ref{lem-resolventi2}.
\end{proof}

In what follows, for any $m\in\nn$ and $q\in (1,\fz)$,
the \emph{Sobolev space $W^{2m,q}(\rn)$} is defined by setting
\begin{align*}
W^{2m,q}(\rn):=\lf\{f\in L^p(\rn):\
\lf\|f\r\|_{W^{2m,q}(\rn)}:=
\lf[\dsum_{k=0}^{2m} \dsum_{|\az|=k}\lf\|D^\az f\r\|^p_{L^p(\rn)}\r]^{1/p}<\fz\r\}.
\end{align*}

\begin{proposition}\label{prop-perburbationresol3}
Let $\lz^{2m}\in \cc\setminus[0,\fz)$, $p\in (1,\fz)$, and $\cl$ and
$\cl_{\eta,\phi}$ be defined, respectively, as in \eqref{eqn-def-HOSO} and
\eqref{eqn-defcleta}. Assume that $\cl$ satisfies
\eqref{eqn-prop-perburbationresol2-3} and $\eta$ satisfies
$|\eta|<\dz |\lz|$ with $\dz$ as in Proposition
\ref{prop-perburbationresol2}. Then
\begin{itemize}
\item[{\rm (i)}] $(\lz^{2m}-\cl_{\eta,\phi})^{-1}$ is holomorphic on $\eta$.

\item[{\rm (ii)}]   For any $f\in L^p(\rn)$,
$(\lz^{2m}-\cl_{\eta,\phi})^{-1}f=e^{-\eta\phi} (\lz^{2m}-\cl)^{-1}(e^{\eta\phi}f)$ in $L^p(\rn)$.
\end{itemize}
\end{proposition}

\begin{proof}
We first prove (i). For any $f\in L^p(\rn)$, let
$u_\eta:=(\lz^{2m}-\cl_{\eta,\phi})^{-1}f$. We need to show the
limit of \eqref{eqn-hlmit} exists in $L^p(\rn)$. By Proposition
\ref{prop-perburbationresol2}, it is easy to see that
 $u_\eta\in W^{2m,p}(\rn)$ and
\begin{align}\label{eqn-defueta}
\lf(\lz^{2m}-\cl_{\eta,\phi}\r)u_\eta =f.
\end{align}
Differentiating both sides of the identity with respect to $\eta$,
we conclude that
\begin{align}\label{prop-perburbationresol3-1}
\lf(\lz^{2m}-\cl_{\eta,\phi}\r)\frac{\pat u_\eta}{\pat
\eta}=\frac{\pat\cl_{\eta,\phi}}{\pat \eta}u_\eta=:g,
\end{align}
where $  \frac{\pat\cl_{\eta,\phi}}{\pat \eta}=\sum_{|\az|=2m}(-1)^m
a_\az \frac{\pat}{\pat \eta }(D+\eta D{\phi})^\az$. Since $u_\eta\in
W^{2m,p}(\rn)$, we know that $g\in L^p(\rn)$. Now, let
$  \nu:=(\lz^{2m}-\cl_{\eta,\phi})^{-1}g\in L^p(\rn)$ and
$$u_{\eta,h}:=\frac{1}{h}\lf(u_{\eta+h}-u_\eta\r)$$
with $h\in\cc$ and $|h|\ll1$ small enough.
Then, using Lemma \ref{lem-resolventi2}, we obtain
\begin{align}\label{eqnlzclphi}
\lf(\lz^{2m}-\cl_{\eta,\phi}\r)\lf(u_{\eta,h}-\nu\r)
&=\frac{1}{h} \lf[\lf(\lz^{2m}-\cl_{\eta,\phi}\r)\lf(\lf(\lz^{2m}-\cl_{\eta+h,\phi}\r)^{-1}
-\lf(\lz^{2m}-\cl_{\eta,\phi}\r)^{-1}\r)\r]f-g\\ \notag
&=\frac{1}{h} \lf(\cl_{\eta+h,\phi}-\cl_{\eta,\phi}\r)\lf(\lz^{2m}-\cl_{\eta+h,\phi}\r)^{-1}f-g\\ \notag
&=\frac{1}{h} \lf(\cl_{\eta+h,\phi}-\cl_{\eta,\phi}\r)u_{\eta+h}-g.
\end{align}
Let $F_h:=\frac{1}{h}\lf(\cl_{\eta+h,\phi}-\cl_{\eta,\phi}\r)
(u_{\eta+h}-u_\eta)$. By \eqref{eqn-defueta} and Lemma
\ref{lem-resolventi2}, we easily know that
\begin{align}\label{eqn-Fh}
F_h=&\frac{1}{h}\lf(\cl_{\eta+h,\phi}-\cl_{\eta,\phi}\r)\lf(\lf(\lz^{2m}-\cl_{\eta+h,\phi}\r)^{-1}
-\lf(\lz^{2m}-\cl_{\eta,\phi}\r)^{-1}\r)f\\ \notag
=&\frac{1}{h}\lf(\cl_{\eta+h,\phi}-\cl_{\eta,\phi}\r)\lf(\lz^{2m}-\cl_{\eta,\phi}\r)^{-1}\circ
\lf(\cl_{\eta+h,\phi}-\cl_{\eta,\phi}\r)\lf(\lz^{2m}-\cl_{\eta+h,\phi}\r)^{-1}f.
\end{align}
From Lemma \ref{lem-identityforcleta}, \eqref{eqn-sm}, and the
assumptions $|\eta|<\dz |\lz|$, $|\dz|<1$, and  $|h|<1$, we deduce that, for any
$u\in W^{2m,p}(\rn)$,
\begin{align*}
\lf|\lf(\cl_{\eta+h,\phi}-\cl_{\eta,\phi}\r)u\r|&\ls \dsum_{|\az|=2m} \lf|\lf[\lf(D+\eta D\phi\r)+hD\phi\r]^\az u
-\lf(D+\eta D\phi\r)^\az u\r|\\
&\ls \dsum_{k=0}^{2m-1}\dsum_{l=1}^{2m-k} \dsum_{|\bz|=k}
\lf|\lf(D+\eta D\phi\r)^\bz u\r||h|^l\\
&\le C_{(\lz)} |h|\lf[\dsum_{k=0}^{2m-1} \lf|\nabla^k u\r|\r],
\end{align*}
where the positive constant $C_{(\lz)}$ depends on $\lz$, but is independent of
$h$ and $u$. Applying Proposition \ref{prop-perburbationresol2} and
letting $|h|\ll 1$ small enough, we obtain
\begin{align*}
&\lf\| \lf(\cl_{\eta+h,\phi}-\cl_{\eta,\phi}\r)\lf(\lz^{2m}-\cl_{\eta,\phi}\r)^{-1}\r\|_{L^p(\rn)\to L^p(\rn)}
+\lf\| \lf(\cl_{\eta+h,\phi}-\cl_{\eta,\phi}\r)\lf(\lz^{2m}-\cl_{\eta+h,\phi}\r)^{-1}\r\|_{L^p(\rn)\to L^p(\rn)}\\
&\hs\ls |h|,
\end{align*}
which, together with \eqref{eqn-Fh}, implies that
\begin{align*}
\|F_h\|_{L^p(\rn)}&= \lf\|\frac{1}{h}\lf(\cl_{\eta+h,\phi}-\cl_{\eta,\phi}\r)\lf(\lz^{2m}-\cl_{\eta,\phi}\r)^{-1}\circ
\lf(\cl_{\eta+h,\phi}-\cl_{\eta,\phi}\r)\lf(\lz^{2m}-\cl_{\eta+h,\phi}\r)^{-1} f\r\|_{L^p(\rn)}\\
&\ls |h|\|f\|_{L^p(\rn)},
\end{align*}
and hence $  \lim_{h\to 0}\|F_h\|_{L^p(\rn)}=0$.
By this, \eqref{eqnlzclphi}, and Proposition \ref{prop-perburbationresol2} again, we conclude that
\begin{align*}
\lf\|u_{\eta,h}-\nu\r\|_{L^p(\rn)}&\ls\lf\|\lf(\lz^{2m}-\cl_{\eta,\phi}\r)\lf(u_{\eta,h}-\nu\r)\r\|_{L^p(\rn)}\\
&\ls \lf\|\frac{1}{h}\lf(\cl_{\eta+h,\phi}-\cl_{\eta,\phi}\r)\lf(u_{\eta+h}-u_{\eta}\r)\r\|_{L^p(\rn)}
+\lf\|\frac{1}{h}\lf(\cl_{\eta+h,\phi}-\cl_{\eta,\phi}\r)u_{\eta}-g\r\|_{L^p(\rn)}\\
&\ls \|F_h\|_{L^p(\rn)}+\lf\|\frac{1}{h} \lf(\cl_{\eta+h,\phi}-\cl_{\eta,\phi}\r)u_\eta-g\r\|_{L^p(\rn)},
\end{align*}
which, together with \eqref{prop-perburbationresol3-1}, shows that the above norm
turns to $0$ as $h\to 0$. This shows that
\begin{align*}
\lim_{h\to 0}\frac{(\lz^{2m}-\cl_{\eta+h,\phi})^{-1}f-(\lz^{2m}-\cl_{\eta,\phi})^{-1}f}{h}=\nu
\end{align*}
in $L^p(\rn)$. Thus, the limit of \eqref{eqn-hlmit} exists in $L^p(\rn)$,
which shows (i).

For (ii), it is easy to see that (ii) holds true when $\eta$ is a pure
imaginary number, because, in this case, $e^{\eta\phi}$ is an isometry on
$L^2(\rn)$ for any ${\phi}\in\mathcal{E}_{2m}(\rn)$ (see also Remark
\ref{rem-cophi} for a similar case). Then, by (i), we conclude that
$(\lz^{2m}-\cl_{\eta,\phi})^{-1}$ is holomorphic on $\eta$ for any
$\eta\in\cc$ satisfying $|\eta|<\dz |\lz|$. This,
combined with the Morera theorem, shows that (ii) holds true for any
such $\eta$, which shows (ii) and
hence completes the proof of Proposition
\ref{prop-perburbationresol3}.
\end{proof}

\begin{remark}\label{rem-prop-perburbationresol3}
Let $\cl_{\eta}$ be the exponential perturbed operator as in \eqref{eqn-defcleta-1} with $\eta\in\cc^n$.
Following the proof of Proposition \ref{prop-perburbationresol3}, we obtain that
$(\lz^{2m}-\cl_{\eta})^{-1}$ is also holomorphic on $\eta$. Moreover, for any given
$p\in (1,\fz)$ and any $f\in L^p(\rn)$,
\begin{align}\label{eqn-resolventper}
(\lz^{2m}-\cl_{\eta})^{-1}f=e^{-\eta x} (\lz^{2m}-\cl)^{-1}(e^{\eta x}f)
\end{align}
in $L^p(\rn)$  (see \cite[Lemma 5.15 and (5.103)]{Tanabe97} for similar
results).
\end{remark}

\section{Estimates for heat kernels \label{s5}}

In this section, we prove the main results of this article, based on
the perturbation estimates established in Section
\ref{s4}. To begin with,
we first summarize some boundedness results of the resolvent of $\cl_{\eta,\phi}$
and $\cl_{\eta}$, respectively, as in \eqref{eqn-defcleta} and  \eqref{eqn-defcleta-1}.

\subsection{Preliminaries on parameters\label{s5.1}}

For any given $m\in\nn$, $q\in (1,\fz)$, $p\in [q,\fz)$, and $s\in (0,2m]$ satisfying
\begin{align}\label{ap1}
0\le n\lf(\frac{1}{q}-\frac{1}{p}\r)\le 2m-s,
\end{align}
let $V$ be a measurable function on $\rn$.
We summarize the following four groups of Schechter-type conditions
based on \eqref{eqn-conditionforT}
in Propositions \ref{prop-perburbationresol},
\ref{prop-bdT1}, \ref{prop-bdT2}, \ref{prop-bdT3},
and \ref{prop-bdT4}, respectively.
\begin{itemize}
\item [(i)] The parameters $p$, $q$, $s$, $\az$, and ${S}_1$ satisfy
\begin{align}\label{ap2}
\begin{cases}
q\in \lf[1,\dfrac{n}{n-s}\r), \\
\az\in (0,(s-n)q+n],\\
{S}_1:=s-2m+n\lf(\dfrac{1}{q}-\dfrac{1}{p}\r)+n-s+\dfrac{\az-n}{q},\\
C_2C|\lz|^{{S}_1}M_{\az,q,p',1/|\lz|}(V)<1
\end{cases}
\end{align}
with $C_2$ and $C$ as in \eqref{eqn-ConstatC3} and \eqref{BC1}.

\item [(ii)] The parameters $p$, $q$, $s$, $\az$, $t$, $\sz$, and ${S}_2$ satisfy
\begin{align}\label{ap3}
\begin{cases}
t,\ \sz\in [1,\fz],\\
\displaystyle \frac{1}{q}=\frac{1}{t}+\frac{1}{\sz},\\
\displaystyle \frac{1}{\sz}\le \frac{1}{p}\le \frac{1}{\sz}+\frac{s}{n},\\
\displaystyle {S}_2:=s-2m+n\lf(\frac{1}{q}-\frac{1}{p}\r)-s\lf[1-\frac{n}{s}\lf(\frac{1}{p}-\frac{1}{\sz}\r)\r],\\
\displaystyle C_2C|\lz|^{{S}_2}\|V\|_{L^t(\rn)}<1
\end{cases}
\end{align}
with $C_2$ and $C$ as in \eqref{eqn-ConstatC3} and \eqref{BC2}.

\item [(iii)] The parameters $p$, $q$, $s$, $\az$, $t$, and ${S}_3$ satisfy
\begin{align}\label{ap4}
\begin{cases}
s\in (0,n/2),\ q\in [2,\fz),\ t\in [q,\fz],\ r\in \lf[q,\dfrac{2n}{n-2s}\r),\\
\az\in \lf(0,n+\dfrac{(2s-n)r}{2}\r],\\
\dfrac{1}{t}+\dfrac{1}{r}=\dfrac{1}{q},\\
{S}_3:=s-2m+n\lf(\dfrac{1}{q}-\dfrac{1}{p}\r)+\dfrac{1}{r}\lf[\dfrac{(n-2s)r}{2}+\az-n\r],\\
C_2C|\lz|^{{S}_3}M_{\az,r,t,1/|\lz|}(V)<1
\end{cases}
\end{align}
with $C_2$ and $C$ as in \eqref{eqn-ConstatC3} and \eqref{BC3}.

\item [(iv)] The parameters $p$, $q$, $s$, $\az$, and ${S}_4$ satisfy
\begin{align}\label{ap5}
\begin{cases}
p=q\in (1,2],\ \az\in (0,n),\\
\az-n\le p(s-n)+\dfrac{np}{p'},\\
2n>p'(n-s),\\
{S}_4:=\az/p-2m,\\
C_2C|\lz|^{{S}_4}M_{\az,p,\fz,1/|\lz|}(V)<1
\end{cases}
\end{align}
with $C_2$ and $C$ as in \eqref{eqn-ConstatC3} and \eqref{BC4}.
\end{itemize}

For the sake of simplicity, we use the same notation $M_{|\lz|}(V)$
to denote the last quantity  in \eqref{ap2} through \eqref{ap5},
respectively. By \eqref{def-Mazqrlz}, it is easy to find that, for
any $c\in (0,\fz)$,
\begin{align}\label{HP}
M_{|\lz|}(cV)=cM_{|\lz|}(V).
\end{align}

We first show that, for parameters satisfying
\eqref{ap2} through \eqref{ap5}, the higher order Schr\"odinger operator
$\cl$ defined  in \eqref{eqn-def-HOSO} can be extended to a closed
operator on $W^{2m,q}(\rn)$ for any given $q\in (1,\fz)$.

\begin{proposition}\label{prop-reb}
Let $\lz^{2m}\in \cc\setminus[0,\fz)$,  $q\in (1,\fz)$, $p\in [q,\fz)$,
and $s\in (0,2m]$ satisfy
\eqref{ap1}. Assume that $\cl$ is as in \eqref{eqn-def-HOSO} and
one of \eqref{ap2} through \eqref{ap5} holds true. Then $\cl$ is a closed linear
operator on $W^{2m,q}(\rn)$.
\end{proposition}

\begin{proof}
For any given $q\in (1,\fz)$, it is known that $P(D)$ can be extended to a closed linear operator
on $L^q(\rn)$ with the domain $$\mathrm{dom}_q(P(D))=W^{2m,q}(\rn).$$
If one of \eqref{ap2} through \eqref{ap5} holds true,
then, by applying Proposition \ref{prop-perburbationresol} [see, in particular,
\eqref{eqnprop-perburbationresolx1}],
we know that, for any $f\in L^q(\rn)$,
\begin{align*}
\lf\|Vf\r\|_{L^q(\rn)}&=\lf\|V\lf(\lz^{2m}-P(D)\r)^{-1}\circ \lf(\lz^{2m}-P(D)\r)f\r\|_{L^q(\rn)}\\ \notag
&<\lf\|P(D)f\r\|_{L^q(\rn)} +|\lz|^{2m}\|f\|_{L^q(\rn)},
\end{align*}
which implies that $V$ is relatively bounded with respect to $P(D)$,
with bound constant strictly less than $1$. Thus, by \cite[p. 171,
Lemma 2.4]{KlausNegal00}, we conclude that $P(D)+V$ is closed on
$W^{2m,q}(\rn)$. This finishes the proof of Proposition
\ref{prop-reb}.
\end{proof}

The next result summarizes the boundedness
of the resolvent $(\lz-\cl_{\eta,\phi})^{-1}$ based on the conditions \eqref{ap2} through \eqref{ap5}.

\begin{proposition}\label{prop-spr}
Let $m\in\nn$, $\lz^{2m}\in \cc\setminus[0,\fz)$, $q\in (1,\fz)$, $p\in
[q,\fz)$, and $s\in (0,2m]$ satisfy \eqref{ap1}. Assume that $\cl$
and $\cl_{\eta,\phi}$ are defined, respectively, as in \eqref{eqn-def-HOSO} and
\eqref{eqn-defcleta}. If one of \eqref{ap2} through \eqref{ap5}
holds true and
\begin{align}\label{eqn-assumptionadd1}
\lf\|\lf(\lz^{2m}-\cl\r)^{-1}\r\|_{L^q(\rn)\to L^q(\rn)}<\fz,
\end{align}
then there exist positive constants $C$ and $\dz\in (0,1)$ such that, for any $\eta\in \cc$
satisfying $|\eta|<\dz |\lz|$ and $f\in L^q(\rn)$,
\begin{align}\label{eqn-assumptionforcleta2x}
\lf\||\lz|^{2m}\lf(\lz^{2m}-\cl_{\eta,\phi}\r)^{-1}f\r\|_{L^q(\rn)}
+\lf\|\Delta^m\lf(\lz^{2m}-\cl_{\eta,\phi}\r)^{-1}f\r\|_{L^q(\rn)}\le C\|f\|_{L^q(\rn)}.
\end{align}
\end{proposition}

\begin{proof}
Let parameters $p$, $q$, $s$, $\az$, and $\{{S}_i\}_{i=1}^4$ satisfy \eqref{ap2} through \eqref{ap5}.
It is easy to see that
\begin{align*}
 C_{(|\lz|)}:=C_2|\lz|^{-[2m-s-n(\frac{1}{q}-\frac{1}{p})]}\lf\|T_{s,|\lz|}\r\|_{L^p(\rn)\to L^q(\rn)}<1.
\end{align*}
Thus, by Propositions \ref{prop-perburbationresol}, we know that, for any $f\in L^p(\rn)$,
\begin{align*}
\lf\||\lz|^{2m}\lf(\lz^{2m}-\cl\r)^{-1}f\r\|_{L^q(\rn)}\ls \frac{1}{1- C_{(|\lz|)}}\|f\|_{L^q(\rn)}.
\end{align*}
Moreover, by \eqref{eqnprop-perburbationresolx1} through \eqref{eqnprop-perburbationresolx2}, we have
\begin{align*}
\lf\|V\lf(\lz^{2m}-\cl\r)^{-1}f\r\|_{L^q(\rn)}&\ls
\lf\|V\lf(\lz^{2m}-P(D)\r)^{-1}\r\|_{L^q(\rn)\to L^q(\rn)}
\lf\|\lf[I-V(\lz^{2m}-P(D))^{-1}\r]^{-1}f\r\|_{L^q(\rn)}\\
&\ls \frac{1}{1- C_{(|\lz|)}}\|f\|_{L^q(\rn)},
\end{align*}
which, together with \eqref{eqn-assumptionadd1}, shows
\eqref{eqn-prop-perburbationresol2-3}. By Proposition
\ref{prop-perburbationresol2},  we know that
\eqref{eqn-assumptionforcleta2x} holds true. This finishes the proof of
Proposition \ref{prop-spr}.
\end{proof}

\begin{remark}\label{rem-prop-spr}
Let $\cl_{\eta}$ be the exponential perturbed operator as in \eqref{eqn-defcleta-1} with $\eta\in\cc^n$.
By Remark \ref{rem-perburbationresol2}, we know that \eqref{eqn-assumptionforcleta2x}  holds true with
$\cl_{\eta,\phi}$ replaced by $\cl_\eta$, under the same assumptions of
Proposition \ref{prop-spr} but without the condition \eqref{eqn-assumptionadd1}.
\end{remark}

\subsection{Davies--Gaffney estimates \label{s5.2}}

In this subsection,
we establish the Davies--Gaffney estimates for the semigroup
$\{e^{-t\cl}\}_{t>0}$ generated by $-\cl$ by proving Theorem \ref{thm-main2}.

\begin{proof}[Proof of Theorem \ref{thm-main2}]
To begin with, we first claim that $-\cl$ generates a bounded
holomorphic semigroup $\{e^{-t\cl}\}_{t>0}$ on $L^2(\rn)$.
Indeed, by Proposition \ref{prop-perburbationresol}, we know that,
for any given $\lz \in \cc\setminus [0,\fz)$ and any $f\in L^2(\rn)$,
\begin{align*}
\lf\|\lf(\lz-\cl\r)^{-1}f\r\|_{L^2(\rn)}\ls \frac{1}{|\lz|}\lf\|f\r\|_{L^2(\rn)},
\end{align*}
which implies that $\cl$ is a sectorial operator on $L^2(\rn)$.
This shows the above claim by applying \cite[Chapter II, Theorem 4.6]{KlausNegal00}.

Based on the $L^2(\rn)$ boundedness of $\{e^{-t\cl}\}_{t>0}$, we
know that, to prove \eqref{eqn-DGE},
it suffices to consider the case $t<[d(E,F)]^{2m}$. Using Proposition
\ref{prop-perburbationresol3}(ii), we find that, for any given $\lz\in
\cc\setminus [0,\fz)$, ${\phi}\in \mathcal{E}_{2m}(\rn)$, and $\eta\in
\rr_+$ satisfying $\eta<\dz |\lz|^{1/(2m)}$, and any $f\in L^2(\rn)$,
\begin{align*}
\lf(\lz-\cl_{\eta,\phi}\r)^{-1}f=e^{-\eta\phi}\lf(\lz-\cl\r)^{-1}(e^{\eta\phi}f)
\end{align*}
holds true in $L^2(\rn)$ for some $\dz\in (0,1)$ sufficiently small. Then, by
Proposition \ref{prop-perburbationresol2}
and Lemma \ref{lem-ed}, we obtain, for any given $\lz\in \cc\setminus [0,\fz)$
and any disjoint compact convex subsets $E$ and $F$,
\begin{align}\label{eqn-DG-1}
&\lf\|\mathbf{1}_E \lf(\lz-\cl\r)^{-1}\mathbf{1}_F\r\|_{L^2(\rn)\to L^2(\rn)}\\ \notag
&\hs=\lf\|\mathbf{1}_Ee^{\eta\phi}\circ e^{-\eta\phi}\lf(\lz-\cl\r)^{-1}e^{\eta\phi}\circ
e^{-\eta\phi}\mathbf{1}_F\r\|_{L^2(\rn)\to L^2(\rn)}\\ \notag
&\hs\le \lf\|e^{\eta\phi}\r\|_{L^\fz(E)}\lf\|\lf(\lz-\cl_{\eta,\phi}\r)^{-1}\r\|_{L^2(\rn)\to
L^2(\rn)}\lf\|e^{-\eta\phi}\r\|_{L^\fz(F)}\\ \notag
&\hs\ls \frac{1}{|\lz|} \exp\lf\{-\eta \lf[\inf_{y\in F} {\phi}(y)-\sup_{x\in E}{\phi}(x)\r]\r\}
\ls \frac{1}{|\lz|} \exp\lf\{-c\eta d(E,F) \r\}
\end{align}
for some $c\in (1,\fz)$. Using the following identity based on functional calculus
(see, for instance, \cite[(2.25)]{AuMcTc98})
\begin{align}\label{eqn-FCHS}
e^{-t\cl}=\frac{1}{2\pi i}\dint_\Gamma e^{-t\lz} \lf(\lz-\cl\r)^{-1}\,d\lz,
\end{align}
where $\Gamma$ is a path in the complex plane that consists of three parts:
\begin{align}\label{eqn-IC}
&\Gamma_{0}:=\lf\{\dz=Re^{i\tz}: |\tz|\ge \mu\r\},\
\Gamma_{+}:=\lf\{\dz=re^{i\mu}: r\ge R\r\}, \ \text{and}\\ \notag
&\Gamma_{-}:=\lf\{\dz=re^{-i\mu}: r\ge R\r\}
\end{align}
with $\mu\in (0,\pi/2)$ being fixed and $R\in (0,\fz)$ to be determined later (see Figure \ref{gamma}
below).
\begin{figure}[ht]
    \centering
    \includegraphics[width=0.5\textwidth]{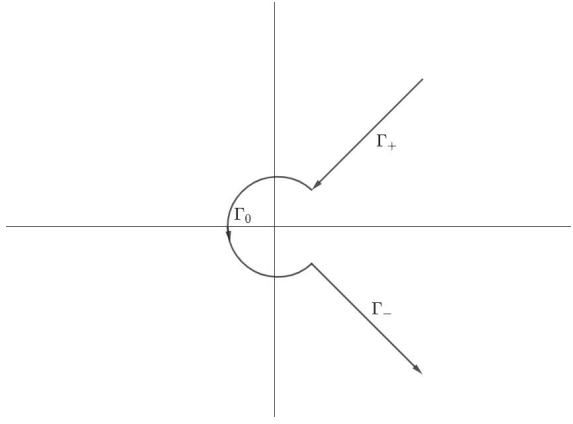}
    \caption{The path $\Gamma$}
    \label{gamma}
\end{figure}

Now, for any $f\in L^2(E)$ with $\supp f\subset E$,  we estimate $\|e^{-t\cl}f\|_{L^2(F)}$
by considering the corresponding integrals over $\Gamma_+$, $\Gamma_-$, and $\Gamma_0$,
 respectively.
To be precise, let $\eta:=\frac{\dz}{c}|\lz|^{1/(2m)}$ and $  R:=\epsilon [\frac{d(E,F)}{t}]^{2m/(2m-1)}$
with $c$ as in \eqref{eqn-DG-1} and $  \epsilon\in (0,1)$ satisfying
$\dz\epsilon^{1/(2m)}>\epsilon$. Then,
by \eqref{eqn-DG-1} and the assumptions $\mu\in (0,\pi/2)$ and $t<[d(E,F)]^{2m}$, we obtain
\begin{align*}
\mathrm{I}_+:=&\dint_{\Gamma_+} e^{-t\mathrm{Re}\,\lz} \lf\|\lf(\lz-\cl\r)^{-1}f\r\|_{L^2(F)}\lf|d\lz\r|\\
\ls & \dint_{R}^\fz \frac{1}{r}e^{-tr \cos \mu} e^{-c\eta d(E,F)}\,dr\lf\|f\r\|_{L^2(E)}\\
\ls & \lf(\frac{[d(E,F)]^{2m}}{t}\r)^{-1/(2m-1)}
\exp\lf\{-\dz\epsilon^{1/(2m)}\frac{[d(E,F)]^{2m/(2m-1)}}{t^{1/(2m-1)}}\r\}\lf\|f\r\|_{L^2(E)}\\
\ls& \exp\lf\{-\wz c_5\frac{[d(E,F)]^{2m/(2m-1)}}{t^{1/(2m-1)}}\r\}\lf\|f\r\|_{L^2(E)}
\end{align*}
by choosing a suitable constant $\wz c_5\in (0,\fz)$, which is the desired estimate. Similarly, we have
\begin{align*}
\mathrm{I}_-:=\dint_{\Gamma_-} e^{-t\mathrm{Re}\,\lz} \lf\|\lf(\lz-\cl\r)^{-1}f\r\|_{L^2(F)}\lf|d\lz\r|
\ls \exp\lf\{-\wz c_5\frac{[d(E,F)]^{2m/(2m-1)}}{t^{1/(2m-1)}}\r\}\lf\|f\r\|_{L^2(E)}.
\end{align*}
For the integral over $\Gamma_0$, using the fact that $  \int_0^{2\pi} e^{-s\cos\tz}\,d\tz\ls \frac{e^s}{s^{1/2}}$
and the assumption $\dz\epsilon^{1/(2m)}>\epsilon$,
we know that
\begin{align*}
\mathrm{I}_0:=&\dint_{\Gamma_0} e^{-t\mathrm{Re}\,\lz} \lf\|\lf(\lz-\cl\r)^{-1}f\r\|_{L^2(F)}\lf|d\lz\r|\\
\le&\lf|\dint_{\tz\in (\mu,-\mu)}e^{-tR\cos \tz}
d\tz\r| \exp\lf\{-\dz\epsilon^{1/(2m)}\frac{[d(E,F)]^{2m/(2m-1)}}{t^{1/(2m-1)}}\r\}\lf\|f\r\|_{L^2(E)}\\
\ls&  \exp\lf\{-\wz c_5\frac{[d(E,F)]^{2m/(2m-1)}}{t^{1/(2m-1)}}\r\}\lf\|f\r\|_{L^2(E)}
\end{align*}
by choosing a suitable constant $\wz c_5\in (0,\fz)$.

Combining the estimates for $\mathrm{I}_+$, $\mathrm{I}_-$, and $\mathrm{I}_0$, we conclude that
\eqref{eqn-DGE} holds true. This finishes the proof of Theorem \ref{thm-main2}.
\end{proof}

\begin{corollary}\label{cor-localDG}
Let $m\in\nn$,
$\cl=P(D)+V$ be the $2m$-order Schr\"odinger operator on
$\rn$ as in \eqref{eqn-def-HOSO}, and $V$ a measurable function on $\rn$.
If one of \eqref{ap2} through \eqref{ap5} holds true with $\sup_{|\lz|\in (w/2,\fz)}M_{|\lz|}(V)<1$
for some $w\in (0,\fz)$ and $M_{|\lz|}(V)$ as in \eqref{HP},
then there exist positive constants $C$ and $c_5$ such that, for any
$t\in (0,\fz)$, disjoint compact convex subsets $E$ and $F$, and $f\in L^2(E)$
with $\supp f\subset E$,
\begin{align}\label{eqn-lLDGE}
\lf\|e^{-t\cl}f\r\|_{L^2(F)}\le C
e^{wt}\exp\lf\{-c_5\frac{[d(E,F)]^{2m/(2m-1)}}{t^{1/(2m-1)}}\r\}\lf\|f\r\|_{L^2(E)}.
\end{align}
\end{corollary}

\begin{proof}
Since one of \eqref{ap2} through \eqref{ap5} holds true, we know that
$  \sup_{|\lz|\in (w/2,\fz)}C_{(|\lz|^{1/(2m)})}<1$, where
$C_{(|\lz|^{1/(2m)})}$ is the same constant as in
\eqref{eqn-conditionforT}. Thus, following the proof of Proposition
\ref{prop-perburbationresol}, we know that, for any $\lz\in
\cc\setminus [0,\fz)$ satisfying $|\lz|>w/2$,
\begin{align}\label{eqn-cloe1}
\lf\|\lf(\lz-\cl\r)^{-1}\r\|_{L^2(\rn)\to L^2(\rn)}\ls \frac{1}{1-a}\frac{1}{|\lz|}.
\end{align}
On the other hand, let $\cl_w:=\cl+w$. It is easy to see that there exists a
$\mu\in (0,\pi/2)$ such that,
for any $\lz\in \cc\setminus [0,\fz)$ satisfying $|\arg \lz|\ge \nu$,
\begin{align*}
\lf|\lz-w\r|>w/2,
\end{align*}
which, together with \eqref{eqn-cloe1}, shows that, for any $\lz\in \cc\setminus [0,\fz)$ with $|\arg \lz|\ge \mu$,
\begin{align*}
\lf\|\lf(\lz-\cl_w\r)^{-1}\r\|_{L^2(\rn)\to L^2(\rn)}=
\lf\|\lf(\lz-w-\cl\r)^{-1}\r\|_{L^2(\rn)\to L^2(\rn)}\ls \frac{1}{1-a}\frac{1}{|\lz-w|}\ls \frac{1}{1-a}\frac{1}{w}.
\end{align*}
Thus, following the proof of Proposition \ref{prop-spr}, we obtain
\begin{align*}
\lf\||\lz|\lf(\lz-(\cl_w)_{\eta,\phi}\r)^{-1}f\r\|_{L^q(\rn)}
+\lf\|\Delta^m\lf(\lz-(\cl_w)_{\eta,\phi}\r)^{-1}f\r\|_{L^q(\rn)}\ls\|f\|_{L^q(\rn)}.
\end{align*}
Since the points $\lz$ in the integral path $\Gamma$ of
\eqref{eqn-FCHS} (see also Figure \ref{gamma}) satisfy $|\arg
\lz|\ge \mu$, following the proof of Theorem \ref{thm-main2}, we
conclude that there exists a positive constant $c_5$ such that,
for any $t\in (0,\fz)$, disjoint compact convex subsets $E$ and $F$, and
$f\in L^2(E)$ with $\supp f\subset E$,
\begin{align*}
\lf\|e^{-t\cl_w}f\r\|_{L^2(F)}\ls
\exp\lf\{-c_5\frac{[d(E,F)]^{2m/(2m-1)}}{t^{1/(2m-1)}}\r\}\lf\|f\r\|_{L^2(E)}.
\end{align*}
This, combined with the identity
\begin{align*}
e^{-t\cl}=e^{-t\cl_w}e^{wt},
\end{align*}
shows that \eqref{eqn-lLDGE} holds true.
This finishes the proof of Corollary
\ref{cor-localDG}.
\end{proof}

\begin{remark}\label{rem-pmDG}
\begin{enumerate}
\item[{\rm (i)}]
For any given $m\in\nn$ and $a\in (-2m,0)$, let $V(x)=\pm|x|^a$ for any $x\in\rn\setminus \{\vec 0_n\}$.
By taking the parameters in \eqref{ap1} and
\eqref{ap5} with $q=p=2$, $s\in (0,2m]$, $\az\in (0,n)$, and $\az\le 2s$,
and using Remark \ref{ex-GSC}, if further assuming $a\in (-{\min\{2s,n\}}/{2},0)$,
we then have
$M_{\az,2,\fz,1/|\lz|}(V)\sim |\lz|^{-\az/2-a}$ with
the positive equivalence constants independent of $\lz$.
This implies that
\begin{align*}
M_{|\lz|}(V)=C_2C|\lz|^{{S}_4}M_{\az,2,\fz,1/|\lz|}(V)\sim
|\lz|^{S_4-(\frac{\az}{2}+a)}
\end{align*}
with the positive equivalence constants independent of $\lz$.
Since $S_4-(\frac{\az}{2}+a)=-a-2m< 0$, it follows that
\begin{align*}
\dsup_{|\lz|>\epsilon_0}M_{|\lz|}(V)<1
\end{align*}
for some $\epsilon_0\in (0,\fz)$.
Applying Corollary \ref{cor-localDG}, we conclude that the semigroup
generated by $-(P(D)+V)$ satisfies the local Davies--Gaffney estimate
\eqref{eqn-lLDGE}.

\item[{\rm (ii)}]  For any given $m\in\nn$ and $a\in (-\fz,0)$, let
$V(x)=\pm (1+|x|)^a$ for any $x\in\rn$. It is easy to see that $\|V\|_{L^t(\rn)}\ls 1$ for any
given $t\in (-n/a,\fz)$, where the implicit positive constant depends on $a$ and $t$.
Thus, by taking the parameters in \eqref{ap1} and
\eqref{ap3} with $q=p=2$, $s=2m$, and $t\in ({n}/{(2m)},\fz)$, we obtain
\begin{align*}
S_2=-2m\lf(1-\frac{n}{2m}\frac{1}{t}\r)\le 0.
\end{align*}
This implies $\sup_{|\lz|>\epsilon_0}M_{|\lz|}(V)=C_2C|\lz|^{{S}_2}\|V\|_{L^t(\rn)}<1$
for some $\epsilon_0\in (0,\fz)$
and hence
the semigroup generated by $-(P(D)+V)$ satisfies the local Davies--Gaffney
estimate \eqref{eqn-lLDGE}.
\end{enumerate}
\end{remark}

\subsection{Gaussian estimates\label{s5.3}}

In this subsection, we establish the Gaussian estimate for the heat
kernel of $\cl$ by proving Theorem
\ref{thm-main1}. We begin with the following lemma concerning
the exponential perturbed operator $\cl_{\eta}f:=e^{-\eta x}\cl \lf(e^{\eta x} f\r)$
defined in \eqref{eqn-defcleta}.

\begin{lemma}\label{lem-dual}
Let $p\in (1,\fz)$. If, for any $\lz\in \cc\setminus [0,\fz)$ and
$\eta\in\cc^n$ satisfying $|\eta|<\dz|\lz|^{1/(2m)}$ with $\dz\in (0,1)$,
$\|(\lz-\cl_{\eta})^{-1}\|_{L^p(\rn)\to L^p(\rn)}<\fz$, then
$\|(\lz-\cl_{\eta})^{-1}\|_{L^{p'}(\rn)\to L^{p'}(\rn)}<\fz$
also holds true for any such $\lz$ and $\eta$.
\end{lemma}

\begin{proof}
For any $\lz\in \cc\setminus [0,\fz)$ and $\eta\in\cc^n$ satisfying
$|\eta|<\dz|\lz|^{1/(2m)}$ with $\dz\in (0,1)$, it is easy to see
that
$$\lf[\lf(\lz-\cl_{\eta}\r)^{-1}\r]^*=\lf(\overline\lz-\cl_{\overline\eta}\r)^{-1}.$$
Since $\overline{\lz}\in \cc\setminus [0,\fz)$ and
$\overline{\eta}\in\cc^n$ also satisfy
$|\overline{\eta}|<\dz|\overline{\lz}|^{1/(2m)}$ with $\dz\in
(0,1)$, we have
\begin{align*}
\lf\|(\lz-\cl_{\eta})^{-1}\r\|_{L^{p'}(\rn)\to L^{p'}(\rn)}
=\lf\|(\overline{\lz}-\cl_{\overline\eta})^{-1}\r\|_{L^p(\rn)\to L^p(\rn)}<\fz.
\end{align*}
This finishes the proof of Lemma \ref{lem-dual}.
\end{proof}

The following Gagliardo--Nirenberg inequality and Sobolev embedding can be
deduced from \cite[(1.2)]{Tri14} and
\cite[Theorem 4.12, PART II]{AdFo03}, respectively.

\begin{lemma}\label{lem-GN2}
\begin{enumerate}
\item[{\rm (i)}] Let $1\le p\le \sz\le \fz$ and $m\in\nn$ satisfy
$0\le \frac{n}{2m}(\frac{1}{p}-\frac{1}{\sz})
\le 1$. Then there exists a positive constant $C$ such that, for any $f\in \mathcal{S}(\rn)$,
\begin{align*}
\lf\|f\r\|_{L^\sz(\rn)}\le C \lf\|f\r\|_{L^p(\rn)}^{1-\tz}\lf\|\Delta^m f\r\|_{L^p(\rn)}^{\tz}
\end{align*}
with $\tz:=\frac{n}{2m}(\frac{1}{p}-\frac{1}{\sz})$.

\item[{\rm (ii)}] Let $m\in\nn$ and $q\in (1,\fz)$ satisfy $(2m-1)q<n<2mq$.
Then $W^{2m,q}(\rn)\subset \mathcal{C}^\gz(\rn)$ with
$\gz=2m-n/q\in (0,1)$, where $\mathcal{C}^\gz(\rn)$ denotes the Lipschitz
space of order $\gz$ on $\rn$ equipped with the norm
\begin{align*}
\lf\|f\r\|_{\mathcal{C}^\gz(\rn)}:=\dsup_{x,y\in\rn,x\ne y}\frac{|f(x)-f(y)|}
{|x-y|^\gz}.
\end{align*}
\end{enumerate}
\end{lemma}

Based on the aforementioned lemmas, we now turn to the proof of
Theorem \ref{thm-main1}

\begin{proof}[Proof of Theorem \ref{thm-main1}]
We first prove (a). Without loss of generality,
we may only consider the case that \eqref{ap1} and one of \eqref{ap2} through \eqref{ap5}
hold true for any given $q\in (1,2]$, and $$\sup_{|\lz|\in(0,\fz)}M_{|\lz|}(V)<1$$ with $M_{|\lz|}(V)$
as in \eqref{HP}. In this case, applying Proposition \ref{prop-spr} and Remark \ref{rem-prop-spr}, we obtain
\begin{align}\label{eqn-heatkernel-1}
\lf\||\lz|\lf(\lz-\cl_{\eta}\r)^{-1}\r\|_{L^q(\rn)\to L^q(\rn)}
+\lf\|\Delta^m\lf(\lz-\cl_{\eta}\r)^{-1}\r\|_{L^q(\rn)\to L^q(\rn)}\ls 1,
\end{align}
which, together with Lemma \ref{lem-dual}, implies that \eqref{eqn-heatkernel-1} holds true also for
any given $q\in[2, \fz)$, $\lz\in \cc\setminus [0,\fz)$, and
$\eta\in\cc^n$ satisfying $|\eta|<\dz|\lz|^{1/(2m)}$.

Now, take $l\in\zz_+$ and choose
the numbers $2=q_0<q_1<\cdots<q_{l-1}<q_l<q_{l+1}=\fz$
satisfying that, if $n< 4m$, then $l=1$; if $n\ge4m$, then
$2(l+1)>\frac{n}{2m}$ and
\begin{align}\label{eqn-sqj}
\begin{cases}
q_l\in \lf(\dfrac{n}{2m},\dfrac{n}{2m-1}\r),\\
\dfrac{n}{q_j}\notin\nn \ \ \text{for any}\ j\in\{1,\ldots,l\}, \\
\dfrac{1}{q_j}-\dfrac{1}{q_{j+1}}<\dfrac{2m}{n} \ \ \text{for any}\ j\in\{0,\ldots,l-1\}.
\end{cases}
\end{align}
Then, for any $j\in\{0,\,\ldots,l\}$, let
\begin{align*}
a_j:=\frac{n}{2m}\lf(\frac{1}{q_j}-\frac{1}{q_{j+1}}\r)\in (0,1)
\end{align*}
with the usual convention made when $q_{l+1}=\fz$. Then, by \eqref{eqn-heatkernel-1} and
Lemma \ref{lem-GN2}, we know that, for any $f\in \mathcal{S}(\rn)$,
\begin{align}\label{eqn-lfz-lql}
\lf\|\lf(\lz-\cl_{\eta}\r)^{-(l+1)}f\r\|_{L^\fz(\rn)}&\ls \lf\|\Delta^m
\lf(\lz-\cl_{\eta}\r)^{-(l+1)}f\r\|_{L^{q_l}(\rn)}^{a_l}\lf\|\lf(\lz-\cl_{\eta}\r)^{-(l+1)}f\r\|_{L^{q_l}(\rn)}^{1-a_l} \\ \notag
&\ls \lf|\lz\r|^{a_l-1}\lf\|\lf(\lz-\cl_{\eta}\r)^{-l} f\r\|_{L^{q_l}(\rn)}.
\end{align}
Repeating the above argument, we then obtain
\begin{align}\label{eqn-forin}
\lf\|\lf(\lz-\cl_{\eta}\r)^{-(l+1)}f\r\|_{L^\fz(\rn)}&\ls \lf|\lz\r|^{\sum_{j=0}^l(a_j-1)}
\lf\|f\r\|_{L^2(\rn)}\sim \lf|\lz\r|^{\frac{1}{2}\frac{n}{2m}-(l+1)}\|f\|_{L^2(\rn)},
\end{align}
which, combined with a duality argument, shows that, for any $g\in L^1(\rn)$,
\begin{align}\label{eqn-RA}
\lf\|\lf(\lz-\cl_{\eta}\r)^{-2(l+1)}g\r\|_{L^\fz(\rn)}\ls
\lf|\lz\r|^{\frac{n}{2m}-2(l+1)}\|g\|_{L^1(\rn)}.
\end{align}
By a well-known result from \cite[p. 503, Theorem 6]{DuSc88}, we know that the operator
$(\lz-\cl_\eta)^{-2(l+1)}$ has an integral kernel $K_{2(l+1),\eta}$ on $\rn\times \rn$ that satisfies,
for any $(x,\,y)\in \rn\times \rn$,
\begin{align}\label{eqn-holder3}
\lf|K_{2(l+1),\eta}(x,y)\r|\ls \lf|\lz\r|^{\frac{n}{2m}-2(l+1)}.
\end{align}
Moreover, using \eqref{eqn-resolventper},
we find that the operator $(\lz-\cl)^{-2(l+1)}$
also has an integral kernel $K_{2(l+1)}$ on $\rn\times \rn$
that satisfies, for any $(x,\,y)\in\rn\times\rn$,
\begin{align*}
\lf|K_{2(l+1)}(x,y)\r|\ls \lf|\lz\r|^{\frac{n}{2m}-2(l+1)}e^{(x-y)\eta}.
\end{align*}
Now taking $\eta:=-\mathrm{sgn}\,(x-y)\dz|\lz|^{1/(2m)}$, we conclude
that, for any $(x,\,y)\in\rn\times\rn$,
\begin{align*}
\lf|K_{2(l+1)}(x,y)\r|\ls \lf|\lz\r|^{\frac{n}{2m}-2(l+1)}\exp\lf\{-\dz|x-y||\lz|^{1/(2m)}\r\}.
\end{align*}
By the following formula
\begin{align*}
e^{-t\cl}=\frac{[2(l+1)-1]!}{2\pi i (-t)^{2(l+1)-1}}\dint_\Gamma e^{-t\lz} \lf(\lz-\cl\r)^{-2(l+1)}\,d\lz
\end{align*}
(see \cite[(3.8)]{AuMcTc98}),
where $\Gamma$ is the curve as in \eqref{eqn-IC} (see also Figure \ref{gamma}),
we know that $e^{-t\cl}$ has an integral kernel
$p_t$ on $(0,\fz)\times \rn\times \rn$ that satisfies, for any
$t\in (0,\fz)$ and $(x,\,y)\in\rn\times\rn$,
\begin{align}\label{eqn-functionalcal}
p_t(x,y)=\frac{[2(l+1)-1]!}{2\pi i (-t)^{2(l+1)-1}}\dint_\Gamma e^{-t\lz} K_{2(l+1)}(x,y)\,d\lz.
\end{align}
From this, we then deduce that
\begin{align*}
\lf|p_t(x,y)\r|&\ls \frac{1}{t^{2(l+1)-1}}\dint_\Gamma e^{-t \,\mathrm{Re}\, \lz} \lf|\lz\r|^{\frac{n}{2m}-2(l+1)}
\exp\lf\{-\dz |x-y||\lz|^{1/(2m)}\r\}\,|d\lz|\\
&=\frac{1}{t^{2(l+1)-1}}\lf(\dint_{\Gamma_0}+\dint_{\Gamma_+}+\dint_{\Gamma_-}\r)\cdots\,|d\lz|\\
&=:\mathrm{I}_0+\mathrm{I}_++\mathrm{I}_-.
\end{align*}

For $\mathrm{I}_+$, by the assumptions that $\mu\in (0,\,\pi/2)$ and $2(l+1)>\frac{n}{2m}$, we find that
\begin{align*}
\mathrm{I}_+&\ls \frac{1}{t^{2(l+1)-1}} \dint_R^\fz e^{-t r \cos \mu} r^{\frac{n}{2m}-2(l+1)}\,dr
\exp\lf\{-\dz |x-y|R^{\frac{1}{2m}}\r\}\\ \notag
&\ls \frac{(Rt)^{n/(2m)-2(l+1)}}{t^{n/(2m)}} \dint_R^\fz e^{-t r \cos \mu}\,d(tr)\,
\exp\lf\{-\dz |x-y|R^{\frac{1}{2m}}\r\}\\ \notag
&\ls \frac{(Rt)^{n/(2m)-2(l+1)}}{t^{n/(2m)}} \exp\lf\{-\dz |x-y| R^{\frac{1}{2m}}\r\}.
\end{align*}
Similarly, we also obtain
\begin{align*}
\mathrm{I}_-\ls \frac{(Rt)^{n/(2m)-2(l+1)}}{t^{n/(2m)}} \exp\lf\{-\dz |x-y| R^{\frac{1}{2m}}\r\}.
\end{align*}

For $\mathrm{I}_0$, we have
\begin{align*}
\mathrm{I}_0\ls \frac{1}{t^{2(l+1)-1}} \dint_{|\tz|\ge \mu}\exp\lf\{-tR\cos \tz\r\}\,d\tz\,
R^{\frac{n}{2m}-2(l+1)+1}\exp\lf\{-\dz |x-y|R^{\frac{1}{2m}}\r\}.
\end{align*}
Using the fact that $\int_0^{2\pi}e^{-s\cos \tz}\,d\tz\ls \frac{e^s}{s^{1/2}}$ for any $s\in (0,\fz)$
(see also \cite[(2.30)]{AuMcTc98}),
we know that
\begin{align*}
\mathrm{I}_0\ls \frac{1}{t^{n/(2m)}} (tR)^{n/(2m)-2(l+1)+1/2} \exp\lf\{tR-\dz |x-y|R^{\frac{1}{2m}}\r\}.
\end{align*}
Now taking $R:=\epsilon \frac{|x-y|^{2m/(2m-1)}}{|t|^{2m/(2m-1)}}$,
we then have $tR=\epsilon\frac{|x-y|^{2m/(2m-1)}}{|t|^{1/(2m-1)}}$ and hence
\begin{align*}
\dz |x-y|R^{\frac{1}{2m}}=\dz |x-y|\epsilon^{\frac{1}{2m}}\frac{|x-y|^{1/(2m-1)}}
{t^{1/(2m-1)}}=\dz \epsilon^{\frac{1}{2m}} \frac{|x-y|^{2m/(2m-1)}}
{t^{1/(2m-1)}}.
\end{align*}
Taking  $\epsilon\in (0,\fz)$ small enough such that $\dz
\epsilon^{1/(2m)}>\epsilon$ and using the assumption
$2(l+1)>\frac{n}{2m}$, we conclude that
\begin{align*}
\mathrm{I}_0&\ls \frac{1}{t^{n/(2m)}} \lf[\frac{|x-y|^{2m/(2m-1)}}
{t^{1/(2m-1)}}\r]^{n/(2m)-2(l+1)+1/2} \exp\lf\{-\lf(\dz\epsilon^{1/(2m)}-\epsilon\r)
\frac{|x-y|^{2m/(2m-1)}}{t^{1/(2m-1)}}\r\}\\
&\ls \frac{1}{t^{n/(2m)}} \exp\lf\{-c_3
\frac{|x-y|^{2m/(2m-1)}}{t^{1/(2m-1)}}\r\}
\end{align*}
by choosing a suitable constant $c_3\in (0,\fz)$.
Combining the estimates of $\mathrm{I}_0$, $\mathrm{I}_+$, and $\mathrm{I}_-$,
we know that \eqref{eqn-mGUB} holds true in the case $q\in (1,2]$.
The case that $q\in [2,\fz)$ is similar and we omit the details, which
shows (a).

We are now in a position to  prove (b). We use the same notation as
in the proof of (a).
Assume first that $n\ge 4m$. In this case,
using Lemma \ref{lem-GN2}(ii) and Proposition
\ref{prop-perburbationresol}, we know that, for any $f\in L^2(\rn)$ and
$x$, $h\in \rn$,
\begin{align*}
\lf|\lf(\lz-\cl\r)^{-2(l+1)}f(x+h)-\lf(\lz-\cl\r)^{-2(l+1)}f(x)\r|&
\ls
\lf\|\Delta^{m}\lf(\lz-\cl\r)^{-2(l+1)}f\r\|_{L^{q_l}(\rn)}|h|^{\gz}\\
\notag &\ls
\lf\|\lf(\lz-\cl\r)^{-2(l+1)+1}f\r\|_{L^{q_l}(\rn)}|h|^{\gz},
\end{align*}
where $\gz:=2m-\frac{n}{q_l}\in (0,1)$ by  \eqref{eqn-sqj}.
Following the argument used in
the proof of \eqref{eqn-RA}, we find that, for any $g\in L^1(\rn)$, and $x$, $h\in \rn$,
\begin{align*}
\lf|\lf(\lz-\cl\r)^{-2(l+1)}g(x+h)-\lf(\lz-\cl\r)^{-2(l+1)}g(x)\r|\ls |\lz|^{1-a_l}|\lz|^{\frac{n}{2m}-2(l+1)}\|g\|_{L^1(\rn)}
|h|^\gz,
\end{align*}
which implies that, for any $x$, $y\in \rn$ and any $\vec 0_n\ne h\in\rn$,
\begin{align}\label{eqn-holde2}
\lf|K_{2(l+1)}(x+h,y)-K_{2(l+1)}(x,y)\r|\ls |\lz|^{1-a_l}|\lz|^{\frac{n}{2m}-2(l+1)}|h|^\gz,
\end{align}
where $K_{2(l+1)}(x,y)$ denotes the integral kernel of
$(\lz-\cl)^{-2(l+1)}$.
Thus, using \eqref{eqn-functionalcal}, we
conclude that, for any $x$, $y$, $h\in \rn$ satisfying $0<|h|<
t^{1/2m}$,
\begin{align*}
\lf|p_t(x+h,y)-p_t(x,y)\r|\ls \frac{1}{t^{2(l+1)-1}}\dint_\Gamma \lf|e^{-t\lz}\r| \lf|K_{2(l+1)}(x+h,y)-K_{2(l+1)}(x,y)\r|\,|d\lz|
\end{align*}
with $\Gamma$ being the curve as in \eqref{eqn-IC}.

We now further assume that  $|x-y|<t^{1/(2m)}$. In this case, for
the integral over $\Gamma_+$ (the corresponding term is denoted by
$\wz{\mathrm{I}}_+$), by  $2(l+1)>n/(2m)$,
$ R=\epsilon\frac{|x-y|^{2m/(2m-1)}}{t^{2m/(2m-1)}}$, and
$|x-y|<t^{1/(2m)}$, we obtain
\begin{align*}
\wz{\mathrm{I}}_+&\ls \frac{|h|^\gz}{t^{2(l+1)-1}}\dint_{R}^\fz e^{-tr\cos\mu} r^{1-a_l+\frac{n}{2m}-2(l+1)}\,dr\\
&\ls \frac{|h|^\gz}{t^{n/(2m)}}\lf(Rt\r)^{n/(2m)-2(l+1)}\frac{1}{t^{1-a_l}}
\dint_{0}^\fz e^{-tr\cos\mu} (rt)^{1-a_l}\,d(rt)
\ls \frac{|h|^\gz}{t^{1-a_l}}\frac{1}{t^{n/(2m)}},
\end{align*}
which implies \eqref{eqn-mHE} by taking $\gz:=2m-\frac{n}{q_l}$ and using
$1-a_l=1-\frac{n}{2m}\cdot\frac{1}{q_l}
=\frac{1}{2m}(2m-\frac{n}{q_l})$. The estimates for the integrals over $\Gamma_-$ and $\Gamma_0$ are similar,
and we omit the details.

If $|x-y|\ge t^{1/2m}$, then, by \eqref{eqn-resolventper}, we first have, for any $h\in\rn$ satisfying
$|h|< t^{2m}$,
\begin{align*}
&\lf|K_{2(l+1)}(x+h,y)-K_{2(l+1)}(x,y)\r|\\
&\hs=\lf|e^{(x+h)\eta}e^{-y\eta}K_{2(l+1),\eta}(x+h,y)-e^{\eta x}e^{-y\eta}K_{2(l+1),\eta}(x,y)\r|\\
&\hs\le e^{(x-y)\eta}\lf|K_{2(l+1),\eta}(x+h,y)-K_{2(l+1),\eta}(x,y)\r|
+\lf|e^{h\eta}-1\r|e^{(x-y)\eta}\lf|K_{2(l+1),\eta}(x+h,y)\r|\\
&\hs=: H_h(x,y)+J_h(x,y).
\end{align*}

From this and \eqref{eqn-functionalcal}, we then deduce
\begin{align*}
\lf|p_t(x+h,y)-p_t(x,y)\r|&\ls \frac{1}{t^{2(l+1)-1}}\dint_\Gamma
\lf|e^{-t\lz}\r| \lf[H_h(x,y)+J_h(x,y)\r]\,|d\lz|
:=\mathrm{H}+\mathrm{J}
\end{align*}
with $\Gamma$ as in \eqref{eqn-IC}.

To estimate $\mathrm{H}$, similarly to \eqref{eqn-holde2}, we have,
for any $t\in (0,\fz)$, and any $x$, $y$, $h\in\rn$ satisfying $|x-y|\ge t^{1/2m}$ and $|h|< t^{2m}$,
\begin{align*}
\lf|K_{2(l+1),\eta}(x+h,y)-K_{2(l+1),\eta}(x,y)\r|\ls |\lz|^{1-a_l}|\lz|^{\frac{n}{2m}-2(l+1)}|h|^\gz
\end{align*}
with $\gz=2m-n/q_l$.
Thus, by letting $  \eta:=-\mathrm{sgn}\,(x-y)\dz|\lz|^{1/(2m)}$
and $  R:=\epsilon\frac{|x-y|^{2m/(2m-1)}}
{t^{2m/(2m-1)}}$, we know that
the integral of $\mathrm{H}_h$ over $\Gamma_+$ satisfies
\begin{align*}
&\frac{1}{t^{2(l+1)-1}}\dint_{\Gamma_+}
\lf|e^{-t\lz}\r|H_h(x,y)\,|d\lz|\\
&\hs\ls \frac{|h|^\gz}{t^{2(l+1)-1}}\dint_{R}^\fz
\exp\lf\{-\dz|x-y| r^{1/(2m)}\r\}e^{-tr\cos\mu} r^{1-a_l+\frac{n}{2m}-2(l+1)}\,dr\\
&\hs\ls \frac{|h|^\gz}{t^{1-a_l}}\frac{1}{t^{n/(2m)}}\exp\lf\{-\frac{c_4|x-y|^{2m/(2m-1)}}{t^{1/(2m-1)}}\r\},
\end{align*}
which implies \eqref{eqn-mHE} with $\gz:=2m-\frac{n}{q_l}$.
The estimates for the integrals of $\mathrm{H}_h$ over $\Gamma_-$ and $\Gamma_0$ are similar,
and we omit the details.

On the other hand, for the integral of $\mathrm{J}_h$ over $\Gamma_+$,
using the Taylor theorem, \eqref{eqn-holder3}, and the assumptions that $|x-y|\ge t^{1/2m}$ and
$|h|<t^{1/(2m)}$, we know that
there exists a positive constant $\tz\in (0,1)$ such that
\begin{align*}
&\frac{1}{t^{2(l+1)-1}}\dint_{\Gamma_+}
\lf|e^{-t\lz}\r|J_h(x,y)\,|d\lz|\\
&\hs\ls \frac{1}{t^{2(l+1)-1}}\dint_{R}^\fz
\exp\lf\{-\dz|x-y| r^{1/(2m)}\r\}e^{-tr\cos\mu} r^{\frac{n}{2m}-2(l+1)}
e^{\tz \dz t^{1/(2m)}r^{1/(2m)}}|h\eta|\,dr\\
&\hs\ls \frac{|h|}{t^{2(l+1)-1}}\dint_{R}^\fz
\exp\lf\{-\dz(1-\tz)|x-y| r^{1/(2m)}\r\}e^{-tr\cos\mu} r^{\frac{n+1}{2m}-2(l+1)}\,dr\\
&\hs\ls \frac{|h|}{t^{1/(2m)}}\frac{1}{t^{n/(2m)}}\exp\lf\{-\frac{c_4|x-y|^{2m/(2m-1)}}{t^{1/(2m-1)}}\r\},
\end{align*}
which implies that \eqref{eqn-mHE} holds true.
The estimates for the integrals of $\mathrm{J}_h$ over $\Gamma_-$ and $\Gamma_0$ are similar,
and we omit the details.

If $2m\le n<4m$, then, by the argument used in \eqref{eqn-sqj}, we have
$1=q_2'<q_1'<q_0=2<q_1<q_2=\fz$.  Without loss of generality, we may
take $q_1'\in (1,2)\cap (\frac{n}{2m},\frac{n}{2m-1})$. Since $(2m-1)q_1'<n<2mq_1'$,
we deduce from Lemma \ref{lem-GN2}(ii) and Proposition
\ref{prop-perburbationresol} that, for any $f\in L^1(\rn)$ and
$x$, $h\in \rn$,
\begin{align*}
\lf|\lf(\lz-\cl\r)^{-2}f(x+h)-\lf(\lz-\cl\r)^{-2}f(x)\r|&
\ls
\lf\|\Delta^{m}\lf(\lz-\cl\r)^{-2}f\r\|_{L^{q_1'}(\rn)}|h|^{\gz}\\
\notag &\ls
\lf\|\lf(\lz-\cl\r)^{-1}f\r\|_{L^{q_1'}(\rn)}|h|^{\gz},
\end{align*}
where $\gz:=2m-\frac{n}{q_1'}\in (0,1)$. Moreover,
using \eqref{eqn-lfz-lql} and a dual argument, we find that
\begin{align*}
\lf\|\lf(\lz-\cl\r)^{-1}f\r\|_{L^{q_1'}(\rn)}\ls |\lz|^{a_1-1}\|f\|_{L^1(\rn)}
\end{align*}
with $a_1:=\frac{n}{2m q_1}$. This, together with $\frac{\gz}{2m}=1-\frac{n}{2m q'_1}$,
implies that, for any $f\in L^1(\rn)$ and $x$, $h\in \rn$,
\begin{align*}
\lf|\lf(\lz-\cl\r)^{-2}f(x+h)-\lf(\lz-\cl\r)^{-2}f(x)\r|\ls |h|^\gz |\lz|^{\gz/(2m)+n/(2m)-2}\|f\|_{L^1(\rn)}
\end{align*}
and the integral kernel $K_2$ of $(\lz-\cl)^{-2}$ satisfies that, for any
$x$, $y$, $h\in\rn$ satisfying $\vec 0_n\ne h\in\rn$,
\begin{align*}
\lf|K_{2}(x+h,y)-K_{2}(x,y)\r|\ls |\lz|^{\gz/(2m)+n/(2m)-2}|h|^\gz.
\end{align*}
Since this estimate is similar to \eqref{eqn-holde2} [with $2(l+1)=2$ therein],
the remainder of the proof is similar to the case $n\ge 4m$, we omit the details. This
shows (b) holds true and hence
finishes the proof of Theorem \ref{thm-main1}.
\end{proof}

Similarly to Corollary \ref{cor-localDG}, we have the following local
Gaussian estimate for the heat kernel of $e^{-t\cl}$.

\begin{corollary}\label{cor-heatkernel}
Let $m\in\nn$, $\cl=P(D)+V$ be the $2m$-order Schr\"odinger operator on
$\rn$ as in \eqref{eqn-def-HOSO}, and $V$ a measurable function on $\rn$.
If \eqref{ap1} and one of \eqref{ap2} through \eqref{ap5} hold true
for any $q\in (1,2]$ or $[2,\fz)$ and
$\sup_{|\lz|\in(w/2,\fz)}M_{|\lz|}(V)<1$ for some  $w\in (0,\fz)$
and $M_{|\lz|}(V)$ as in \eqref{HP},
then the operator $\cl$ possesses a heat kernel $p_t$ on
$(0,\fz)\times \rn\times \rn$ that satisfies the following assertions:
\begin{itemize}
\item[{\rm(i)}]
there exist positive constants
$C$ and $c_3$ such that, for any $t\in (0,\fz)$ and $(x,\,y)\in \rn\times \rn$,
\begin{align}\label{eqn-lGUB}
\lf|p_t(x,y)\r|\le \frac{C}{t^{n/(2m)}} \exp\lf\{wt-c_3\frac{|x-y|^{2m/(2m-1)}}
{t^{1/(2m-1)}}\r\};
\end{align}

\item[{\rm(ii)}] if, in addition, $n\ge 2m$,
there exists a $\gz\in(0,1)$ and positive constants $C$ and $c_4$ such that,
for any $t\in (0,\fz)$, $(x,\,y)\in\rn\times\rn$, and $h\in\rn$ satisfying $|h|< t^{1/2m}$,
\begin{align}\label{eqn-lHE}
\lf|p_t(x+h,y)-p_t(x,y)\r| \le
\frac{C}{t^{n/(2m)}}\exp\lf\{wt-c_4\frac{|x-y|^{2m/(2m-1)}}{t^{1/(2m-1)}}\r\}
\lf[\frac{|h|}{t^{1/(2m)}}\r]^\gz.
\end{align}
\end{itemize}
\end{corollary}

\begin{remark}\label{rem-pmGE}
\begin{enumerate}
\item[(i)] For any given $m\in \nn$ and $a\in (-2m,0)$, let $V(x)=\pm|x|^a$ for any $x\in\rn\setminus \{\vec 0_n\}$.
By taking the parameters in \eqref{ap1} and
\eqref{ap5} with $q=p=2$, $s\in (0,2m]$, $\az\in (0,n)$, and $\az\le 2s$,
as in Remark \ref{rem-pmDG}(i), if further assume $a\in (-{\min\{2s,n\}}/{2},0)$,
then
\begin{align*}
M_{|\lz|}(V)=C_2C|\lz|^{{S}_4}M_{\az,2,\fz,1/|\lz|}(V)\sim
|\lz|^{\frac{\az}2-2m-\lf(\frac{\az}{2}+a\r)}
\sim|\lz|^{-2m-a}
\end{align*}
and $\sup_{|\lz|>\epsilon_0}M_{|\lz|}(V)<1$
for some $\epsilon_0>0$ with the positive equivalence constants independent of
$\lz$.

On the other hand, in the case $q=1$, by Remark \ref{ex-GSC}, we know that
$V$ is in the Kato class $K_{2m}(\rn)$. Using \cite[Proposition 2.2]{DeDiYa14},
we find that the operator $T_{2m,\dz}$ defined as in \eqref{def-Tlz} is bounded on $L^1(\rn)$ and
\begin{align*}
\lim_{\dz\to \fz}\lf\|T_{2m,\dz}\r\|_{L^1(\rn)\to L^1(\rn)}=0.
\end{align*}
Following the arguments used in the proof of Proposition \ref{prop-spr} and using an
interpolation of linear operators
between $L^1(\rn)$ and $L^2(\rn)$, we conclude that, for any given $q\in (1,2)$,
\eqref{eqn-assumptionforcleta2x} holds true when $|\lz|$ large enough.
Thus, by Corollary \ref{cor-heatkernel}, we know that
the semigroup generated by $-(P(D)+V)$ has a heat kernel satisfying the local Gaussian estimates
\eqref{eqn-lGUB} and \eqref{eqn-lHE}.

\item[(ii)] For any given $m\in\nn$ and $a\in (-\fz,0)$, let
$V(x)=\pm(1+|x|)^a$ for any $x\in\rn$.
It is easy to see that $\|V\|_{L^t(\rn)}\ls 1$ for any given
$t\in (-n/a,\fz)$ with the implicit positive constant depending only on $a$,
$t$, and $n$. Thus, by taking the parameters in \eqref{ap1} and
\eqref{ap3} with $q=p\in (1,\fz)$, $s:=2m$, and $t\in[\frac{n}{2m},\fz)$, we obtain
\begin{align*}
S_2=-2m\lf(1-\frac{n}{2mt}\r)\le 0.
\end{align*}
This indicates that the semigroup generated
by $-(P(D)+V)$ has a heat kernel satisfying the local Gaussian estimates \eqref{eqn-lGUB} and
\eqref{eqn-lHE} when $n\ge 2m$.

Moreover, if $a\in (-\fz, -2m)$ and $V(x)=\pm c(1+|x|)^a$ for any $x\in\rn$
with $c\in (0,\fz)$,
then $S_2=0$ by taking $t:=n/(2m)$. Thus, by taking $c$ sufficiently small,
we have
\begin{align*}
\dsup_{|\lz|\in (0,\fz)}M_{|\lz|}(V)<1.
\end{align*}
This, combined with Theorem \ref{thm-main1}, implies that the semigroup generated
by $-(P(D)+V)$ has a heat kernel satisfying the
Gaussian estimates \eqref{eqn-mGUB} and \eqref{eqn-mHE}
when $n\ge 2m$.
\end{enumerate}
\end{remark}

\noindent\textbf{Acknowledgements.}\quad The authors would like to
thank Professor Jacek Dziuba\'{n}ski for a motivating discussion on a related
subject of this article which led us to consider the problem of this article.
The authors would also like to thank Professor Xiaohua Yao for some helpful discussions
on the subject of this article.

\bigskip

{\small\noindent Jun Cao

\smallskip

\noindent
Department of Applied Mathematics, Zhejiang University of Technology,
Hangzhou 310023, People's Republic of China

\smallskip

\noindent{\it E-mail:} \texttt{caojun1860@zjut.edu.cn}

\bigskip

\noindent {Yu Liu}

\smallskip

\noindent School of Mathematics and Physics, University of Science and Technology Beijing,
Beijing 100083, People's Republic of China

\smallskip

\noindent{\it E-mail:} \texttt{liuyu75@pku.org.cn}

\bigskip

\noindent {Dachun  Yang} (Corresponding author)

\smallskip

\noindent Laboratory of Mathematics and Complex Systems (Ministry of Education of China),
School of Mathematical Sciences, Beijing Normal University, Beijing 100875, People's Republic of China

\smallskip

\noindent{\it E-mail:} \texttt{dcyang@bnu.edu.cn}

\bigskip

\noindent {Chao Zhang}

\smallskip

\noindent School of Statistics and Mathematics,
Zhejiang Gongshang University, Hangzhou 310018,
People's Republic of China

\smallskip

\noindent{\it E-mail:} \texttt{zaoyangzhangchao@163.com}}

\end{document}